\documentclass[12pt,a4paper]{amsart}
\usepackage{mathtools}
\usepackage{xcolor}
\usepackage{amsmath,amstext,amssymb,amsthm,amsfonts}
\usepackage[all]{xy} \xyoption{poly}

\newcommand{\nc}{\newcommand}
\nc{\one}{\mbox{\bf 1}}
\nc{\invtensor}{\underset{\leftarrow}{\otimes}}
\nc{\const}{\operatorname{const}}

\usepackage{tikz}
\usepackage{pgffor}
\usepackage[framemethod=tikz]{mdframed}

\usepackage{longtable}
\usepackage{booktabs}
\usepackage{array}
\newcolumntype{R}{>{\raggedleft\arraybackslash}p{6cm}}

\nc{\ad}{\operatorname{ad}}

\nc{\tr}{\operatorname{tr}}
\nc{\Arc}{\operatorname{Arc}}
\nc{\arc}{\operatorname{arc}}
\nc{\Ch}{\operatorname{Ch}}
\nc{\diag}{\operatorname{diag}}

\nc{\tp}{\operatorname{top}}
\nc{\rank}{\operatorname{rank}}
\nc{\corank}{\operatorname{corank}}
\nc{\codim}{\operatorname{codim}}
\nc{\sdim}{\operatorname{sdim}}
\nc{\mult}{\operatorname{mult}}
\nc{\ds}{\operatorname{ds}}
\nc{\tail}{\operatorname{tail}}
\nc{\howl}{\operatorname{howl}}
\nc{\spn}{\operatorname{span}}
\nc{\Sym}{\operatorname{Sym}}
\nc{\sym}{\operatorname{sym}}
\nc{\id}{\operatorname{id}}
\nc{\Id}{\operatorname{Id}}
\nc{\Ree}{\operatorname{Re}}
\nc{\hi}{\operatorname{hi}}
\nc{\htt}{\operatorname{ht}}
\nc{\at}{\operatorname{at}}
\nc{\str}{\operatorname{str}}
\nc{\Iso}{\operatorname{Iso}}
\nc{\Ker}{\operatorname{Ker}}
\nc{\rker}{\operatorname{rKer}}
\nc{\im}{\operatorname{Im}}
\nc{\osp}{\mathfrak{osp}}
\nc{\sgn}{\operatorname{sgn}}

\nc{\Mod}{\operatorname{Mod}}
\nc{\DS}{\operatorname{DS}}
\nc{\Soc}{\operatorname{Soc}}
\nc{\Inj}{\operatorname{Inj}}
\nc{\Hom}{\operatorname{Hom}}
\nc{\End}{\operatorname{End}}
\nc{\supp}{\operatorname{supp}}
\nc{\Card}{\operatorname{Card}}
\nc{\Ann}{\operatorname{Ann}}
\nc{\Ind}{\operatorname{Ind}}
\nc{\Coind}{\operatorname{Coind}}
\nc{\wt}{\operatorname{wt}}

\nc{\ch}{\operatorname{ch}}
\nc{\sch}{\operatorname{sch}}
\nc{\mdim}{\operatorname{mdim}}
\nc{\Stab}{\operatorname{Stab}}
\nc{\Sch}{\operatorname{Sch}}
\nc{\Irr}{\operatorname{Irr}}
\nc{\Spec}{\operatorname{Spec}}
\nc{\Res}{\operatorname{Res}}
\nc{\Aut}{\operatorname{Aut}}
\nc{\Ext}{\operatorname{Ext}}
\nc{\Prec}{\operatorname{Prec}}
\nc{\Fract}{\operatorname{Fract}}
\nc{\gr}{\operatorname{gr}}
\nc{\deff}{\operatorname{def}}
\nc{\core}{\operatorname{core}}
\nc{\HC}{\operatorname{HC}}
\nc{\dpth}{\operatorname{dpth}}
\nc{\jJ}{\operatorname{J}}
\nc{\KW}{\operatorname{KW}}
\nc{\red}{\operatorname{red}}
\nc{\pari}{\operatorname{dex}}
\nc{\hwt}{\operatorname{hwt}}
\nc{\wdchi}{\widetilde{\chi}}
\nc{\wdH}{\widetilde{H}}
\nc{\wdN}{\widetilde{N}}
\nc{\wdM}{\widetilde{M}}
\nc{\wdO}{\widetilde{O}}
\nc{\wdR}{\widetilde{R}}

\nc{\wdV}{\widetilde{V}}

\nc{\wdC}{\widetilde{C}}

\nc{\Obj}{\operatorname{Obj}}
\nc{\Dglie}{\operatorname{{\mathcal D}glie}}
\nc{\Fin}{\operatorname{\mathcal{F}in}}
\nc{\pr}{\operatorname{pr}}
\nc{\Adm}{\operatorname{\mathcal{A}dm}}
\nc{\Sg}{{\cS(\fg)}}
\nc{\Shg}{{\cS(\fhg)}}
\nc{\Ug}{{\cU(\fg)}}
\nc{\Uhg}{{\cU(\fhg)}}
\nc{\Sh}{{\cS(\fh)}}
\nc{\Uh}{{\cU(\fh)}}
\nc{\Uhh}{{\cU(\fhh)}}
\nc{\Zg}{{{\mathcal{Z}}(\fg)}}

\nc{\Vir}{{\mathcal{V}ir}}
\nc{\NS}{{\mathcal{N}S}}

\nc{\tZg}{{\widetilde{\mathcal Z}({\mathfrak g})}}
\nc{\Zk}{{\mathcal Z}({\mathfrak k})}

\newcommand{\cE}{\mathcal{E}}

\newcommand{\cP}{\mathcal{P}}

\nc{\Up}{{\mathcal U}({\mathfrak p})}
\nc{\Ah}{{\mathcal A}({\mathfrak h})}
\nc{\Ag}{{\mathcal A}({\mathfrak g})}
\nc{\Ap}{{\mathcal A}({\mathfrak p})}
\nc{\Zp}{{\mathcal Z}({\mathfrak p})}
\nc{\cR}{\mathcal R}
\nc{\cS}{\mathcal S}
\nc{\cT}{\mathcal{T}}
\nc{\cC}{\mathcal C}
\nc{\cA}{\mathcal A}
\nc{\cU}{\mathcal U}
\nc{\cZ}{\mathcal Z}
\nc{\cM}{\mathcal M}
\nc{\cL}{\mathcal L}
\nc{\cF}{\mathcal F}
\nc{\fg}{\mathfrak g}
\nc{\cB}{\mathcal{B}}

\nc{\F}{\tilde{\mathcal F}}

\nc{\fo}{\mathfrak o}

\nc{\CO}{\mathcal O}
\nc{\CR}{\mathcal R}

\nc{\cK}{\mathcal{K}}
\nc{\cW}{\mathcal{W}}
\nc{\bM}{\mathbf{M}}
\nc{\bL}{\mathbf{L}}
\nc{\bN}{\mathbf{N}}

\nc{\zq}{\mathpzc q}

\nc{\fl}{\mathfrak l}
\nc{\fn}{\mathfrak n}
\nc{\fm}{\mathfrak m}
\nc{\fp}{\mathfrak p}
\nc{\fh}{\mathfrak h}
\nc{\ft}{\mathfrak t}
\nc{\fk}{\mathfrak k}
\nc{\fb}{\mathfrak b}
\nc{\fs}{\mathfrak s}
\nc{\fB}{\mathfrak B}

\nc{\vareps}{\varepsilon}
\nc{\varesp}{\varepsilon}
\nc{\veps}{\varepsilon}

\nc{\fsl}{\mathfrak{sl}}
\nc{\fgl}{\mathfrak{gl}}
\nc{\fso}{\mathfrak{so}}
\nc{\fosp}{\mathfrak{osp}}
\nc{\fsp}{\mathfrak{sp}}
\nc{\fq}{\mathfrak q}
\nc{\fsq}{\mathfrak{sq}}
\nc{\fpsl}{\mathfrak{psl}}

\nc{\fhg}{\hat{\fg}}
\nc{\fhn}{\hat{\fn}}
\nc{\fhh}{\hat{\fh}}
\nc{\fhb}{\hat{\fb}}
\nc{\hrho}{\hat{\rho}}

\nc{\hsl}{\hat{\fsl}}
\nc{\fpo}{\mathfrak{po}}
\nc{\dirlim}{\underset{\rightarrow}{\lim}\,}
\nc{\nen}{\newenvironment}
\nc{\ol}{\overline}
\nc{\ul}{\underline}
\nc{\ra}{\rightarrow}
\nc{\lra}{\longrightarrow}
\nc{\Lra}{\Longrightarrow}
\nc{\bo}{\bar{1}}
\nc{\Lla}{\Longleftarrow}

\nc{\Llra}{\Longleftrightarrow}

\nc{\thla}{\twoheadleftarrow}

\nc{\lang}{(}
\nc{\rang}{)}

\nc{\hra}{\hookrightarrow}

\nc{\iso}{\overset{\sim}{\lra}}

\nc{\ssubset}{\underset{\not=}{\subset}}

\nc{\vac}{|0\rangle}

\nc{\simka}{{\ \scriptscriptstyle _{{\sim}}^\text{\tiny{k}}\ }}

\newcommand{\ssum}{\mathit{sum}}

\nc{\Thm}[1]{Theorem~\ref{#1}}
\nc{\Prop}[1]{Proposition~\ref{#1}}
\nc{\Lem}[1]{Lemma~\ref{#1}}
\nc{\Cor}[1]{Corollary~\ref{#1}}
\nc{\Conj}[1]{Conjecture~\ref{#1}}
\nc{\Claim}[1]{Claim~\ref{#1}}
\nc{\Defn}[1]{Definition~\ref{#1}}
\nc{\Exa}[1]{Example~\ref{#1}}
\nc{\Rem}[1]{Remark~\ref{#1}}
\nc{\Note}[1]{Note~\ref{#1}}
\nc{\Quest}[1]{Question~\ref{#1}}
\nc{\Hyp}[1]{Hypoth\`ese~\ref{#1}}
\nen{thm}[1]{\label{#1}{\bf Theorem.\ } \em}{}
\nen{prop}[1]{\label{#1}{\bf Proposition.\ } \em}{}
\nen{lem}[1]{\label{#1}{\bf Lemma.\ } \em}{}
\nen{cor}[1]{\label{#1}{\bf Corollary.\ } \em}{}
\nen{conj}[1]{\label{#1}{\bf Conjecture.\ } \em}{}

\nen{claim}[1]{\label{#1}{\bf Claim.\ } \em}{}

\nen{defn}[1]{\label{#1}{\bf Definition.\ } }{}
\nen{exa}[1]{\label{#1}{\bf Example.\ } }{}
\nen{rem}[1]{\label{#1}{\em Remark.\ } }{}
\nen{note}[1]{\label{#1}{\em Note.\ } }{}
\nen{exer}[1]{\label{#1}{\em Exercise.\ } }{}
\nen{sket}[1]{\label{#1}{\em Sketch of proof.\ } }{}
\nen{quest}[1]{\label{#1}{\bf Question.\ } \em}{}

\nen{hyp}[1]{\label{#1}{\bf Hypoth\`ese.\ } \em}{}
\begin{document}
\setcounter{section}{0}
\setcounter{tocdepth}{1}
\title[Character formulas]{Gruson-Serganova character formulas and the Duflo-Serganova cohomology functor}
\author{{\rm M. Gorelik, Th. Heidersdorf}}

\address{T. H.: Mathematisches Institut Universit\"at Bonn, Germany}
\email{heidersdorf.thorsten@gmail.com}
\address{M. G..: Weizmann Institute of Science, Rehovot, Israel}
\email{maria.gorelik@gmail.com}

\date{}

\begin{abstract} We establish an explicit formula for the character of an irreducible finite-dimensional representation of $\mathfrak{gl}(m|n)$. The formula is a finite sum with integer coefficients in terms of a basis $\mathcal{E}_{\mu}$ (Euler characters) of the character ring. We prove a simple formula for the behaviour of the ``superversion'' of $\mathcal{E}_{\mu}$ in the $\mathfrak{gl}(m|n)$ and $\mathfrak{osp}(m|2n)$-case under the map $ds$ on the supercharacter ring induced by the Duflo-Serganova cohomology functor $DS$. As an application we get combinatorial formulas for superdimensions, dimensions and $\mathfrak{g}_0$-decompositions  for $\mathfrak{gl}(m|n)$ and $\mathfrak{osp}(m|2n)$.
\end{abstract}

\subjclass[2010]{17B10, 17B20, 17B55, 18D10.}

\medskip

\keywords{Representations of supergroups, Character formula, Duflo-Serganova functor, Superdimensions, Tame modules, Orthosymplectic Lie superalgebra, General linear Lie superalgebra}

\maketitle

%\tableofcontents

\section{Introduction}
Let $\fg$ be  a finite-dimensional Kac-Moody superalgebra. Denote by
$W$ the Weyl group of $\fg$.

\subsection{A brief history of character formulas}\label{intro1}
Let $L(\lambda)$ be a simple finite-dimensional $\fg$-module. In 1977 V.~Kac~\cite{Kac-repr}
showed that
the Weyl character formula  
$$Re^{\rho}\ch L(\lambda)= \sum_{w\in W}\sgn(w) w(e^{\lambda+\rho}),$$
 where 
 $R$ is the Weyl denominator and $\rho$ is the Weyl vector holds if 
 $L(\lambda)$ is typical.
In 1980 I.~Bernstein and D.~Leites~\cite{BL}
established for   $\fg=\fgl(1|n)$ the character formula 
\begin{equation}\label{BL}
Re^{\rho}\ch L(\lambda)=\sum_{w\in W}\sgn(w) w\bigl(\frac{e^{\lambda}}{(1+e^{-\beta})}\bigr) 
\end{equation} 
where 
$\beta \in\Delta^+_1$ satisfies $(\beta|\lambda)=0$. This formula was extended to 
the $\osp(2|2n)$-case  in~\cite{vdJ} and to $\fgl(m|n)$-modules of atypicality one
in~\cite{vdJHKTM}. In 1998 J.~Germoni produced similar character formulas for 
the cases $\osp(3|2)$ and 
$D(2|1;a)$; except for the case of the standard $\osp(3|2)$-module  Germoni's formula is very similar to~(\ref{BL}): the only difference is 
the factor $1/2$ appearing in the right-hand side in certain cases
(other formulas were obtained earlier by van der Jeugt in~\cite{vdJ32},\cite{vdJD21}).
In 1990 J. van der Jeugt, J. Hughes, R.C. King and 
J. Thierry-Mieg~\cite{vdJHKTM} suggested to write  the character formula in the general $\mathfrak{gl}(m|n)$-case as a sum of terms
$\sum_{w\in W}\sgn(w) w\bigl(\frac{e^{\lambda}}{\prod_{\beta\in U}(1+e^{-\beta})}\bigr)$
for some $U\subset\Delta_1$ satisfying $(U|\lambda)=0$.

In 1994 V.~Kac and M.~Wakimoto~\cite{KW} conjectured that
\begin{equation}\label{KWformula}
Re^{\rho}\ch L(\lambda)=j^{-1}\sum_{w\in W}\sgn(w) w\bigl(\frac{e^{\lambda+\rho}}{\prod_{\beta\in S}(1+e^{-\beta})}\bigr)\end{equation}
if the following conditions (which we will call KW-conditions) hold:
$\lambda$ is the highest weight of $L$ with respect to a base $\Sigma$
containing $S$, $(S|\lambda+\rho)=(S|S)=0$ and the cardinality of $S$
is equal to the atypicality of $L$. This conjecture was established
in~\cite{CHR},\cite{CK} and~\cite{GK}. For each $S\subset\Delta$ 
satisfying $(S|\lambda)=(S|S)=0$ we set
$$\KW(\lambda,S):=\sum_{w\in W}\sgn(w) w\bigl(\frac{e^{\lambda}}{\displaystyle\prod_{\beta\in S}(1+e^{-\beta})}\bigr).$$
(We call the above terms ``Kac-Wakimoto terms'' 
 since the condition $(S|S)=0$
is crucial for our argument).

Notice that Germoni's formulas
demonstrate that the KW-conditions are not necessary for~(\ref{KWformula})  to be valid:
the only atypical $\osp(3|2)$-module satisfying the KW-condition
is trivial whereas ~(\ref{BL}) holds for each $L\not\cong V_{st}$.

The first general formula for $\ch L$ was discovered by V.~Seganova~\cite{S-irr} in the $\fgl(m|n)$-case  by expressing the character as an infinite sum over characters of Kac modules. This algorithmic solution was enhanced by J.~Brundan \cite{Brundan-gl} who showed that the values of the coefficients  in the infinite sum can be computed in terms of weight diagrams. This description was then used Y.~Su and R.~Zhang \cite{SuZh1} to establish a finite character formula for $\mathfrak{gl}(m|n)$. For the $\osp$-case a finite character formula was produced by C.~Gruson and V.~Serganova in~\cite{GS}. For the exceptional Lie superalgebras
the character formulas were proven by J.~Germoni (\cite{Ger}) and L.~Martirosyan (\cite{Lilit}). For $\fq_n$ an implicit finite character formula was given by I.~Penkov and V.~Serganova in~\cite{PS2}; 
 Y.~Su and R.~B.~Zhang~\cite{SZq} wrote this formula explicitly 
using~\cite{B-q};
for the $\fp_n$-case an infinite Serganova-type character formula was recently proven by B.-H.~Hwang and J.-H.~Kwon~\cite{HK}.

The Kac-Wakimoto character formula suggests the following two refinements of 
van der Jeugt-Hughes-King-Thierry-Mieg's proposal: to present 
$Re^{\rho} \ch L$ as  a linear combinations of $\KW(\nu, S(\nu))$
where

I.  $S_{\nu}$ is maximal, i.e. $|S_{\nu}|$ equals to the atypicality of $L(\nu)$
or
  
II. $S_{\nu}$ can be embedded into a certain base $\Sigma_L$.

The Kac-Wakimoto is of both types.
Germoni's formula and the Su-Zhang formula for $\fgl(m|n)$ are of type I;
the Gruson-Serganova formula for $\osp$ and the Su-Zhang formula for $\fq_n$ are of type II. For the exceptional algebras the 
character formulas in~\cite{Ger},\cite{Lilit} can be rewritten
as type I or as type II formulas.

\subsection{The Gruson-Serganova type character formula}\label{overview}  
In this paper we obtain a formula of type II for $\fgl(m|n)$ (by above,
such formulas were obtained early for all other cases).
The main difference of the $\fgl(m|n)$-case is that 
$\Sigma_L$ depends on $L$.

Let  $\Irr$
be the set of isomorphism classes of the finite-dimensional irreducible
$\fg$-modules. 
 The terms $\{\ch L,\ L\in\Irr \}$ form  a natural
basis of the character ring. We say $\ch L$ is given by a \emph{Gruson-Serganova type character formula} if it can be written as
 a sum 

\begin{equation}\label{GSformula}
Re^{\rho}\ch L=\sum_{L'\in \Irr} b_{L,L'} \KW(L'),\end{equation}
where the  Kac-Wakimoto terms
$\KW(L)$ have the following properites:

\begin{enumerate}
\item
$\{\cE_{L}:=(Re^{\rho})^{-1}\KW(L),\ L\in\Irr \}$ form  a basis of the character ring
(where $R$ is the Weyl denominator). The terms $\cE_L$ are equal to the Euler characteristics $\mathcal{E}_{\lambda}$ (for a suitable choice of parabolic) of Penkov-Serganova and hence we can equally write the character of $L(\lambda)$ as a finite sum with integral coefficients in the Euler characters);

\item
the character formula $Re^{\rho}\ch L=\sum_{L'\in \Irr} b_{L,L'} \KW(L)$ is finite;

\item
the matrix $B:=(b_{L,L'})$ is a lower triangular matrix 
with integral entries and $1$s on the main diagonal; moreover, there exists a diagonal matrix
$D$ with $D^2=Id$ such that the entries of  $DBD^{-1}$ are non-negative (the entries of $DBD^{-1}$ can be interpreted
as a number of certain paths in a dirtected graph).

\item  $\KW(L):=j(L)^{-1}\KW(\lambda^{\dagger}, S_L)$, 
where $j(L)$ is a scalar, $\lambda^{\dagger}$ is the $\rho$-shifted highest weight of $L$ with respect to a certain base 
$\Sigma_L$ containing $S_L$.
\end{enumerate}

The scalar $j(L)$ is an order of the ``smallest factor'' in 
$\Stab_{W}\lambda^{\dagger}$ (for instance, $j(L)=|S|!$ for $\fgl$-case).
The set $S_L$ is  a maximal subset of $\Sigma_L$ satisfying 
$(\lambda^{\dagger}|S)=(S|S)=0$.

We call the cardinality of $S(L)$ the {\em tail } of $L$ ($\tail(L)$); this is a non-negative integer which is less than or equals to the atypicality of $L$.
If $b_{L,L'}\not=0$, then $L$ and $L'$ lie in the same block and
$\tail(L')\leq \tail(L)$. 
 We call the highest weight of $L$ a {\em Kostant weight}
if $\tail(L)$ is equal to the atypicality of $L$; in this case $b_{L,L'}=\delta_{L,L'}$.
From~\cite{CHR},\cite{CK} it follows that for $\fgl(m|n),\osp(2m+1|2n)$, 
$L(\lambda)$ satisfies  the  Kac-Wakimoto character formula if and only if 
its highest weight is a Kostant weight; this also holds for the $\osp(2m|2n)$-modules of atypicality greater than one, see~\ref{KostantKW}.

In the $\osp$-case~(\ref{GSformula}) was obtained in~\cite{GS} in a slightly different form (property (iv)  is established in~\Prop{propE} below).  
In this case $\Sigma_L=\Sigma$ is 
the usual ``mixed'' base  and $\lambda^{\dagger}=\lambda+\rho$.
In this paper  we will establish~(\ref{GSformula}) in the
$\fgl$-case. In this case 
$\Sigma_L$ depends on $L$; in Section~\ref{euler-gl} we describe the assignment
$L\mapsto \lambda^{\dagger}$; the image of this assigment 
can be naturally described in terms of weight diagrams.

For a finite-dimensional module $N$ define the $\xi$-character
$$\ch_{\xi} N:=\dim (N_{\ol{0}})_{\nu} e^{\nu}+\xi\dim (N_{\ol{1}})_{\nu} e^{\nu}$$ (where $\xi$ is a formal variable with $\xi^2=1$); clearly, the $\mathbb{Z}$-span of
 $\xi$-characters form a ring, which we denote by $\Ch_{\xi}(\fg)$.
The character
ring $\Ch(\fg)$  (resp.,
the supercharacter ring $\Sch(\fg)$) is a factor of $\Ch_{\xi}(\fg)$ by $\xi=1$ (resp., $\xi=-1$); these rings were explicitly described by A. N. Sergeev and A. P. Veselov (\cite{SV}).
Several important notions (for instance, $\dim N$ and $\sdim N$) can be viewed
as linear maps from  $\Ch_{\xi}(\fg)$. The Gruson-Serganova
formula give an expression of  $\Ch_{\xi} L$ see~(\ref{grGrS}) and
a formula for $\sch L$ in terms of $\cE_L^-$ (which are ``superanalogues''
of $\cE_L$).

An  important example is
the map  $\mult_{L'}: \Ch(\fg)\to\mathbb{Z}$ which
assigns to $\ch N$ the (non-graded) multiplicity $[N:L']$,
where $L'$ is a $\fg_0$-module; in~\Cor{cor123} we give formulas for
 $\mult_{L'}(\cE_L)$ and for $\dim (\cE_L)$. 

Another important linear map is induced  by the Duflo-Serganova monoidal functor $\DS_x: \Fin(\mathfrak{g}) \to \Fin(\mathfrak{g}_x)$ where $\mathfrak{g}_x$ is a smaller rank Lie superalgebra.
C. Hoyt and Sh. Reif \cite{HR} showed that $\DS_x$ induces a ring homomorphism
$$\ds_x:\Sch(\fg)\to \Sch(\fg_x)$$
given by $\ds_x: \sch N\mapsto\sch\DS_x(N)$; moreover,
 $\ds_x$ coincides with the evaluation of $\sch$ to a subalgebra $\fh_x\subset\fh$.
In~\Thm{thmcut} we show that $\ds_x(\cE^-_L)$ is given by a simple formula
(for the exceptional Lie superalgebras a similar formula
follows from~\cite{Gdex}).
Since $\DS_x$ preserves $\sdim$, this gives a formula for $\sdim \cE^-_L$, see~\Cor{corsdim}. It turns out that
$\sdim \mathcal{E}^-_L=0$ except for the case
when $\tail(L)$ equals to the defect of $\fg$
(by above,  $\tail(L)$ is less than or equals to the defect of $\fg$).

Using~(\ref{GSformula}) and the forementioned formulas one obtains
the expressions for $\sch \DS_x(L)$, $\sdim L$,  $[L:L']$ and $\dim L$ for $L\in\Irr$
(we do not write this long expressions).
 Note that  $\DS_x(L)$  is described  in~\cite{HW}, \cite{GH} and~\cite{Gdex}; various
formulas for $\dim L$ and $\sdim L$ appeared in~\cite{SuZh1},\cite{Lilit},\cite{HW},and~\cite{GH}.

\subsection{Euler characters}
The Euler characters  were originally defined as Euler characteristics of the cohomology of vector bundles on a super flag variety \cite{S-irr}. They were first introduced in \cite{P} \cite{PS1} \cite{PS2} and play also a crucial role in Brundan's work on characters in the $\mathfrak{q}(n)$-case \cite{B-q}\cite{B-Q-II}. We will describe 
the Euler characters for the ``core-free case'', i.e. for the principal block
of $\fgl(d|d)$ or $\osp(2d+t|2d)$, where $t=0,1,2$. (A similar description
works for all $\osp$-weights and for the ``stable weights'' in $\fgl$-case.)

Fix a flag of parabolic subalgebras
$$\fg=\fp^{(d)}\supset \fp^{(d-1)}\supset\ldots\supset \fp^{(0)}=\fb,$$
 where $d$ is the defect of $\fg$ and $\fl^{(i)}:=[\fp^{(i)},\fp^{i}]$ is of defect
$i$ (one has $\fl^{(i)}=\fgl(i|i)$ for $\fg=\fgl(d|d)$, $\fl^{(i)}=\osp(2i+t|2i)$,
for $\fg=\osp(2d+t|2d)$). 
For a pair of parabolic subalgebras $\mathfrak{q} \subset \mathfrak{p} \subset \mathfrak{g}$ containing a fixed Borel $\mathfrak{b}$ let $\Gamma_{\mathfrak{p},\mathfrak{q}}(V)$ denote the maximal finite dimensional quotient of the induced module $\cU(\fp)\otimes_{\cU(\fq)}V$. Then $\Gamma_{\fp,\fq}$ defines a functor from the category of finite-dimensional $\fq$-modules $\Fin(\fq)$ to $\Fin(\fp)$; and we denote by $\Gamma^i_{\fp,\fq}$ its derived functors as in \cite{GS}. For $\lambda,\mu\in\Lambda^+_{m|n}$ we consider the Poincar\'e polynomial
in the variable $z$
$$K^{\lambda,\mu}_{\fp,\fq}(z):=\sum_{i=0}^{\infty} [\Gamma_{\fp,\fq}^i
(L_{\fq}(\lambda)):L_{\fp}(\mu)]z^i,$$
where $L_{\fq}(\lambda)$ (resp. $L_{\fp}(\mu)$) stands for the corresponding simple $\fq$ (resp. $\fp$) module. Prop. 1 in~\cite{GS} expresses the Euler characteristic
\begin{equation}\label{EKla}
\cE_{\lambda,\fp}:=\sum_{\mu\in\Lambda^+_{m|n}}
K^{\lambda,\mu}_{\fg,\fp}(-1)\ch L(\mu)\end{equation}
in terms of $\ch L_{\fp}(\lambda)$.  The polynomials 
$K^{\lambda,\nu}_{\fp,\fq}(z)$ for the ``neighbouring parabolics'' were
computed in~\cite{Sselecta} in the $\fgl$-case and in~\cite{GS}
in the $\osp$-case. These polynomials
can be conveniently described
in terms of so-called ``arc diagrams'', see Section \ref{}.
The coefficients $K^{\lambda,\mu}_{\fg,\fp}(-1)$ can be computed iteratively
via the formula
\begin{equation}\label{iterative} 
K^{\lambda,\mu}_{\fg,\fq}(-1)=\sum_{\nu} K^{\lambda,\nu}_{\fp,\fq}(-1)
K^{\nu,\mu}_{\fg,\fp}(-1)\end{equation}
established in~\cite{GS}, Thm. 1.
 
  In the $\fgl$-case the matrix
$A_{\fb}:=(K^{\lambda,\mu}_{\fg,\fb}(-1))$ is invertible and the inverse
matrix can be explicitely described; this
gives the Serganova character formula~\cite{Sselecta}; this is an ``infinite formula'':
some  rows of $A^{-1}_{\fb}$ have infinitely many non-zero entries. For the
$\osp$-case one has $\cE_{\lambda,\fb}=0$ for some $\lambda$'s, so
the matrix $A_{\fb}$ is not invertible.
In order to obtain the  Gruson-Serganova character formula we take
for each $\lambda$ the ``maximal suitable parabolic'', setting
$\fp_{\lambda}:=\fp^{\tail(\lambda)}$, where
$\tail(\lambda)$ the maximal $i$ such that $\lambda|_{\fh\cap \fl^{(i)}}=0$.
(One has $\tail(\lambda)=\tail(L(\lambda))$ for $\osp$-weights and for the ``stable'' $\fgl$-weights).  The iterative formula~(\ref{iterative}) allows to interpret
 $K^{\lambda,\mu}_{\fg,\fp}(-1)$ in terms of ``decreasing paths'' in a certain directed  graph.  This graph has
several nice properties, which allow to express the inverse matrix
$(K^{\lambda,\mu}_{\fg,\fp^{\lambda}}(-1))^{-1}$ in terms of the paths in this graph; using~(\ref{EKla}) we obtain
\begin{equation}\label{for3}
\ch L(\lambda)=\sum_{\mu\in\Lambda^+_{m|n}} (-1)^{||\lambda||-||\mu||} 
d^{\lambda,\mu}_<\mathcal{E}_{\mu,\fp_{\mu}}\end{equation}
 for stable weights $\lambda$ where
$d^{\lambda,\mu}_<$ is the number of ``increasing paths'' from $\mu$ to $\lambda$ in the directed graph and $||\lambda||$ is defined in~\ref{normlambda}. 
Since $\dim L_{\fp_{\mu}}(\mu)=1$,  Prop. 1 in~\cite{GS} 
gives an explicit formula for $\mathcal{E}_{\mu,\fp_{\mu}}$; using the denominator
identity from~\cite{KW}  this formula can be rewritten
as
$Re^{\rho}\mathcal{E}_{\mu,\fp_{\mu}}=\KW(L)$. This gives (\ref{GSformula}).

For the $\osp$-case this program was executed in~\cite{GS}. In the first part
of our paper we execute a similar program for $\fgl$.
The graph in the $\fgl$-case has an easier description than in the $\osp$-case, but its structure is more complicated: by contrast with the $\osp$-case
each component has infinitely many sources. 
We show that in the  $\fgl$-case  each vertex has finitely many predecessors and this property allows
us to obtain~(\ref{for3}). We will
reveal some additional details in~\ref{method} below.

To link~(\ref{for3}) and~(\ref{GSformula})
 we show in Proposition \ref{propE} that the Euler characters are proportional to Kac-Wakimoto terms. This result is also fundamental when we study the effect of $ds$ on Euler characters in the second part of the paper.

A similar approach works
for the exceptional Lie superalgebras and for $\fq_n$; in these cases
 each component of the graph has a unique source. To the best of our knowledge,
the Gruson-Serganova type character formula is not known for the $\fp_n$-case;
we expect that in this case 
each component has infinitely many sources as for the $\fgl(m|n)$-case.

\subsection{Method of Proof}\label{method}
The proof of~(\ref{for3}) uses iterated parabolic induction.  We will  outline this proof  
 for the principal block
of $\fgl(d|d)$.  A similar proof
works for all $\osp$-weights and for the ``stable weights'' in $\fgl$-case.
On the other hand, the character formula for a simple module of atypicality $d$
can be reduced to this case.

The graphs $\hat{\Gamma}^{\chi}$ and $\Gamma^{\chi}$
are directed graphs with 
 the same set of vertices enumerated by the highest weights of
the irreducible modules in the principal block. In $\hat{\Gamma}^{\chi}$
the vertices $\mu,\lambda$
are joined by the
edge $\mu\overset{e}{\longrightarrow }\lambda$
if $K^{\lambda,\mu}_{\fp^{(s)},\fp^{(s-1)}}\not=\delta_{\lambda,\mu}$
for
such edge we set $b(e):=s$.
For $\Gamma^{\chi}$ we require that
$\fp^{s-1}\supset \fp_{\lambda}$ (in other words, $\Gamma^{\chi}$ can be obtained
from $\hat{\Gamma}^{\chi}$
by deleting the edges with $b(e)\leq \tail(\lambda)$).
By~\cite{MS}, for
the $\fgl(m|n)$-case 
$K^{\lambda,\mu}_{\fp^{(s)},\fp^{(s-1)}}$ is either $\delta_{\lambda,\nu}$ 
or $z^i$
with $i\equiv ||\lambda||-||\mu||$. In particular, if
$\mu$ and $\lambda$ are connected by an edge in $\hat{\Gamma}^{\chi}$, then 
$K^{\lambda,\mu}_{\fp^{(s)},\fp^{(s-1)}}=(-1)^{||\lambda||-||\mu||}$.

We define a ``decreasing path'' in $\hat{\Gamma}^{\chi}$, $\Gamma^{\chi}$  as a path with a decreasing 
function $b$. The formula~(\ref{iterative}) allows to express
$K^{\lambda,\mu}_{\fg,\fb}(-1)$  can be written as a sum of
$(-1)^{length(P)+||\lambda||-||\mu||}$, where $P$ runs through the decreasing paths from
$\mu$ to $\lambda$ 
in $\hat{\Gamma}^{\chi}$.
 We show that 
$K^{\lambda,\mu}_{\fg,\fp_{\lambda}}(-1)$
has the similar formula in terms of  the decreasing paths
in ${\Gamma}^{\chi}$ (the proof uses the fact that a path 
in ${\Gamma}^{\chi}$ lies in $\Gamma^{\chi}$ if the last edge
of this path lies in $\Gamma^{\chi}$).

The matrix 
$A_{>}:=(K^{\lambda,\mu}_{\fg,\fp_{\lambda}}(-1))$ is invertible;
by above, its entries can be written in terms of the decreasing paths from
$\mu$ to $\lambda$ in ${\Gamma}^{\chi}$, where the decreasing path
is defined in terms of the function $b$. We change $b$ to another function
$b'$ in such a way that
the function $b(e)$ is substituted by another function $b'(e)$
such that a path is decreasing for $b$ if and only if it is decreasing for $b'$. 
 The function $b'$ has the following advantage:

(*) if $\nu$ is the start of an edge $e_1$ and the end of an edge
$e_2$, then $b'(e_1)\not=b'(e_2)$ 

(this propery does not hold for the function $b$).
By above, $K^{\lambda,\mu}_{\fg,\fp_{\lambda}}(-1)$
can be written as a sum of
$(-1)^{length(P)+||\lambda||-||\mu||}$, where $P$ runs through the descreasing paths
from $\mu$ to $\lambda$
in ${\Gamma}^{\chi}$. The property (*) implies that the entries of $A_>^{-1}$  
can be written as a sum of
$(-1)^{||\lambda||-||\mu||}$, where $P$ runs through the {\em increasing} paths
with respect to $b'$; this gives the formula~(\ref{for3}) ($d^{\lambda,\mu}_{<}$ stands for the number of increasing paths
from $\mu$ to $\lambda$ with respect to $b'$). The graph $\hat{\Gamma}^{\chi}$ and the
functions $b,b',\deg(e)$ can be naturally described in terms of arc diagrams, (see~\ref{identific-graph}).

In order to prove the finiteness of the formula~(\ref{for3}) we 
show that each vertex $\lambda$ has a finite set of predecessors in 
the $\Gamma^{\chi}$ (for $\fgl$-case this property does not hold for
$\hat{\Gamma}^{\chi}$;
for $\osp$ and $\fq$-cases the property hold in both cases, since 
$\{\mu \in\Lambda_{m|n}| \mu\leq\lambda\}$ is finite).

\subsection{Modified superdimensions} The Kac-Wakimoto conjecture states that
  $\sdim L(\lambda)\not=0$ if and only if $L(\lambda)$
has a maximal atypicality; this conjecture was
 proven by V. Serganova  in \cite{Skw}. In the $\fgl(m|n)$-case
 $\sdim L(\lambda)$
was computed in \cite{HW}.

Consider the case $\fg=\fgl(m|n), \osp(M|N)$.
 Fix a triangular decomposition in the {\it usual way} (see \cite{GS},\cite{GSBGG}, etc.), i.e. a distinguished base for $\fgl(m|n)$ and the mixed base for $\osp(m|2n)$. 
Let $L(\lambda)$ be a finite-dimensional
 simple module of atypicality $k$. In this case applying $\DS_x$ with $rk(x) = k$ to $L(\lambda)$ gives by 
Theorem \ref{thmcut} an isotypic representation $ L(\lambda')^{\oplus m(\lambda)}$ of $\mathfrak{g}_x = \mathfrak{gl}(m-k|n-k)$ in the $\mathfrak{gl}$-case and in the $\mathfrak{osp}$-case either an isotypic representation $L(\lambda')^{\oplus m(\lambda)}$ 
of $\mathfrak{g}_x = \mathfrak{osp}(m-2k|2n-2k)$ (if $L(\lambda')$ is $\sigma$ invariant for the involution $\sigma$ of $OSp$, see Section \ref{sigma}) or $ ( L(\lambda') \oplus L(\lambda')^{\sigma})^{\oplus m(\lambda)}$ else. We put $L^{core} = L(\lambda')$ in the $\mathfrak{gl}$ and $\mathfrak{osp}(2m+1|2n)$-case and \[ L^{core} := \begin{cases} L(\lambda') & \text{ if } \lambda' \text{ is } \sigma- \text{invariant} \\ L(\lambda') \oplus L(\lambda')^{\sigma} & \text{ else } \end{cases} \] in the $\mathfrak{osp}(2m|2n)$-case. Then $L^{\core}$ only depends on the central character of $\lambda$. Using this notation we obtain in case where the atypicality of $L(\lambda)$ equals the rank of $x$ the uniform formula 
$$\DS_x(L(\lambda)) \cong \Pi^{i} 
(L^{core})^{\oplus m(\lambda)}$$ for some parity shift $\Pi^{i}$.

Identifying $\fg_x$ with a subalgebra
of $\fg$ as in~\cite{DS}, we can interpret the above formula as follows:
for a simple $\fg_x$-module $L'$ the  ``super multiplicity'' of $L'$ in
 $L(\lambda)$ is zero if $[L^{core}:L']=0$ and is $\pm m(\lambda)$
otherwise, see~\ref{supermult}.

If $L(\lambda)$ is maximal atypical, $\fg_x$ is one of the algebras $\fgl_k$, $\mathfrak{o}_k$, $\mathfrak{sp}_k,\mathfrak{osp}(1|2k)$.

%The numbers $m(\lambda)$ can be computed in the equal rank case. In the $\fgl(n|n)$-case  $|\sdim L(\lambda)| = m(\lambda)$
%is equal to the number of
%increasing paths from $\lambda$ to the Kostant weights
%(these are $\mu$'s with $\dim L(\mu)=1$). 
%In the $\osp(2n+1|2n), \osp(2n+2|2n)$ (resp., $\osp(2n|2n)$)-case
%one has $m(\lambda) = \sdim L(\lambda)=d^{\lambda,0}_{<}$ 
%(respectively $|\sdim L(\lambda)|=2 d^{\lambda,0}_{<}$),
%where  $d^{\lambda,0}_{<}$ is the number of
%increasing paths from $\lambda$ to $0$. 
%(In all cases it is easy to see
%that the existence of such paths is equivalent to
%the maximal atypicality condition). Therefore we reprove the Kac-Wakimoto conjecture and establish a combinatorial expression for the superdimension in the $\mathfrak{osp}(m|2n)$-case.
%
%\textcolor{blue}{I am confused about factor $2$ in the $\osp(2n|2n)$-case.
%I think $L^{\core}=\mathbb{C}$ is $\sigma$-invariant in this case
%(for $x$ of rank $n$ $\fg_x=\osp(0|0)=0$), so we do not have $2$s. 
%Also in our graphs the paths are from the Kostant weights to $\lambda$ and 
%not like in the  above text.
%I think you suggested to define Kostant weights in all cases. Let us  change the above paragraph to  something like:}

The numbers $m(\lambda)$ can be computed in the equal rank case. 
In this case   $m(\lambda)=|\sdim L(\lambda)|$
is equal to the number of
increasing paths from the Kostant weights, which are
 $\mu$'s with $\dim L(\mu)=1$, to $\lambda$.
(In all cases it is easy to see
that the existence of such paths is equivalent to
the maximal atypicality condition). Therefore we reprove the Kac-Wakimoto conjecture and establish another combinatorial expressions for the superdimensions.

%The category $\Fin(\mathfrak{osp}(m|2n))$ is equivalent to the category of $SOSp(m|2n)$-modules. If one works instead with $OSP(m|2n)$-modules, the behaviour of $DS$ simplifies since we no longer have to care about $\sigma$-twists and we obtain the uniform formula therefore \[ \sdim(L(\lambda)) = \pm m(\lambda) \sdim(L^{core}).\]

If $L(\lambda)$ is not maximal atypical, one can introduce a modified superdimension $\sdim^k$ on the thick tensor ideal spanned by the irreducible representations of atypicality $k$ instead. We show that the modified superdimension is given by the formula $\sdim^k(L(\lambda)) = \pm m(\lambda) \sdim^0(L^{core}) $ for $L(\lambda)$ of atypicality $k$ where $sdim^0$ is the (unique up to a scalar) modified superdimension on the thick ideal of projective objects, reproving results of \cite{Skw} and \cite{Kujawa-generalized}. 

 By~\cite{GH}  the isotypic multiplicity in $\osp$-case can be expressed  in terms of the arc diagram of $\lambda$.

\subsection{Structure of the article}
We recall some backgrounds in Sections \ref{sec:prel} and \ref{sec:weights}. In particular, in Section~\ref{sec:weights} we discuss
stability, tail and weight diagrams.
We define parabolic induction functors and their derived versions in Section \ref{sec:para-ind}. The main results - formula~(\ref{GSformula}) for $\fgl(m|n)$ and the behaviour of $\mathcal{E}^-_{\mu}$'s under $ds$ - are proven in Section \ref{sec:character} and Theorem \ref{thmcut}. For the  relationship of~(\ref{GSformula})  to other existing $\fgl(m|n)$-character formulas (notably the one from Su-Zhang) see \ref{other-formulas}. Section \ref{euler-gl} deals with 
$\cE_{\mu,\fp_{\mu}}$ and $\KW(L)$
 in the $\mathfrak{gl}(m|n)$-case; in particular, we describe the assignment
$L\mapsto \lambda^{\dagger}$ and establish property (iv).
The results on superdimensions and modified superdimensions are assembled in 
Section~\ref{sec:sdim}.
 We discuss properties of $KW(L)$ in Appendix~\ref{sec:RR}.  
In~\ref{KWR} we compute $\dim\cE_{\lambda}$
and obtain a formula for the multiplicity of a $\fg_0$-module in
$L(\lambda)$, see~\Cor{cor123}, \ref{Lasg0} and~\ref{xiversion} for the graded versions.

%Section \ref{exaDS} collects some example computations of $DS(L(\lambda))$. %and section \ref{sec:BS} is a dictionary between our conventions and the work of Ehrig and Stroppel.

{\bf  Acknowledgments.} The  authors are grateful to  V.~Hinich,
 V.~Serganova  for numerous 
helpful suggestions and to Sh.~Reif, A.~Sherman and C.~Stroppel for
stimulating discussions. 

The research of Thorsten Heidersdorf was partially funded by the Deutsche Forschungsgemeinschaft (DFG, German Research
Foundation) under Germany's Excellence Strategy – EXC-2047/1 – 390685813.

\subsection{Index of definitions and notation} \label{sec:app-index}
Throughout the paper the ground field is $\mathbb{C}$; 
$\mathbb{N}$ stands 
 for the set of non-negative integers. We will use the standard
Kac notation for the root systems. 

\begin{center}
\begin{longtable}{llr}

$\KW(\lambda, S), \Ch(\fg), \Sch(\fg)$    & &\ref{overview}\\

$\cF,\ \Lambda^+_{m|n}$ & & \ref{tildeF}\\

$\sigma$ & & \ref{sigma} \\

$\fg_x,\ rk(x), \ds_x,\ds_j$ & & \ref{DSdefi}\\

$S_s,\ \Sigma\ , \rho$ & & \ref{tri}\\

iso-set, $\at(\lambda), \at(L)$,  stable weight, $\tail(\lambda), \fg_{\lambda}$  & & \ref{atypdef}\\

$\Lambda^{\geq}_{m|n}$, weight diagrams,$\diag(\lambda)$, stability and tail for the diagrams & & \ref{wtdiag}\\

$\core(\lambda),\chi_{\lambda}, \Lambda^{\chi}, \core(\chi),$, core-free, $\howl(\lambda)$ & & \ref{corewt}\\

$||\lambda||_{gr},||\lambda||,\tau$, Kostant weights  & & \ref{normlambda}\\

$K^{\lambda,\mu}_{\fp,\fq}(z)$ & & \ref{sec:pi}\\

$\cE_{\lambda},\ \fp_{\lambda}$ & & \ref{termsE}\\

$\rho_L$ & & \ref{propE}\\

arc diagrams for $\fgl$, $b(\nu;\lambda)$, $b'(\nu;\lambda)$ & & \ref{Poincaregl}\\

graph $D_{\fg}$ for $\fgl$, $D_{\fg}^{\chi}$  & & \ref{graphDg}\\

$\Lambda^{\chi}_{st}, d'_{\lambda,\mu}, d^{\lambda,\mu}_<$, $T_{a,a+1}$  & &
\ref{Taa+1}\\

$j(\nu), \KW(\nu), \KW(L) $ & & \ref{anotheryear} \\

$\hwt_{\Sigma'} L$ & & \ref{hwt}\\

$\hat{\Gamma}^{\chi}_{st}, \Gamma^{\chi}_{st}, \hat{\Gamma}^{\chi}, 
\Gamma^{\chi} $, increasing/decreasing paths & &
\ref{graphs}\\

$\lambda^{\dagger}$, $\tail(\nu), \KW(\nu)$ for non-stable weights & & \ref{tail-up}\\

$\cE^-_{\nu}$ & & \ref{defpi}\\

$L^{\core}$ & & \ref{core}\\

$\sgn(w), \jJ_Y,\ \ R_0,\ R_1,\ \rho_0,\ \rho_1,\ Re^{\rho}$ & & \ref{JW}\\

$\cR, \cR_{\Sigma}, \ \supp$ & & \ref{cR}\\

$\cP_{\chi}, \Theta^{V}_{\chi,\chi'} $ & & \ref{trans}\\

\end{longtable}
\end{center}

%%%%%%%%%%%%%%%%%%%%%%%%%

\section{Preliminaries} \label{sec:prel}

We denote by $\Pi$ the parity change functor. Throughout
the Sections~\ref{sec:prel}--\ref{sec:sdim} $\fg$ stands for one of the Lie superalegbras $\fgl(m|n),\osp(2m|2n)$ or $\osp(2m+1|2n)$.

\subsection{Notation}\label{tildeF}
We use the standard notation: 
the root system $\Delta$ lies in the lattice $\Lambda_{m|n}\subset \fh^*$ spanned by $\{\vareps_i\}_{i=1}^m\cup\{\delta_i\}_{i=1}^n$. We denote by $\Lambda$ the lattice
spanned by $\{\vareps_i\}_{i=1}^{\infty}\cup\{\delta_i\}_{i=1}^{\infty}$
and view $\Lambda_{m|n}$ as a subset of $\Lambda$. We define the parity homomorphism
$p:\Lambda\to\mathbb{Z}_2$ by
$p(\vareps_i)=\ol{0}$, $p(\delta_j)=\ol{1}$ for all $i,j$.

The category $\Fin$ of finite dimensional representations of $\mathfrak{g}$ with parity preserving morphisms is the direct sum of two categories: $\tilde{\cF}$
with the modules whose weights lie in $\Lambda_{m|n}$ and   $\tilde{\cF}^{\perp}$
with the modules whose weights lie in $\fh^*\setminus \Lambda_{m|n}$.
The category $\tilde{\cF}^{\perp}$
 is semisimple and $\DS_x(\tilde{\cF}^{\perp})=0$ for $x\not=0$;
 all simple modules in  $\tilde{\cF}^{\perp}$
 are typical,
their characters  are given by the Weyl-Kac character formula
(in our notations $\sch L(\lambda)=\mathcal{E}^-_{\lambda}$).

The category $\tilde{\cF}$  is canonically
isomorphic to the category of $G$-modules, where $G$ is a classical supergroup corresponding to $\fg$:

$G:=GL(m|n)$ for $\fgl(m|n)$ and $G:=SOSp(m|n)$ for $\fg=\osp(m|n)$.

We fix
 the same  triangular decomposition as in \cite{GS},\cite{GSBGG},\cite{MS}:
for $\fgl(m|n)$ we choose the base $\Sigma$ which contains only one odd
root $\vareps_m-\delta_1$ and in the $\osp$-case we choose a base $\Sigma$
which contains a maximal possible number of odd roots;
we always consider $\Sigma$ as
the ordered set with respect to the usual order (see examples in \ref{tri} below).

We denote by $\Lambda^+_{m|n}$ the set of dominant weights in $\Lambda_{m|n}$:
$$\Lambda^+_{m|n}:=
\{\lambda\in\Lambda_{m|n}|\ \dim L(\lambda)<\infty\}.$$
The simple modules in $\tilde{\cF}$ are of the form $L(\lambda),\Pi(L(\lambda))$ for
$\lambda\in\Lambda_{m|n}^+$.

The category $\tilde{\cF}$ decomposes into a direct sum two  categories \[ \tilde{\cF} = \cF \oplus \Pi \cF \] such that the simple objects in $\cF$ are labelled by the dominant integral weights. Note that $\tilde{\cF}$ and $\cF$
are tensor categories.

\subsection{$OSp, SOSp$ and $\sigma$}
\label{sigma}

One has $O(2r+1)=SO(2r+1)\times\mathbb{Z}_2$ and
$O(2r)=SO(2r)\rtimes\mathbb{Z}_2$; we can choose the subgroup
$\mathbb{Z}_2$ in such a way that $\mathbb{Z}_2$
acts on $\osp(2r|2n)$ by an involutive automorphism
$\sigma$ which stabilizes the Cartan algebra $\fh$. For $r>1$
(i.e., $\fo_r\not=\mathbb{C}$),
$\sigma$ induces a Dynkin diagram
involution   given by
$$\sigma(\delta_j)=\delta_j,\ \ j=1,\ldots,n;\ \
\sigma(\vareps_i)=\vareps_i\ \text{ for } i=1,\ldots,r-1;\ \
\sigma(\vareps_r)=-\vareps_r.$$

For odd $m = 2r+1$ the orthosymplectic supergroup $OSp(m|2n)$ is a direct product \[ OSp(2r +1|2n) \cong SOSp(2r+1|2n) \times \mathbb{Z}_2\] where the non trivial element of $\mathbb{Z}_2$ acts as minus the identity. For even $m=2r$ it is a semidirect product \[ OSp(2r|2n) \cong SOSp(2r|2n) \rtimes \mathbb{Z}_2. \] The underlying even group of $OSp(m|2n)$ is $O(m) \times Sp(2n)$ and $SO(m) \times Sp(2n)$ in the $SOSp$-case.

The automorphism $\sigma$ can be extended to $\fosp(2r|2n)$. For $m > 1$ the involution $\sigma$ on $\mathfrak{h}^*$ is given by \begin{align*} \sigma(\delta_j) & = \delta_j \ & j=1,\ldots,n \\ \sigma(\vareps_i) & = \vareps_i & i =1,\ldots,r-1 \\ \sigma(\vareps_r) & = - \vareps_r & .
\end{align*}

A finite-dimensional $SO(2r)$-module $N$ can be extended to $O(2r)$
if and only if $N^{\sigma}\cong N$. Similarly, a
finite-dimensional
$SOSp(2r|2n)$-module $N$ can be extended to $OSp(2r|2n)$
if and only if $N^{\sigma}\cong N$. See also \cite{ES} for more details.

\subsection{The $\DS$-functor}\label{DSdefi}
The $\DS$-functor was introduced in \cite{DS}. We recall the definition
 below.

For a $\fg$-module $M$ and $g\in\fg$ we set
$$M^g:=\Ker_M g.$$

We fix now an $x \in \fg_1$ with $[x,x]=0$.
We set $\fg_x:=\fg^{\ad x}/[x,\fg]$; note that $\fg^{\ad x}$ and $\fg_x$ are  Lie superalgebras.
For a $\fg$-module $M$ we set
$$\DS_x(M)=M^x/xM.$$
Observe that $M^x, xM$ are $\fg^{\ad x}$-invariant and $[x,\fg] M^x\subset xM$,
so $\DS_x(M)$ is a $\fg^{\ad x}$-module and $\fg_x$-module.
Thus $\DS_x: M\to \DS_x(M)$ is a symmetric monoidal functor from the category of $\fg$-modules to
the category of $\fg_x$-modules.

\subsubsection{Remark}\label{supermult}
Notice that the action of $x$ provides a $\fg^{\ad x}$-isomorphism $M/M^x\iso \Pi(xM)$.
This implies that the ``super multiplicity'' of a simple $\fg^{\ad x}$-module $L'$
in a $\fg$-module $M$ equals to the ``super multiplicity'' of $L'$
in the $\fg^x$-module $\DS_x(M)$:
$$[M:L']-[M:\Pi(L')]=[\DS_x(M):L']-[\DS_x(M):\Pi(L')].$$

In many examples $\fg_x$ can be idenitified with a subalgebra of $\fg$;
in this case the same holds for a simple $\fg_x$-module $L'$.
The examples of such situation includes the cases when $\fg$
is a finite-dimensional Kac-Moody algebra (and $x$ is arbitrary), see~\cite{DS}.

\subsubsection{}
Let $G_0$ denote $GL(m) \times GL(n)$ (the $\fgl$-case) or $O(m) \times Sp(2n)$ (the $\fosp$-case). Then there exists $g\in G_0$
and isotropic mutually orthogonal  linearly independent roots
$\alpha_1,\ldots,\alpha_j$ such that
$$Ad_g(x)=x_1+\ldots+x_j,\ \ \text{ where }x_i\in\fg_{\alpha_i}.$$
The number $j$ does not depend on the choice of $g$ and
is denoted by $\rank x$ (or $rk(x)$) \cite{DS}. The Lie superalgebra $\fg_x$ depends only on the rank of $x$. For $rk(x) = k$ we have \begin{align*} \fg_x \cong \begin{cases} \mathfrak{gl}(m-k|n-k) & \mathfrak{g} = \mathfrak{gl}(m|n) \\ \mathfrak{osp}(m-2k|2n-2k) & \mathfrak{g} = \mathfrak{osp}(m|2n). \end{cases} \end{align*}

Take $x:=x_1+\ldots+x_j$ as above. Then 
the algebra
$\fh_x:=\fh^{\ad x}/([x,\fg]\cap\fh)$ is a Cartan subalgebra of $\fg_x$.
The functor $\DS_x$ induces a  ring homomorphism
$\ds_x:\Sch(\fg)\to \Sch(\fg_x)$ such that
$$\sch \DS_x(N)=\ds_x(\sch N)$$
for each $N\in\Fin(\fg)$. This homomorphism can be described as follows:
the restriction $f\mapsto f|_{\fh^x}$ gives a ring homomorphism
$\Sch(\fg)\to \Sch(\fg^{\ad x})$; the image of this map
lies in $\Sch(\fg_x)$ (which is a subring in $\Sch(\fg^{\ad x})$)
and  $\ds_x:\Sch(\fg)\to \Sch(\fg_x)$ is the corresponding map. 
If we choose $\fh_x\subset \fh^x$ such that $\fh^{\ad x}=\fh_x\oplus ([x,\fg]\cap\fh)$, then $\ds_x$ is given by $f\mapsto f|_{\fh_x}$, see~\cite{HR}, Lemma 4.
 
\subsubsection{}

In this paper we will describe the action of $\ds_x$ on a certain
basis of $\Sch(\fg)$. We do not use $\DS$, but $\ds$
only; and while $DS_x$ depends on $x$ (even for fixed rank \cite{HW}), $\ds_x$ depends only on the rank of $x$ and we simply write $\ds_k$ for $ds_x$ with $rk(x) = k$. Then
\begin{equation}\label{DSxj}
\ds_j=(\ds_1)^j.
\end{equation}

\section{Weights, roots and diagrams} \label{sec:weights}
We use the standard notation for the roots of $\fg_0$ and denote by
$\Pi_0$  a standard set of simple roots. In what follows 
we consider only bases $\Sigma$ of $\Delta$ which are compatible with $\Pi_0$, 
that is $\Delta^+(\Sigma)_0=\Delta^+(\Pi_0)$. By~\cite{Sgrs}, all such bases are connected by chains of odd reflections. These bases can be encoded by
words consisting of $m$ letters $\vareps$ and $n$ letters
$\delta$ (see examples below). 

We fix a standard bilinear form on $\fh^*$: $(\vareps_i|\vareps_j)=\delta_{ij}=-
(\delta_i|\delta_j)$, $(\vareps_i|\delta_j)=0$.

\subsection{The base $\Sigma$ and the sets $S_s$}\label{tri}
We set $S_0:=\emptyset$ and introduce the sets $S_s$ for 
$s=1,\ldots,\min(m,n)$ as follows. 
$$S_s:=\left\{\begin{array}{ll}
\{\varesp_{m+1-i}-\delta_i\}_{i=1}^s &\ \text{for $\fgl(m|n)$}\\
\{\delta_{n-i}-\vareps_{m-i}\}_{i=0}^{s-1} & \text{ for $\osp(2m|2n)$}\\
\{\vareps_{m-i}-\delta_{n-i}\}_{i=0}^{s-1} & \text{ for $\osp(2m+1|2n)$.}\\
\end{array}\right.$$
Notice that $S_{\min(m,n)}$ is a basis of a maximal isotropic subspace of $\fh^*$.

For $\fgl(m|n)$ we take the base $\Sigma$ corresponding to the word
$\vareps^m\delta^n$:
$$\Sigma:=\{\vareps_1-\vareps_2,\ldots,\vareps_{m-1}-\vareps_m,
\vareps_m-\delta_1,\delta_1-\delta_2,\ldots,\delta_{n-1}-\delta_n\}.$$
For $\osp$-case
we denote by $\Sigma$  a base containing  $S_{\min(m,n)}$ (such a base is unique):
this is the base $\delta^{n-m}(\delta\vareps)^m$ (resp.,  $\vareps^{m-n}(\delta\vareps)^n$ )
for $\osp(2m|2n)$ with $n\geq m$ (resp., $n\leq m$) and
$\delta^{n-m}(\vareps\delta)^m$ (resp.,  $\vareps^{m-n}(\vareps\delta)^n$ )
for $\osp(2m+1|2n)$ with $n\geq m$ (resp., $n\leq m$). For instance,
$$\Sigma=\{\delta_1-\delta_2,\ldots,\delta_{n-m+1}-\vareps_1,\vareps_1-\delta_{n-m+2},\delta_{n-m+2}-\vareps_2,\ldots,\vareps_{m-1}-\delta_{n},
\delta_n\pm\vareps_m
\}$$
for $\osp(2m|2n)$ with $n\geq m$.

\subsubsection{Remark} In~\cite{GH} we used the same bases $\Sigma$, but in the $\osp(2m|2n)$-case we chose different  $S$ for  different blocks.
The formulas in~\Prop{propE} hold for both choices of $S$.

\subsubsection{}
We denote by $\rho$ the Weyl vector of $\fg$ (see~\ref{Weylvector}).
Note that $\rho$ is unique for $\osp(M|2n)$ with $M\not=2$;
for $\osp(2|2n)$ we take $\rho=\sum_{i=1}^n (n-i)\delta_i$ and 
for $\fgl(m|n)$ we take $\rho = \sum_{i=1}^m (m+1-i)\vareps_i
-\sum_{i=1}^n i\delta_i$.  Note that
 $(\rho| S_{\min(m,n)})=0$; in the $\osp(2m|2n)$-case one has $\sigma(\rho)=\rho$.

\subsection{Atypicality, stability and tails}\label{atypdef}
We call $S\subset \Delta_1$ an {\em iso-set} if $S$ forms a basis
of an isotropic subspace in $\fh^*$, i.e.
$S$ is linearly independent and $(S|S)=0$. For instance, $S_r$ is  an iso-set.

For $\lambda\in {\fh}^*$ we denote by
$\at(\lambda)$ the atypicality of $\lambda$ (i.e. the
maximal cardinality of an iso-set orthogonal to $\lambda$).
The atypicality of $L(\lambda)$ is equal to $\at(\lambda+\rho)$.

The stability is usually introduced for a weight diagram.
Below we will introduce this notion for a weight (and a fixed base $\Sigma$).

We say that $\fg_s\subset\fg$ is an {\em  equal rank  subalgebra}
if $\fg_s$ is of the following form: 
$\fg_s=\fgl(s|s)$ for $\fgl$-case,
$\fg_s=\osp(2s+1|2s)$ for $\fg=\osp(2m+1|2n)$, 
$\fg_s=\osp(2s|2s)$ or $\osp(2s+2|2s)$ for $\fg=\osp(2m|2n)$,
and, in addition, $\fg_s$ has a base $\Sigma_s\subset\Sigma$.
Note that $\rho'_s:=\rho|_{\fg_s\cap \fh}$ satisfies $(\rho'_s|\alpha)=2(\alpha|\alpha)$ for each $\alpha\in\Sigma_s$, so $\rho'_s$ is ``a Weyl vector'' for
$\fg_s$ ($\rho'_s$ is the usual Weyl vector except for $\fgl$-case).
Observe that $\fg$ contains a unique copy
of $\fg_s$ for each $s$ with $0<s\leq \min(m,n)$.

\subsubsection{}\begin{defn}{}
In the $\fgl$-case we say that $\lambda\in\Lambda^+_{m|n}$
is a {\em stable} weight if there exists an equal rank  subalgebra
 $\fg_s\subset \fg$ such that for
 $$\at(\lambda+\rho)|_{\fh\cap\fg_s}=\at(\lambda+\rho)=s\  (=defect\ \fg_s).$$
\end{defn}

\subsubsection{}\begin{defn}{plambda}
Take $\lambda\in\Lambda^+_{m|n}$ which is assumed to be stable
for $\fgl(m|n)$-case. We denote by
  $\fg_{\lambda}$  the maximal equal  rank   subalgebra
of $\fg$ satisfying $\lambda|_{\fh\cap [\fg_{\lambda},\fg_{\lambda}]}=0$;
we call $\fg_{\lambda}$ the {\em tail subalgebra} of $\lambda$
and denote by $\tail(\lambda)$ the   defect of 
 $\fg_{\lambda}$. 
\end{defn}

\subsubsection{Examples}
The $\fgl(3|3)$-weight $\lambda$ with $\lambda+\rho=3\vareps_1+2\vareps_2-2\delta_2-5\delta_3$
is stable (with $\fg_s=\fgl(2|2)$), one has $\fg_{\lambda}=\fgl(1|1)$ and
$\tail(\lambda)=1$; the $\fgl(3|3)$-weight $\nu$ with $\nu+\rho=3\vareps_1+\vareps_2-\delta_2-5\delta_3$
is stable (with $\fg_s=\fgl(2|2)$), one has $\fg_{\nu}=\fgl(2|2)$ and
$\tail(\nu)=2$.

\subsection{Weight diagrams}\label{wtdiag}
Many properties of a finite dimensional representation $L(\lambda)$ can be better understood by assigning a {\em weight diagram} to the weight $\lambda$ (see e.g. \cite{Skw} \cite{HW} \cite{EAS}). These were first defined in \cite{BS} for $\mathfrak{gl}(m|n)$ and then for $\mathfrak{osp}(m|2n)$ in \cite{GS} \cite{Skw} and for $OSp$ in \cite{ES}. Note that the conventions how to draw these weight diagrams differ. The original weight diagrams of \cite{BS} use a different labeling of the vertices: Our $>$ is a $\times$, our $<$ a $\circ$ and our $\times$ a $\vee$. For the difference between the weight diagrams of \cite{GS} and \cite{ES} in the $\mathfrak{osp}$-case see \cite[Proposition 6.1]{ES}. We follow essentially \cite{GS}.

We denote by $\Lambda^{\geq}_{m|n}$ the following subgroup of $\fh^*$:
$$\Lambda^{\geq}_{m|n}:=
\{\displaystyle\sum_{i=1}^m a_i\vareps_i+\sum_{j=1}^n b_i\delta_j|\ a_1\in\frac{1}{2}\mathbb{Z},\
a_1-b_1\in\mathbb{Z},\  a_i-a_j,b_i-b_j\in\mathbb{N},
 \text{  for }i<j\}.$$
The set $\Lambda^{\geq}_{m|n}$ contains $\rho,
\Lambda^+_{m|n}$ and $\Lambda^+_{m|n}+\rho$.

\subsubsection{}
We assign the  weight diagram (a labeling of the real line $\mathbb{R}$ by certain symbols) 
to each weight $\displaystyle\sum_{i=1}^m a_i\vareps_i+\sum_{j=1}^n b_i\delta_j\in \Lambda^{\geq}_{m|n}$  using the following rules:

for $\fgl(m|n)$ with odd $m-n$ we put  $>$  (resp., $<$) at the position with the coordinate $j$ if $a_i=j$ (resp., $b_i=-j$)
for some $i$;

for $\fgl(m|n)$ with even $m-n$  we put  $>$  (resp., $<$) at the position with the coordinate $j+\frac{1}{2}$ if $a_i=j$ (resp., $b_i=-j$)
for some $i$;

for $\osp(2m|2n)$ we put  $>$  (resp., $<$) at the position with the coordinate $t$ if $|a_i|=j$ (resp., $|b_i|=j$)
for some $i$; if $a_m\not=0$  we add the sign $+$ (resp., $-$)
if $a_m>0$ (resp., $a_m<0$);

for $\osp(2m+1|2n)$ we put  $>$  (resp., $<$) at the position with the coordinate $j-1/2$ if $a_i=j$ (resp., $b_i=j$)
for some $i$; we add the sign $+$ (resp., $-$) if the zero position is occupied
by $\times^p$ for $p>0$ and $(\lambda+\rho|\vareps_i)=\frac{1}{2}$ for some $i$
(resp., $(\lambda+\rho|\vareps_i)\not=\frac{1}{2}$ for each $i$).

If $>,<$ occupy the same position we write these symbols as $\times$ ($\times^s$ stands for
$s$ symbols $<$ and $s$ symbols $>$; $\frac{>}{\times^s}$ stands for
$s$ symbols $<$ and $s+1$ symbols $>$). We put  an ``empty symbol'' $\circ$ at the  non-occupied
positions with the coordinates in $a_1+\mathbb{Z}$;  sometimes instead of $\circ$
we put its coordinate (for instance, $0\circ \times$ means that $\times$ has the coordinate $2$). For a diagram $f$ we denote by $f(a)$ the symbols
at the $a$-th position.

Note that for $\osp(2m+1|2n)$-case our  diagram
is obtained from  the diagram used in \cite{GS} by the shift
by $-1/2$.

\subsubsection{Examples}
The diagram of $\rho$ for $\osp(2n+1|2n)$
has the sign $-$ and contains $n$ symbols $\times$ in the zero position;
we write this as $(-) \times^n$; similarly, for $\osp(2n|2n)$
the diagram of $\rho$ is $\times^n$ and for $\fgl(3|3)$
the diagram of $\rho$ is $\times\times\times$, where the rightmost symbol
$\times$ appears in the  position $0$.
For $\fgl(m|n)$ we sometimes add a coordinate of $\circ$ instead one empty symbol; for instance, for $\fgl(4|3)$   the diagram of $\rho$  can be written as
$$0\times\times\times>\ \text{ or } -1\circ \times\times\times>.$$

\subsubsection{}
We assign to each $\lambda\in \Lambda^{\geq}_{m|n}$ the diagram of $\lambda+\rho$ constructed as above; this diagram will be denoted 
by $\diag(\lambda)$. This procedure gives a one-to-one correspondence
between $\Lambda^+_{m|n}$ and the diagrams containing $k$ symbols $\times$,
$m-k$ symbols $>$ and $n-k$ symbols $<$ (where $k\leq \min(m,n)$)
with the following additional properties:

\begin{itemize}
\item the atypicality of $L(\lambda)$ is equal to the number of symbols $\times$ in the
diagram of $\lambda+\rho$;

\item in the $\fgl$-case the coordinates of the occupied positions lie in $\mathbb{Z}$  and  each occupied position
contains exactly one of the signs $\{>,<,\times\}$;

\item in the $\osp$-case 
the coordinates of the occupied positions lie in $\mathbb{N}$ and
each non-zero occupied
position contains exactly one of the signs $\{>,<,\times\}$;

\item in the $\osp(2m|2n)$-case 
the zero position does not contain $<$, contains at most one symbol $>$ and
an arbitrary number of $\times$; a  diagram $f$ has a sign if and only if $f(0)=\circ$;

\item in the $\osp(2m+1|2n)$-case 
the zero position contains at most one of the symbols $>,<$ and
an arbitrary number of $\times$; 
a  diagram $f$ has a sign if and only if $f(0)=\times^i$ for $i>0$.
\end{itemize}

\subsubsection{Remark: $OSp(2m|2n)$-modules}
By~\ref{sigma} simple $OSp(2m|2n)$-modules are in one-to-one correspondence with the unsigned
$\osp(2m|2n)$-diagrams.

\subsubsection{Tail  in the diagramic language}\label{tails}
 It is easy to see that in the $\fgl$-case 
 $\lambda\in\Lambda^{\geq}_{m|n}$ is  stable
if and only if  all symbols $\times$ in $\diag(\lambda)$ preceed 
all symbols $<,>$. 

For $\lambda\in\Lambda^+_{m|n}$ we can easily express 
$\tail(\lambda)$ in terms of $f:=\diag(\lambda)$:

\begin{itemize}
\item
for $\osp(2m|2n)$-case $\tail(\lambda)$ equals to the number of symbols
$\times$ in the zero position of $f$;

\item for $\osp(2m+1|2n)$-case $\tail(\lambda)$ is the number of symbols
$\times$ in the zero position of $f$ if $f$ does not have $(+)$ sign
and is less by $1$ if $f$ has the sign $(+)$;

\item for a stable weight $\lambda$ in the $\fgl$-case $\tail(\lambda)$ equals to the maximal length of the subdiagram $\times\times\cdots\times$
which starts from the first symbol $\times$ in $\diag(\lambda)$. 
\end{itemize}

For instance, in the $\osp$-case 
$\tail(\circ\times\times)=0$ and $\tail((+)\times^3\times)=2$; in the $\fgl$-case one has
$\tail(\circ\times\times\circ\times\times\times)=2$.

Note that in the $\fgl$-case $\tail(\lambda)\not=0$ if 
$\lambda$ is an atypical stable weight.

\subsection{Cores and howls}\label{corewt}
We call the symbols $>,<$ the {\em core symbols}.
A {\em core diagram} is a
 weight diagram which does not contain symbols $\times$
 and does not have a sign.

For a weight diagram $f$ we denote by $\core(f)$ the core diagram which is
obtained from the diagram of $f$ by replacing all symbols $\times$ by $\circ$ and deleting the sign. For instance,
$\core(<\circ\times >)=<\circ\circ >$. For a weight $\lambda$ 
we set 
$$\core(\lambda):=\core(\diag(\lambda)).$$

\subsubsection{}\label{Llambda}
We say that a $\fg$-central character is {\em dominant} if
$\cF(\fg)$ contains modules with this central character.
We denote by $\chi_{\lambda}$ the central character
of $L(\lambda)$. For a dominant central character $\chi$ 
we set
$$\Lambda^{\chi}:=\{\lambda\in \Lambda^+_{m|n}| \chi_{\lambda}=\chi\}.$$

For $\fgl(m|n),\osp(2m+1|2n)$-case  the dominant central characters are parametrized by the core diagrams, i.e.
for $\lambda,\nu\in\Lambda^+_{m|n}$
$$\chi_{\lambda}=\chi_{\nu} \ \Longrightarrow\ \ \ \core(\lambda+\rho)=\core(\nu+\rho);$$
for $\osp(2m|2n)$ the same holds for the atypical dominant central characters and one has
$$\chi_{\lambda}\in\{\chi_{\nu},\chi_{\nu^{\sigma}}\} \ \Longrightarrow\ \ \ \core(\lambda)=\core(\nu).$$

For a dominant central character $\chi=\chi_{\lambda}$ we set 
$\core(\chi):=\core(\lambda)$.
By above, a dominant central character is determined by its core  for $\fg=\fgl(m|n), \osp(2m+1|2n)$;
for $\osp(2m|2n)$ this holds for atypical dominant central characters.

For $\osp(2m|2n)$-case we introduce $t\in\{0,2\}$ for each dominant
central character $\chi$ (resp., for each $\lambda\in\Lambda^+_{m|n}$) in the following way: 
 $t=0$ if $\core (\chi)$ (resp., $\core(\lambda)$)
has an empty zero position and $t=2$ otherwise (i.e., the zero position
is occupied by $>$); for $\osp(2m+1|2n)$ we set $t=1$ and for $\fgl(m|n)$ we set $t:=0$. We will sometimes use the notation $t(\chi)$ or $t(\lambda)$; one has $t(\lambda):=t(\chi_{\lambda})$.

\subsubsection{}
We say that a diagram $f$ is {\em core-free} if $\core(f)=\emptyset$
or $\fg=\osp(2m|2n)$ and $\core(f)=>$ ($>$ occupies the zero position).

\subsubsection{}
By~\cite{GS}, the blocks in $\cF(\fg)$ are parametrized by the dominant central characters; for $\fgl(m|n)$ the block of atypicality
$s$ is equivalent to the block $\chi_0$ in $\fgl(s|s)$;
for $\osp(M|2n)$-case the block of atypicality
$s$ is equivalent to the block $\chi_0$ in $\osp(2s+t|2s)$. 
The equivalences are described in \cite{GS}. 
For $\lambda \in  \Lambda^+_{m|n}$ let $\howl(\lambda)$ be the corresponding weight in  $\chi_0$.
Diagrammatically the passage from $\lambda$ to $\howl(\lambda)$ essentially amounts to removing the core symbols $<,>$ from $\diag(\lambda)$
except for the symbol $>$ at the zero position
in $\osp(2m|2n)$-case (see \cite{GS} \cite{GH} for details).
(In particular, $\howl(\lambda)$ has a core-free  diagram.)
If $\tail(\lambda)$ is defined, then $\tail(\lambda)=\tail(\howl(\lambda))$. 

For example, if $\diag(\lambda)=>\times<\times\circ<\circ\times$,
then
$$\diag(\howl(\lambda))=\left\{\begin{array}{lcl}
\times\times\circ\circ\times & & \text{ for }\fg=\fgl(4|5)\\
>\times\times\circ\circ\times & & \text{ for }\fg=\osp(8|10)\\
(+)\times\times\circ\circ\times & & \text{ for }\fg=\osp(9|10).\\
\end{array}\right.$$
%
%
% We say that $\lambda\in\Lambda^+_0$ if $\fg$ is one of the algebras
% $\fgl(n|n),\osp(2n+t|2n)$ with $t=0,1,2$ and $\chi_{\lambda}=\chi_0$.
% By above,  proofs can be reduced to the case when
%$\lambda\in \Lambda^+_0$.
%\textcolor{red}{Maybe the last sentence should be changed in view of
%the formula for $\cE_{\lambda}$ in the non-stable $\fgl$-case}
%
%Observe that $\lambda\in\Lambda^+_0$
%if the only non-empty symbols in $\lambda+\rho$ are $n$ symbols $\times$ if $\fg\not=\osp(2n+2|2n)$; for $\osp(2n+2|2n)$ the diagram of
%$\lambda+\rho$ contains, in addition, one symbol $>$ in the zero position.

\subsection{The functions $||\lambda||$ and 
$||\lambda||_{gr}$}\label{normlambda}
Let $\lambda\in\Lambda_{m|n}^+$ be such that $\lambda+\rho$ is atypical. Set $f:=\diag(\howl(\lambda))$.

\begin{defn}{}
Let   
  $a_1\leq \ldots\leq a_j$ be the coordinates of the symbols $\times$ in $f$ ($j=\at(\lambda+\rho)$). 

$$||\lambda||:=\left\{\begin{array}{lcl}
\displaystyle\sum_{i=1}^j a_i& & \text{ for } t(\lambda)\not=2\\
\tail(\lambda)-j+\displaystyle\sum_{i=1}^j a_i & & \text{ for }t(\lambda)=2\\
\end{array}\right.$$
and 
$$||\lambda||_{gr}:=\left\{\begin{array}{lcl}
\displaystyle\sum_{i=1}^j (a_i-a_j)-\frac{j(j-1)}{2}\ & & \text{ for }\fgl(m|n)\\
\displaystyle\sum_{i=1}^j a_i& & \text{ for }\osp(2m|2n)\\
j-\tail(\lambda)+\displaystyle\sum_{i=1}^j a_i & & \text{ for }\osp(2m+1|2n)\\
\end{array}\right.$$
\end{defn}

Notice that $||\lambda||_{gr}\in\mathbb{N}$  and that $||\lambda||\in\mathbb{N}$ for $\osp(M|N)$ and
$||\lambda||\in\mathbb{Z}$ for $\fgl(m|n)$. Moreover, $||\lambda||_{gr}=0$
if and only if $\howl(\lambda)=0$ in the $\osp$-case
and $\howl(\lambda)\in\mathbb{Z}(\sum_{i=1}^j (\vareps_i-\delta_i)$ for
$\fgl$-case (i.e., $\dim L(\howl(\lambda))=1$). 
In the $\osp(2m|2n)$-case  the condition
$||\lambda||=0$ is equivalent to  $\howl(\lambda)=0$   (resp.,
$\howl(\lambda)=0,\vareps_1$) for $\fg=\osp(2m|2n)$ (resp., for $\fg=\osp(2m+1|2n)$).
Note that in the $\fgl(m|n)$-case
$||\lambda||_{gr}$ is invariant under the shift of the diagram.

If $\lambda+\rho$ is typical, we set $||\lambda||_{gr}$
(we do not define $||\lambda||$ in this case).

%\subsubsection{Example}
%The singly atypical case ($j=1$) is described in Section~\ref{app-one}.
%In the notation of this section we have $||\lambda_i||=i$,
%$||\lambda_i||=0$  for $\fg=\fgl(1|1)$,
%$||\lambda_i||=|\lambda_i||_{gr}=|i|$ for $\fg=\osp(2|2)$;
%for $\fg=\osp(3|2),\osp(4|2)$ we have
%$||\lambda_i||_{gr}=i$  and
%$||\lambda_0||=||\lambda_1||=0$, $||\lambda_i||=i-1$
%for $i>1$.

\subsubsection{Remark}\label{plambdanu}
For $t\not=2$ one has
$(-1)^{p(\howl(\lambda))}=(-1)^{||\lambda||}$.
If $\lambda,\nu$ are stable $\fgl$-weights
with $\chi_{\lambda}=\chi_{\nu}$, then
$(-1)^{p(\lambda)-p(\nu)}=(-1)^{||\lambda||-||\nu||}$.

\subsubsection{Remark}\label{tau}
The odd-looking formulas for $||\lambda||$ with $t(\lambda)=2$ 
and for $||\lambda||_{gr}$ with $t(\lambda)=1$
can be interpreted as follows. 
Consider $f'$ which is obtained from $f$ by removing $>$ from
the zero position and then shifting all entires at the non-zero positions of $f$ by one position to the left; then $||\lambda||=\displaystyle\sum_{i=1}^j a'_i$, where $a'_i$ are the coordinates
of $\times$ in $f'$. The above operation induces 
a bijection
 $\tau$ between the core-free  $\osp(2m+2|2m)$-weights and  the core-free
$\osp(2m+1|2n)$-weights: this bijection, introduced in~\cite{GS},
assigns to $f$ the   diagram $f'$ with the sign
chosen in such a way that  $\tail(f)=\tail(\tau(f))$.
For instance, 
$$\tau(\overset{\times}{>}\circ\times)=-\times\times,\ \ \tau(\overset{\times}{>})=-\times,\ \tau(>\times)=+\times,\ 
\tau(>\circ \times)=\circ \times.
$$
One has
$$||\lambda||=||\tau(\lambda)||,\ \ \ ||\lambda||_{gr}=||\tau(\lambda)||_{gr}.$$

\subsubsection{}\begin{defn}{kostant}
We call $\lambda\in\Lambda^+_{m|n}$ {\em a Kostant weight}
if  $\dim L(\howl(\lambda))=1$. 
\end{defn}

Note that $\dim L(\howl(\lambda))=1$ means that $\howl(\lambda)=0$ (resp.,  $\howl(\lambda)\in\mathbb{Z}\str$) for the $\osp$-case 
(resp., for the $\fgl$-case).

Observe that $||\lambda||_{gr}=0$ if and only if
$\lambda$ is a Kostant weight ($||\lambda||_{gr}$ can be seen as the ``distance'' to the nearest Kostant weight). 
 For instance,  for $\fgl(3|3)$ and
 $\diag(\lambda)=\times\times\circ\times$ one has 
$||\lambda||_{gr}=1$.

For the $\fgl$-case this term was used in \cite{BS2}; in\cite{CHR} these weights are called {\em totally connected}. If in addition $\lambda$ is stable, such weight is called a {\em ground state} in \cite{HW} \cite{W}. The Kostant weights are precisely the weights where all $\times$ are adjacent to each other discounting possible core symbols.

\subsubsection{Remark}\label{KostantKW}
For the $\fgl(m|n)$-case the modules satisfying the KW-conditions (see~\ref{intro1}) were classified in~\cite{CHR}; for the
$\osp(M|N)$-case this was done in~\cite{CK}. The results of these classification can be formulated
as follows. Except for the case $\fg=\osp(2m|2n)$ with $t=0$ and atypicality $1$,
 $L(\lambda)$ satisfies the KW-conditions 
if and only if $\lambda$ is a Kostant weight.
 For the case $\fg=\osp(2m|2n)$ with $t=0$
all simple finite-dimensional modules of atypicality $1$ satisfy the KW-conditions. The latter case has the following interpretation.
 Let  $\cF(\osp(2m|2n))^{\chi}$ be  a block of 
atypicality $1$ with $t=0$. Since $\osp(2|2)=\fsl(1|2)$ we have
$$\cF(\osp(2m|2n))^{\chi}\iso \cF(\osp(2|2))^{\chi_0}=\cF(\fsl(1|2))^{\chi_0}{\iso}
\cF(\fsl(1|1))^{\chi_0}$$
so the image of each simple module $L\in \cF(\osp(2m|2n))^{\chi}$ 
is the trivial $\fsl(1|1)$-module (that is $\dim(\howl(\howl(L))=1$ even if
$\dim\howl(L)\not=1$).

From the above description, it follows that the
KW-conditions are  compatible with the equivalence of categories
given by the transaltion functors $T_{a,a+1}$ described in~\ref{Taa+1}.
This is not true in general:
the switch functor $\cF^{\chi_0}(\osp(2m+1|2n))\iso \cF^{\chi_0}(\osp(2m+1|2n))$
given by $N\mapsto (N\otimes V_{st})^{\chi_0}$ maps the trivial module
(satisfying the KW-conditions) to the standard module, which does not satisfy these conditions.

%%%%%%%%%%%%%%%%%%%%%%%%%%%%%%%%%%%

%%%%%%%%%%%%%%%%%%%%%%%%%%%%%%%%%%%%

\section{Parabolic induction, Euler characters and character formulas}
\label{sec:para-ind}
We define parabolic induction functors $\Gamma^i_{\fp,\fq}$ and the Poincar\'e polynomials $K^{\lambda,\mu}_{\fp,\fq}(z)$ in Section \ref{sec:pi} and Euler characters $\mathcal{E}_{\lambda}$ in Section~\ref{termsE}. We give a diagrammatic description of the Poincar\'e polynomials in the $\mathfrak{gl}(m|n)$-case. This leads to a character formula for $\ch L(\lambda)$.

\subsection{The functors $\Gamma^i_{\fp,\fq}$}\label{sec:pi}
Let $\fq\subset\fp\subset\fg$ be a pair of parabolic subalgebras containing $\fb$ and
let $V$ be a finite-dimensional
$\fq$-module. We denote by $\Gamma_{\fp,\fq}(V)$ the maximal
finite-dimensional quotient of the induced module $\cU(\fp)\otimes_{\cU(\fq)}V$. View $\Gamma_{\fp,\fq}$ as a functor from the category of finite-dimensional $\fq$-modules
to the category of finite-dimensional $\fp$-modules
and define the derived functors $\Gamma^i_{\fp,\fq}$ as in \cite{GS}
($\Gamma^i_{\fp,\fq}:=\Gamma_i(P/Q,\bullet)$ in the notations of \cite{GS}).
By \cite{GSBGG}
for $\fgl(m|n)$ these functors coincide with the functors
$\Gamma^i_{\fp,\fq}$ defined in \cite{MS}.

For $\lambda,\mu\in\Lambda^+_{m|n}$ we consider the Poincar\'e polynomial
in the variable $z$ as
$$K^{\lambda,\mu}_{\fp,\fq}(z):=\sum_{i=0}^{\infty} [\Gamma_{\fp,\fq}^i
(L_{\fq}(\lambda)):L_{\fp}(\mu)]z^i,$$
where $L_{\fq}(\lambda)$ (resp. $L_{\fp}(\mu)$) stands for the corresponding simple $\fq$ (resp. $\fp$) module.

\subsubsection{}
 Fix a central character $\chi$ and  a flag of parabolic subalgebras
$$\fg=\fp^{(d)}\supset \fp^{(d-1)}\supset\ldots\supset \fp^{(0)}=\fb,$$
 where $d$ is the defect of $\fg$ and $\fl^{(i)}:=[\fp^{(i)},\fp^{i}]$ 
is given by 
  $\fl^{(i)}=\fgl(i|i)$ for $\fg=\fgl(d|d)$, $\fl^{(i)}=\osp(2i+t|2i)$,
for $\fg=\osp(2d+t|2d)$. 

 The polynomials
 $K^{\lambda,\nu}_{\fp^{(i)},\fp^{(i+1)}}(z)$ for the ``neighbouring parabolics'' were
 given in~\cite{Sselecta} in the $\fgl$-case and in~\cite{GS}
 in the $\osp$-case. In the $\fgl$-case
 we will describe these polynomials 
 in terms of so-called ``arc diagrams'' in~\ref{Poincaregl} below. Using these polynomials
the values $K^{\lambda,\mu}_{\fg,\fp_{\lambda}}(-1)$ can be computed iteratively
 using the formula
\begin{equation}\label{iterative} 
K^{\lambda,\mu}_{\fg,\fq}(-1)=\sum_{\nu} K^{\lambda,\nu}_{\fp,\fq}(-1)
 K^{\nu,\mu}_{\fg,\fp}(-1)\end{equation}
 established in~\cite{GS}, Thm. 1.

\subsection{The terms $\mathcal{E}_{\lambda}$}\label{termsE}
Take $\lambda\in\Lambda^+_{m|n}$ which is assumed to be stable for $\fgl$-case.
Let $\fg_{\lambda}$ be the tail subalgebra of $\lambda$ (see \ref{plambda}). 
As in \cite{GS} we introduce
\begin{equation}\label{Elambda}
\mathcal{E}_{\lambda}:=R^{-1}e^{-\rho}\jJ_W
\bigl(\frac{e^{\lambda+\rho}}{\prod_{\alpha\in\Delta(\fg_\lambda)_1^+}
(1+e^{-\alpha})}\bigr),\end{equation}
see~\ref{JW} for notation.
Clearly, $\cE_{\lambda}\in\cR$, see~\ref{cR} for notation.
By \cite{GS}, Prop.1 (Euler characteristic formula) one has
\begin{equation}\label{Euler}
\mathcal{E}_{\lambda}=\sum_{\mu\in\Lambda^+_{m|n}}
K^{\lambda,\mu}_{\fg,\fp_{\lambda}}(-1)\ch L(\mu),\ \ \text{ where }
\fp_{\lambda}:=\fb+\fg_{\lambda}.
\end{equation}
The sum in the right-hand
side of the formula is finite (see, for example,~\cite{GS}, Lemma 3).
\subsubsection{Remark} The perspective of \cite{GS} \cite{S-irr} is a bit different. The $\mathcal{E}_{\lambda}$'s are defined as actual Euler characters . It is important not to confuse the Euler character $\mathcal{E}_{\lambda}$ of \cite{GS} with the Euler character $\mathcal{E}_{\lambda}$ of \cite{GSBGG}. In the latter case $\mathcal{E}_{\lambda}$ simply equals for $\mathfrak{gl}(m|n)$ the character of the Kac module
 $K(\lambda)$.

\subsubsection{}\label{E0}
In the case when  $\fg_{\lambda}=\fg$
the formula~(\ref{Delta1R}) gives
$\mathcal{E}_{\lambda}=e^{\lambda}=\ch L(\lambda)$.

\subsection{}
\begin{prop}{propE}
Take $\lambda\in\Lambda^+_{m|n}$ which is stable in the $\fgl(m|n)$-case.
Set $s:=\tail(\lambda)$.
\begin{enumerate}
\item
In the $\osp$-case we have
$$
j_sRe^{\rho}\mathcal{E}_{\lambda}=\KW(\lambda+\rho, S_s),
$$
where 
$j_s=\max(1,2^{s-1}s!)$ for $t=0$ and $j_s=2^ss!$
 for $t=1,2$. 
\item
If $\fg=\fgl(m|n)$ and $\lambda$ is stable, then
$$
j_sRe^{\rho}\mathcal{E}_{\lambda}=\KW(\lambda+\rho_L, S_s),
$$
where $j_s=(-1)^{[\frac{s}{2}]}s!$ and $\rho_L$ is the Weyl vector for the base $\vareps^{m-s}(\vareps\delta)^s\delta^{n-s}$.

\end{enumerate}
\end{prop}
\begin{proof}
For $\osp$-case set $\Sigma_L:=\Sigma$; for $\fgl$-case let $\Sigma_L$ be the base 
$\vareps^{m-s}(\vareps\delta)^s\delta^{n-s}$. Denote by
$\rho_{\lambda}$  (resp., $\rho'_{\lambda}$) the Weyl vector for $\Delta(\fg_{\lambda})$
 with respect to the base
$\Sigma\cap \Delta(\fg_{\lambda})$  (resp., $\Sigma_L\cap \Delta(\fg_{\lambda})$).
Let $W_{\lambda}\subset W$
be the Weyl group of $\fg_{\lambda}$. Note that $S_s$ is the maximail
iso-set in $\Delta(\fg_{\lambda})$.
Combining~(\ref{Delta1R}) and~(\ref{denom}) 
we obtain
$$\jJ_{W_{\lambda}}
\bigl(\frac{e^{\rho_{\lambda}}}
{\displaystyle\prod_{\alpha\in\Delta(\fg_{\lambda})_1^+}
(1+e^{-\alpha})}\bigr)=j_s^{-1}\jJ_{W_{\lambda}}
\bigl(\frac{e^{\rho'_{\lambda}}}
{\displaystyle\prod_{\alpha\in S_s}
(1+e^{-\alpha})}\bigr).$$

One has $\jJ_W=\jJ_{W/W_{\lambda}}\cdot\jJ_{W_{\lambda}}$, where $W/W_{\lambda}$
stands for any set of representatives.
Using $W_{\lambda}$-invariance
of $\lambda$ and $\rho-\rho_{\lambda}$ we obtain
$$\begin{array}{ll}
{R}e^{{\rho}}\mathcal{E}_{\lambda}&=\jJ_{W}\bigl(\frac{e^{\lambda+\rho}}
{\displaystyle\prod_{\alpha\in\Delta(\fg_{\lambda})_1^+}
(1+e^{-\alpha})}\bigr)=\jJ_{W/W_{\lambda}}\bigl(\jJ_{W_{\lambda}}
\bigl(\frac{e^{\lambda+\rho}}
{\displaystyle\prod_{\alpha\in\Delta(\fg_{\lambda})_1^+}
(1+e^{-\alpha})}\bigr)\bigr)\\
&=j_s^{-1}\jJ_{W/W_{\lambda}}\bigl(\jJ_{W_{\lambda}}
\bigl(\frac{e^{\lambda+\rho-\rho_{\lambda}+\rho'_{\lambda}}}
{\displaystyle\prod_{\alpha\in S_s}
(1+e^{-\alpha})}\bigr)\bigr)=j_s^{-1}\jJ_W
\bigl(\frac{e^{\lambda+\rho-\rho_{\lambda}+\rho'_{\lambda}}}
{\displaystyle\prod_{\alpha\in S_s}
(1+e^{-\alpha})}\bigr)\\
& =\KW(\lambda+\rho-\rho_{\lambda}+\rho'_{\lambda}, S_s).
\end{array}$$

For the $\osp$-case one has $\Sigma=\Sigma_L$, so $\rho_{\lambda}=\rho'_{\lambda}$;
this gives (i). For $\fgl$-case notice that
$\Sigma_L$ is obtained from $\Sigma$ by the chain of odd reflections
with respect to the roots in $\Delta(\fg_{\lambda})$;
this gives $\rho_L-\rho=\rho'_{\lambda}-\rho_{\lambda}$ and establishes
(ii).
\end{proof}

\subsection{The  $\osp$-case}\label{Elambdaosp}
Consider the case  $\fg=\osp(M|2n)$ ($M=2m$ or $M=2m+1$).
Theorems 3, 4 and Remark after Thm. 3 of~\cite{GS} imply
that for  $\lambda\in\Lambda^+_{m|n}$ one has
$$\ch L(\lambda)=\sum_{\mu\in\Lambda^+_{m|n}} (-1)^{||\lambda||-||\mu||}
d^{\lambda,\mu}_<\mathcal{E}_{\mu},$$
where
$d^{\lambda,\mu}_<$ is the number of ``increasing paths'' from $\diag(\mu)$ to $\diag(\lambda)$ in the graph $D_{\fg}$ described in~\cite{GS}, Sect. 11; we  will recall some properties of this graph  below.

\subsubsection{Properties of $D_{\fg}$}\label{propertiedDosp}
The connected components
of $D_{\fg}$ correspond to the dominant central characters, so
for each component $D_{\fg}^{\chi}$ we can define $t\in\{0,1,2\}$ via the corresponding central
character. The map $\lambda\to\howl(\lambda)$ 
gives an isomorphism $D_{\fg}^{\chi}\iso D^{\chi_0}_{\osp(2k+t|2k)}$ for
$k:=\at(\chi), t:=t(\chi)$; the map $\tau$ induces an isomorphism $D^{\chi_0}_{\osp(2k+1|2k)}
\iso D^{\chi_0}_{\osp(2k+t|2k)}$.

Assume that $\diag(\mu)$ is a predecessor of $\diag(\lambda)$ in $D_{\fg}$.
From~\cite{GS}, Sect. 11, we conclude that for the cases $t=0,2$ 
 $\diag(\lambda)$  is obtained from $\diag(\mu)$ 
 by moving several symbols $\times$ to the right; moreover, if $\diag(\lambda)$ has a sign, then $\diag(\mu)$ has the same sign. Using the isomorphism
 induced by $\tau$, we conclude that for $t=1$, $\diag(\lambda)$  is obtained from $\diag(\mu)$ 
 by moving several symbols $\times$ to the right or by changing the sign $-$ to the sign $+$. This implies
$$\lambda>\mu,\ \ \ \howl(\lambda)>\howl(\mu),\ \ ||\lambda||\geq ||\mu||,\ \
 ||\lambda||_{gr}>||\mu||_{gr},\ 
\ \  \tail(\mu)\geq \tail(\lambda)$$
and that if $\lambda$ is stable, then $\mu$ is stable. Moreover,  
\begin{equation}\label{uprise}
\lambda-\mu\in \left\{\begin{array}{lcl}
\frac{1}{2}\mathbb{N}\Pi_0+\mathbb{N}(\delta_n-\vareps_m)+\mathbb{N}(\delta_n+\vareps_m)
 & & \text{ for } t=0\\
\frac{1}{2}\mathbb{N}\Pi_0& & \text{ for } t=1,2.
\end{array}
\right.
\end{equation}

By above, $D_{\fg}$ is $\mathbb{N}$-graded
with respect to $||\ ||_{gr}$ (if $\diag(\mu)$ is a predecessor of $\diag(\lambda)$,
then $||\mu||_{gr}<||\lambda||_{gr}$).
In particular,  each vertex in $D_{\fg}$ has finitely many predecessors.

The map $\tau$ described in~\ref{tau} gives an isomorphism of 
the graph $D_{\osp(2m+1|2n)}$ and the subgraph of
$D_{\osp(2m+2|2n)}$ which correspond to the union of connected components
 with $t=2$.

\subsubsection{}
We conclude that for $\osp(M|N)$ we have
$$\begin{array}{l}
d^{\lambda,\lambda}_<=1,\ \ \ \   d^{\lambda,\mu}_<=d^{\howl\lambda,\howl\mu}_<\geq 0,\\
d^{\lambda,\mu}_<\not=0 \ \Longrightarrow\ 
\chi_{\lambda}=\chi_{\mu},\ \ ||\mu||\leq ||\lambda||,\ \ 
 \tail(\mu)\geq \tail(\lambda),\ 
 ||\mu||_D<||\lambda||_D.
\end{array}$$
Moreover the sum in the right-hand side of the character formula is finite
and the  terms $\{\cE_{\lambda}\}_{\lambda\in\Lambda^+_{m|n}}$ form
a basis in the character ring of $\cF$.

Using~\ref{plambdanu} we obtain
$$(-1)^{p(\lambda)}\ch L(\lambda)=\sum_{\mu\in\Lambda^+_{m|n}} (-1)^{p(\mu)}
d^{\lambda,\mu}_<\mathcal{E}_{\mu}\ \ \text{ if $\lambda$ is stable and }\ 
t(\lambda)\not=2.$$

\subsubsection{Remark} The $\mathfrak{q}(n)$-case can be treated with the same methods. The character of $L(\lambda)$ can be written as a finite sum in the Euler characters where the coefficients are again given by the number of increasing paths in a certain bimarked graph. As for the $\mathfrak{osp}$-case the finiteness is automatic since each vertex $\lambda$ in this graph has a finite number of predecessors. However Y. Su and R. B.~Zhang already obtained in \cite{SZq} a similar character formula based on earlier work of J. Brundan \cite{B-q}, so that we have refrained from including this case.

\subsection{The Poincar\'e polynomials in the $\fgl$-case}\label{Poincaregl}
Let $k$ be the degree of atypicality of $\lambda+\rho$. 
For $i=0,\ldots,k-1$ the polynomials $K^{\lambda,\mu}_{\fp^{(i+1)},\fp^{(i)}}(z)$
were computed in \cite{Sselecta} (see also Cor. 3.8 in \cite{MS}).
We will recall the  diagramic interpretation (which was provided by Serganova in one of 
her wonderful talks) in~\ref{diagramic}.

Let $\fg:=\fgl(m|n)$. We identify a weight $\lambda\in\Lambda^+_{m|n}$
and the diagram of $\lambda+\rho$, which we denote by $\diag(\lambda)$.

\subsubsection{Arc diagrams}
Take a weight  $\nu\in\Lambda^+_{m|n}$. Denote by $\diag(\nu)$ 
the weight diagram of $\nu+\rho$. The arc diagram
$\Arc(\nu)$ consists of the  arcs $\arc(a;a')$,
where $a<a'$ and $\diag(\nu)$ has $\times$ (resp., $\circ$) at the position 
$a$ (resp., $a'$). These arcs satisfy
the following properties: 

\begin{itemize}
\item
each symbol $\times$ is connected by an arc
to exactly one symbol $\circ$;

\item
 each symbol $\circ$ is connected to at most
one symbol $\times$;

\item 
the arcs do not intersect;

\item
 each symbols $\circ$ situated  under an arc
is  connected to a symbol $\times$.
\end{itemize}

The arc diagram
$\Arc(\nu)$ is unique and can be constructed
in the following way:
we pass from right to left through the weight diagram and connect each of the finitely many crosses $\times$ with the next empty symbol to the right by an arc (ignoring core symbols).

\subsubsection{Example}

\begin{center}
%\medskip
 
 \scalebox{0.7}{
\begin{tikzpicture}
 %\draw (-1,0) -- (7,0);
\foreach \x in {} %vee
     \draw[very thick] (\x-.1, .1) -- (\x,-0.1) -- (\x +.1, .1);
\foreach \x in {} %wedge
     \draw[very thick] (\x-.1, -.1) -- (\x,0.1) -- (\x +.1, -.1);
\foreach \x in {-1,0,2,5,6,11} %cross
     \draw[very thick] (\x-.1, .1) -- (\x +.1, -.1) (\x-.1, -.1) -- (\x +.1, .1);
%\foreach \x in {1,3,4,7,8,9,10,12}  \draw[semithick] \circ; %circle
     %\draw[very thick]  node at (0,0) [fill=white,draw,circle,inner sep=0pt,minimum size=6pt]{};
     \draw[very thick]  node at (1,0) [fill=white,draw,circle,inner sep=0pt,minimum size=6pt]{};
     \draw[very thick]  node at (3,0) [fill=white,draw,circle,inner sep=0pt,minimum size=6pt]{};
     \draw[very thick]  node at (4,0) [fill=white,draw,circle,inner sep=0pt,minimum size=6pt]{};
     \draw[very thick]  node at (7,0) [fill=white,draw,circle,inner sep=0pt,minimum size=6pt]{};
     \draw[very thick]  node at (8,0) [fill=white,draw,circle,inner sep=0pt,minimum size=6pt]{};
     \draw[very thick]  node at (9,0) [fill=white,draw,circle,inner sep=0pt,minimum size=6pt]{};
     \draw[very thick]  node at (10,0) [fill=white,draw,circle,inner sep=0pt,minimum size=6pt]{};
     \draw[very thick]  node at (12,0) [fill=white,draw,circle,inner sep=0pt,minimum size=6pt]{};
%\foreach \x in {0} %cross
     %\draw[very thick] (\x-.1, +0.8) -- (\x +.1, +0.6) (\x-.1, +0.6) -- (\x +.1, +0.8);

\draw (-1,-0.5) node {-1};
\draw (0,-0.5) node {0};
\draw (1,-0.5) node {1};
\draw (2,-0.5) node {2};
\draw (3,-0.5) node {3};
\draw (4,-0.5) node {4};
\draw (5,-0.5) node {5};
\draw (6,-0.5) node {6};
\draw (7,-0.5) node {7};
\draw (8,-0.5) node {8};
\draw (9,-0.5) node {9};
\draw (10,-0.5) node {10};
\draw (11,-0.5) node {11};
\draw (12,-0.5) node {12};

% node[pos=(0, -0,5)]{0};

%%caps,cups
\draw[very thick] [-,black,out=90, in=90](0,0.2) to (1,0.2);
\draw[very thick] [-,black,out=90, in=90](2,0.2) to (3,0.2);
\draw[very thick] [-,black,out=90, in=90](11,0.2) to (12,0.2);
\draw[very thick] [-,black,out=90, in=90](6,0.2) to (7,0.2);
\draw[very thick] [-,black,out=90, in=90](5,0.2) to (8,0.2);
\draw[very thick] [-,black,out=90, in=90](-1,+0.2) to (4,0.2);
%\draw[very thick] [-,black,out=90, in=90](0,+0.9) to (9,0.2);

%\foreach \x in {} \draw + at (-1,0);

\end{tikzpicture} }
\medskip

\text{Arc diagram for $\times \times\circ \times \circ\circ\times\times \circ \circ \circ \circ\times \circ$.}
\end{center}

\subsubsection{Definition}\label{diagramic}
For a weight diagram $f$ we denote by $f^u_a$
the weight diagram obtained from $f$ by interchanging the symbols at the positions $u$ and $a$.

Let $\lambda,\nu\in\Lambda^+_{m|n}$
be such that $(\lambda)=\diag(\nu)_a^u$.
We say that $\lambda$ is obtained from $\nu$ by a {\em move}
if $\diag(\nu)$ has $\times$ at the position $a$, $\circ$ at the position
$u$ and  $u$ lies under the arc originated at $a$, that is 
$\Arc(\nu)$ contains $\arc(a;a')$  with $a<u\leq a'$.
 For such a move we define the weight as the number of arcs
in $\Arc(\nu)$ which are ``strictly above'' $u$
(for instance, if $u=a'$, then the move has zero weight).

Observe that if  $\lambda$ can be obtained from $\nu$ by a move
as above, then such move is unique.
 In this case we set
$b'(\nu;\lambda):=u$. Note that $\diag(\lambda)$ has $\times$ at the $u$-th
position; we set $b(\nu;\lambda):=i+1$, where $i$ is the number
of the symbols $\times$ with the coordinates less than $u$
in $\diag(\lambda)$.

We will consider only the case of stable $\lambda$, so the symbols
$\times$ in $\diag(\lambda)$ preceed the core symbols.
If $\lambda$ is obtained from $\nu$ by a move, then $\nu$ is stable.
A  move is called a {\em non-tail move} if 
$b(\nu;\lambda)>\tail(\lambda)$.

 \subsubsection{Example}

 Take $\nu$ with $\diag(\nu)=0\times\times\times$ with arc diagram

\begin{center}
%\medskip
 
 \scalebox{0.7}{
\begin{tikzpicture}
 %\draw (-1,0) -- (7,0);
\foreach \x in {} %vee
     \draw[very thick] (\x-.1, .1) -- (\x,-0.1) -- (\x +.1, .1);
\foreach \x in {} %wedge
     \draw[very thick] (\x-.1, -.1) -- (\x,0.1) -- (\x +.1, -.1);
\foreach \x in {1,2,3} %cross
     \draw[very thick] (\x-.1, .1) -- (\x +.1, -.1) (\x-.1, -.1) -- (\x +.1, .1);
%\foreach \x in {1,3,4,7,8,9,10,12}  \draw[semithick] \circ; %circle
     %\draw[very thick]  node at (0,0) [fill=white,draw,circle,inner sep=0pt,minimum size=6pt]{};
     \draw[very thick]  node at (0,0) [fill=white,draw,circle,inner sep=0pt,minimum size=6pt]{};
     \draw[very thick]  node at (4,0) [fill=white,draw,circle,inner sep=0pt,minimum size=6pt]{};
     \draw[very thick]  node at (5,0) [fill=white,draw,circle,inner sep=0pt,minimum size=6pt]{};
     \draw[very thick]  node at (6,0) [fill=white,draw,circle,inner sep=0pt,minimum size=6pt]{};
     \draw[very thick]  node at (7,0) [fill=white,draw,circle,inner sep=0pt,minimum size=6pt]{};
 %    \draw[very thick]  node at (9,0) [fill=white,draw,circle,inner sep=0pt,minimum size=6pt]{};
    % \draw[very thick]  node at (10,0) [fill=white,draw,circle,inner %sep=0pt,minimum size=6pt]{};
 %    \draw[very thick]  node at (12,0) [fill=white,draw,circle,inner sep=0pt,minimum size=6pt]{};
%\foreach \x in {0} %cross
     %\draw[very thick] (\x-.1, +0.8) -- (\x +.1, +0.6) (\x-.1, +0.6) -- (\x +.1, +0.8);

%\draw (-1,-0.5) node {-1};
\draw (0,-0.5) node {0};
\draw (1,-0.5) node {1};
\draw (2,-0.5) node {2};
\draw (3,-0.5) node {3};
\draw (4,-0.5) node {4};
\draw (5,-0.5) node {5};
\draw (6,-0.5) node {6};
\draw (7,-0.5) node {7};
%\draw (8,-0.5) node {8};
%\draw (9,-0.5) node {9};
%\draw (10,-0.5) node {10};
%\draw (11,-0.5) node {11};
%\draw (12,-0.5) node {12};

% node[pos=(0, -0,5)]{0};

%%caps,cups
\draw[very thick] [-,black,out=90, in=90](1,0.2) to (6,0.2);
\draw[very thick] [-,black,out=90, in=90](2,0.2) to (5,0.2);
\draw[very thick] [-,black,out=90, in=90](3,0.2) to (4,0.2);
%\draw[very thick] [-,black,out=90, in=90](6,0.2) to (7,0.2);
%\draw[very thick] [-,black,out=90, in=90](5,0.2) to (8,0.2);
%\draw[very thick] [-,black,out=90, in=90](-1,+0.2) to (4,0.2);
%\draw[very thick] [-,black,out=90, in=90](0,+0.9) to (9,0.2);

%\foreach \x in {} \draw + at (-1,0);

\end{tikzpicture} }
\medskip

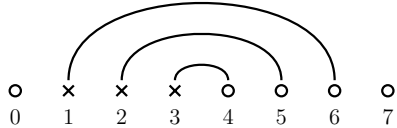
\text{Arc diagram for $0 \times \times \times \circ$}
\end{center}

 There are $6$ weights $\lambda_1,\ldots,\lambda_6$ which can be obtained from $\nu$; in all cases $b(\nu;\lambda)=3$. For instance,
$\lambda_1$ with
$\diag(\lambda_1)=1\times\times\times$  can be obtained from
$\nu$ by a move of weight $2$ with $b'(\nu;\lambda_1)=4$.
 Similarly, $\lambda_2$ with
$\diag(\lambda_2)=1\times\times\circ\times$  can be obtained from
$\nu$ by a move of weight $1$ with $b'(\nu;\lambda_1)=5$.
 Another example is $\lambda_3$ with
$\diag(\lambda_3)=0\times\circ\times\times$  can be obtained from
$\nu$ by a move of weight $1$ with $b'(\nu;\lambda_1)=4$.

From the weight $\lambda_2$ we can obtain $\mu$  with
$\diag(\mu)=1\times\circ\times\times$ 
by a move of weight $0$ with $b'(\lambda_2;\mu)=4$
and  $b(\lambda_2;\mu)=2$.

Among the above examples only the first move is a tail move.

\subsubsection{}\label{identific-graph}
Let $\lambda$ be a stable weight and
$k:=\at(\lambda+\rho)$.

For $i=1,\ldots,k$ the results of~\cite{Sselecta} give
$$\begin{array}{ll}
K^{\lambda,\lambda}_{\fp^{(i)},\fp^{(i-1)}}(z)=1 \\
K^{\lambda,\mu}_{\fp^{(i)},\fp^{(i-1)}}(z)=z^s\ & \text{ for } \mu\not=\lambda\end{array}$$
if   $\mu$ is obtained from $\lambda$ by a  move of weight $s$
and $b(\lambda;\mu)=i$. In all other cases $K^{\lambda,\lambda}_{\fp^{(i)},\fp^{(i-1)}}(z)=0$. Moreover,
$K^{\lambda,\mu}_{\fg,\fp^{(k)}}(z)=\delta_{\lambda,\mu}$
(see,  Lemma 5 (a ``Typical Lemma'') in~\cite{GS}) if $\fp^{(k)}\not=\fg$.

\subsubsection{}
\begin{lem}{lemAA}
Take a stable weight $\lambda\in\Lambda^+_{m|n}$. Assume that
 $\lambda$ is obtained from $\nu$ 
by a  move of weight $w$.
\begin{enumerate}
\item
Then  $\nu$ is stable, $\core(\lambda)=\core(\nu)$ and
%\begin{equation}\label{legal}
$$\lambda>\nu,\ \ \ \  ||\lambda||-
||\nu||-(w+1)\in 2\mathbb{Z},\ \ \ \tail(\nu)\leq b(\nu;\lambda)$$
\item
If the move is a non-tail move, then
%\begin{equation}\label{nontail}
$$||\lambda||_{gr}>||\nu||_{gr},\ \
\tail(\lambda)\leq\tail(\nu).$$
\end{enumerate}
\end{lem}
\begin{proof}
The first assertion  immediately follows from the formula
$\diag(\lambda)=\diag(\nu)_a^u$ for $a<u$.
Consider the case of a non-tail move, i.e. 
$b(\nu;\lambda)\leq\tail(\lambda)$.

 Since $\lambda,\nu$ are stable,
their diagrams start from the subdiagrams $\times\ldots\times$
containing, respectively, $\tail(\lambda)$ and $\tail(\nu)$ symbols $\times$.
Let $A_{\lambda}$ (resp., $A_{\nu}$) be  the coordinates of symbols $\times$ in these subdiagrams (one has 
$A(\lambda)=\{u_{\lambda}+i\}_{i=0}^{\tail(\lambda)-1}$, where
$u_{\lambda}$ is the minimal coordinate of the non-empty symbol
in $\diag(\lambda)$).
The inequality $b(\nu;\lambda)\leq\tail(\lambda)$ means that
$u>\max A_{\lambda}$. This gives $A_{\lambda}\subset A_{\nu}$
and implies (ii).
\end{proof}

\subsection{Graph $D_{\fg}$}\label{graphDg}
Let $D_{\fg}$ be a graph with the set of vertices
enumerated by  $\Lambda^+_{m|n}$. We identify the weight $\lambda$
with $\diag(\lambda)$.
We join $f,g$ by the edge $f {\longrightarrow}g$
if $\core(f)=\core(g)$ and
$\howl(g)$ is obtained from $\howl(f)$ by a non-tail move described in~\ref{Poincaregl}.

Recall that  $f {\longrightarrow}g$ implies that $\howl(g)$ is obtained from
$\howl(f)$ by moving $\times$ from a position $a$ to an empty position
$u>a$. We mark each edge by the corresponding $u$.

\subsubsection{Subgraphs $D^{\chi}_{\fg}$}\label{conn}
Clearly, if $\lambda$ and $\nu$ lie in the same connected component
of $D_{\fg}$, then $\chi_{\lambda}=\chi_{\nu}$.
Denote by $D^{\chi}_{\fg}$ the full subgraph with the vertices 
$\lambda$ such that $\chi_{\lambda}=\chi$. 
If $\chi$ has atypicality $s$, then
 the map $f\mapsto \howl(f)$  gives an isomorphism
of $D^{\chi}_{\fg}$ and $D^{\chi_0}_{\fgl(s|s)}$.  

If $\at \lambda\leq 1$, then the corresponding vertex is isolated.
It is not hard to see that $D^{\chi}_{\fg}$ is connected for $\at\chi>1$.

\subsubsection{}
\begin{cor}{corAA}
\begin{enumerate}

\item
Let $\nu$ be a predecessor of $\lambda$. Then 
$\diag(\lambda)$ is obtained from $\diag(\nu)$ by moving some symbols
$\times$ to the right. In particular, $\core(\lambda)=\core(\nu)$,
$\lambda>\nu$ and
$$\howl(\lambda)>\howl(\nu),\ \ ||\lambda||>||\nu||,\ \
 ||\lambda||_{gr}>||\nu||_{gr},\ 
\ \ \fg_{\howl(\lambda)}\subset\fg_{\howl(\nu)}.$$

\item
Any vertex in $D_{\fg}$ has a finite number of predecessors.
\end{enumerate}
\end{cor}
\begin{proof}
Let $\nu$ be a predecessor of $\lambda$. Then $\diag(\lambda)$ is obtained from $\diag(\nu)$ by moving several symbols $\times$ to the right; this gives $\lambda>\nu$ and $||\lambda||>||\nu||$; the rest of the formulas
in (i) follow from~\Lem{lemAA}.
By above, for (ii) it is enough to consider the case $\fg=\fgl(s|s)$
and $\core(\lambda)=\emptyset$.
By~\Lem{lemAA}, the set of  predecessors  of ${\lambda}$ in $D_{\fg}$
lie in the following set
$$\{\nu\in\Lambda^+_{s|s}|\ \  
\core(\nu)=\emptyset,\ \ ||\nu||_{gr}<||\lambda||_{gr},\  A_{\lambda}\subset A_{\nu}\},$$
where $A_{\lambda}, A_{\nu}$ as in the proof of~\Lem{lemAA}.
In particular, the coordinates of all non-empty symbols in $\diag(\nu)$ 
lie between $u_{\lambda}-s$ and $u_{\lambda}-s+||\lambda||_{gr}$, where $u_{\lambda}$ is the minimal coordinate 
of the non-empty symbol in $\diag(\lambda)$. This gives (ii).
\end{proof}

\subsubsection{}
We call a path in $D_{\fg}$ {\em increasing} (resp., {\em decreasing}) if 
the marks strictly increase (resp., decrease) along the path.

\subsubsection{Example}
If $\lambda$ is a Kostant weight, then $\mathcal{E}_{\lambda}=e^{\lambda}$.

For $\fgl(n|n)$ the adjoint representation $\mathbf{Ad}$ has a three step Loewy filtration \[ \mathbf{Ad} = \begin{pmatrix} \mathbb{C} \\ \Pi(L(\epsilon_1 - \delta_n)) \\ \mathbb{C} \end{pmatrix}\] The middle term with highest weight $\lambda=\vareps_1-\delta_n$ corresponds to the diagram
$$-n\underbrace{\times\times...\times}_{n-1\text{ times}}\circ \times;$$
this diagram is connected to the Kostant weights
$$-n\underbrace{\times\times...\times}_{n\text{ times}}\ \ \ \ \ \ \ \
-n-1\underbrace{\times\times...\times}_{n\text{ times}}
$$
the corresponding weights $0$ and $\mu=\displaystyle\sum_{i=1}^n (\delta_i-\vareps_i)$.
This gives
$$\ch L(\vareps_1-\delta_3)=\mathcal{E}_{ \vareps_1-\delta_n}-1-e^{\mu},\ \ \
\sch L(\vareps_1-\delta_3)=\mathcal{E}^-_{ \vareps_1-\delta_n}+1+e^{\mu}.
$$
Notice that $\sdim(\mathbf{Ad}) = 0$, hence $\sdim L(\vareps_1-\delta_n)=2$.

\subsubsection{Examples}\label{exaDg}
For $\fgl(1|1)$ the graph $D_{\fg}$ does not have edges.

For $\fgl(2|2)$  the vertices $\lambda$ with $\core(\lambda)\not=\emptyset$
  are isolated;
the vertices with $\core(\lambda)=\emptyset$ form a connected component
$D^{\chi_0}_{\fg}$  of the following form

$$\begin{array}{ccc}
&\searrow\ \text{\tiny{3}} & \\
0\times\times & \xrightarrow{3} & 0\times\circ\times  \xrightarrow{4} 0\times\circ\circ\times \
\xrightarrow{5} 0\times\circ\circ\circ\times
\cdots\\
&\searrow\ \text{\tiny{4}} & \\
1\times\times &\xrightarrow{4} & 1\times\circ\times  \xrightarrow{5} 1\times\circ\circ\times \
\xrightarrow{6} 1\times\circ\circ\circ\times
\cdots\\
\end{array}$$
The left column corresponds to the Kostant weights ($||\lambda||_{gr}=0$);
the next column to the $\lambda$s with $||\lambda||_{gr}=1$ and so on.

\subsubsection{}
\begin{prop}{Kostant}
(i) Each vertex is connected to a Kostant weight by an increasing path.

(ii) The Kostant weights are the sources of the graph $D_{\fg}$.
\end{prop}
\begin{proof}
For (i)  take $f_0$ with $||f_0||_{gr}\not=0$ and let $u$ be the coordinate of
the rightmost symbol $\times$ (which is not in the tail, since $f$ is not a Kostant weight) and $a$ be maximal such that
$a<u$ and $f(a)=\circ$. Set $f_1=(f)^u_a$. Then $f=(f_1)^u_a$
and $f$ is obtained from $f_1$ by a non-tail move with $b'(f_1,f)=u$.
If $||f_1||_{gr}\not=0$, we 
 construct $f_2$ by the same rule.
Continuing this process we obtain an increasing path
$$f_r\to f_{r-1}\to \ldots \to f_1$$
with $||f_{i+1}||_{gr}<||f_i||_{gr}$;
thus  for some $r$ one has $||f_r||_{gr}=0$. This gives (i). 
Now (ii) follows from (i) and the inequality
$||\lambda||_{gr}>||\nu||_{gr}$ if $\nu$ is a predecessor of
$\lambda$. 
 \end{proof}

\subsection{Character formula for $\fgl(m|n)$}\label{Taa+1}
Take $\fg=\fgl(m|n)$ with a distinguished base $\Sigma$. For $\mu,\lambda\in\Lambda^+_{m|n}$ denote by $\cP^{>}(\mu,\lambda)$  the set of decreasing paths from $\mu$ to $\lambda$  and  by
$d^{\lambda,\mu}_<$  the number of increasing paths
from $\mu$ to $\lambda$ in the graph $D_{\fg}$. Set
$$d'_{\lambda,\mu}:=(-1)^{||\lambda||-||\mu||} \sum_{P\in \cP^{>}(\mu,\lambda)} (-1)^{length P}.$$
By above, $d^{\lambda,\mu}_<=d^{\howl(\lambda),\howl(\mu)}_<$
and $d'_{\lambda,\mu}=d'_{\howl(\lambda),\howl(\mu)}$.

\subsubsection{}
Let $\Lambda^{\chi}_{st}$ be the set of stable weights in $\Lambda^{\chi}$. In the next section we will prove the following formulas for $\lambda\in\Lambda^{\chi}_{st}$:
\begin{equation}\label{EL}
\begin{array}{l}
\cE_{\lambda}=\displaystyle\sum_{\mu\in\Lambda^{\chi}} 
d'_{\lambda,\mu} \ch L(\mu),\\
\ch L(\lambda)=\displaystyle\sum_{\mu\in\Lambda^{\chi}_{st}} (-1)^{||\lambda||-||\mu||} 
d^{\lambda,\mu}_<\mathcal{E}_{\mu}.
\end{array}\end{equation}

Notice that, by~\Cor{corAA},  the right-hand
sides of the above  formulas have finite number of non-zero summands.

\subsubsection{}
For a non-stable weight $\lambda$ we introduce $\cE_{\lambda}$ by the first formula
in~(\ref{EL}), i.e. 
\begin{equation}\label{Ela}
\cE_{\lambda}:=\sum_{\mu\in\Lambda^{\chi}} 
d'_{\lambda,\mu}  \ch L(\mu).\end{equation}
If $\lambda$ is stable and $\mu$ is not stable, then $d^{\lambda,\mu}_<=0$ 
(by~\Cor{corAA}). Therefore the second formula in~(\ref{EL}) can be rewritten as
\begin{equation}\label{EL1}
\ch L(\lambda)=\sum_{\mu\in\Lambda^{\chi}} (-1)^{||\lambda||-||\mu||} 
d^{\lambda,\mu}_<\mathcal{E}_{\mu}=\sum_{\mu\in\Lambda^+_{m|n}} (-1)^{||\lambda||-||\mu||} 
d^{\lambda,\mu}_<\mathcal{E}_{\mu}\end{equation}
if $\lambda$ is stable.
By~(\ref{EL}) the matrices $(d'_{\lambda,\nu})=(d'_{\howl(\lambda),\howl(\nu)})$
 and $((-1)^{||\lambda||-||\mu||} 
d^{\lambda,\mu}_<)$ are mutually inverse. Using~(\ref{Ela}) we deduce~(\ref{EL1})
for  each $\lambda\in\Lambda^+_{m|n}$.

\subsubsection{}
We retain notation of~\ref{ThetaV}. Fix a central character $\chi$
and denote by $\Fin^{\chi}$ the full subcategory of $\Fin$ 
of the modules with the central character $\chi$.
We will consider translation functors $T^{V}_{\chi,\chi'}$ for special cases when these functors are  equivalence of categories and
$V$ is either  the standard representation or its dual. These functors can be desribed as follows.

Recall that for  a weight diagram $f$ denote by $(f)^{a+1}_a$ the diagram
$f'$ obtained from $f$ by interchanging the symbols in the positions $a$ and $a+1$.
We denote by $T_{a,a+1}$  the corresponding operations on
${\Lambda}^{\geq}_{m|n}$ and on the central characters: 
$T_{a,a+1}(\nu)=\nu'$ such that
$\diag(\nu')=T_{a,a+1}(\diag(\nu))$ and $T_{a,a+1}(\chi)=\chi'$
such that $\core(\chi')=T_{a,a+1}(\core(\chi))$.

For $V=V_{st}, V_{st}^*$ 
the translation functor 
$T^{V}_{\chi,\chi'}:\Fin^{\chi}\iso\Fin^{\chi'}$ is an equivalence of categories if $\chi'=T_{a,a+1}(\chi)$  for some $a$ and exactly one of the positions
$a,a+1$ in $\core(\chi)$ is empty (so
for $\lambda\in\Lambda^{\chi}$ exactly one of the positions
$a,a+1$ in $\diag(\lambda)$ is occupied by a core symbol and $T_{a,a+1}$
interchanges this core symbol with $\circ$ or $\times$ respectively).

One has
$$T^{V}_{\chi,\chi'}(L(\lambda))=L(T_{a,a+1}(\lambda)).$$
Note that $\howl(\lambda)=\howl(T_{a,a+1}(\lambda))$.
By a slight abuse of notation, we  denote such functor by $T_{a,a+1}$.

\subsubsection{}
\begin{lem}{lemEla}
For $\lambda':=T_{a,a+1}(\lambda)$ one has
$$Re^{\rho}\cE_{\lambda}=\Theta_{\chi,\chi'}(Re^{\rho}\cE_{\lambda'})$$
where 
 $\Theta_{\chi,\chi'}:\cR_{\Sigma}\to\cR_{\Sigma}$ is  the  ring homomorphism
corresponding to  $T_{a,a+1}$ (see~\ref{ThetaV}).
\end{lem}
\begin{proof}
For each $\mu\in \Lambda^{\chi}$ set $\mu':=T_{a,a+1}(\mu)$.
By~(\ref{Ela})
$Re^{\rho}\cE_{\lambda}=\sum_{\mu'\in\Lambda^{\chi'}}  d'_{\lambda,\mu}
 Re^{\rho}\ch L(\mu)$.
Using~\ref{ThetaV} we get
$$\Theta_{\chi,\chi'}(Re^{\rho}\cE_{\lambda})=\sum_{\mu\in\Lambda^{\chi}} d'_{\lambda,\mu}
 Re^{\rho}\ch L(\mu').$$
Since $\howl(\mu')=\howl(\mu)$ one has $d'_{\lambda,\mu}=d'_{\lambda',\mu'}$, so
$$Re^{\rho}\cE_{\lambda'}=\sum_{\mu'\in\Lambda^{\chi'}}  d'_{\lambda',\mu'}
 Re^{\rho}\ch L(\mu')=\sum_{\mu'\in\Lambda^{\chi'}}  d'_{\lambda,\mu}
 Re^{\rho}\ch L(\mu')$$
as required.
\end{proof}

\subsection{Another form of the character formula}\label{anotheryear}
In the $\osp$-case we retain the notation of~\Prop{propE} and set 
$$\KW(\nu):=\KW(\nu+\rho, S_{\tail(\nu)}),\ \ j(\nu):=j_{\tail(\nu)}.$$
For $\fgl$-case we will introduce $\KW(\nu)$ in~\ref{y0f}
and set $j(\nu):=\tail(\nu)!$.

\subsubsection{}
\begin{cor}{corchKW} 
$$Re^{\rho}\ch L(\lambda)=\sum_{\mu\in\Lambda^+_{m|n}}(-1)^{||\lambda||-||\mu||}
 \frac{d^{\lambda,\mu}_<}{j(\mu)}\KW(\mu).$$
\end{cor}
\begin{proof}
Combining~\Prop{propE} and~\ref{Elambdaosp} we obtain the assertion for the
$\osp$-case. 
For the $\fgl$-case we combine~(\ref{EL1}), \Lem{lemEla} and~\Cor{corEla} (ii).
\end{proof}

\subsubsection{Remark}
Setting $\KW(L(\nu)):=j(\nu)^{-1}\KW(\nu)$ we obtain the formula~(\ref{GSformula}).

\subsubsection{Remark}The graph $D_{\fg}$
is an oriented graph; this graph
does not have multi-edges for the $\fgl$-case and for the $\osp$-case
with $t=1,2$; for $t=0$ the graph has double edges.

\subsection{Highest weights of $L$ with respect to different bases}\label{hwt}
Fix any base  $\tilde{\Sigma}$ compatible with $\Pi_0$ (i.e.
 $\Delta^+(\tilde{\Sigma})\cap\Delta_0=\Delta^+(\Pi_0)$) 
and denote the
Weyl vector by $\tilde{\rho}$. For a simple finite-dimensional
 module $L$ denote by $\hwt L$ the ``$\rho$-twisted highest
weight of $L$'' i.e. $\hwt_{\tilde{\Sigma}} L=\nu+\tilde{\rho}$, where
$\nu$ is the highest weight of $L$ with respect to $\tilde{\Sigma}$.
If  $\beta\in\tilde{\Sigma}$ is isotropic
and $r_{\beta}$ is the corresponding odd reflection, then
$\hwt_{r_{\beta}\tilde{\Sigma}} L=
\hwt_{\tilde{\Sigma}} L$ if $(\hwt_{\tilde{\Sigma}} L|\beta)\not=0$ and
$\hwt_{r_{\beta}\tilde{\Sigma}} L=
\hwt_{\tilde{\Sigma}} L+\beta$ otherwise. Using this procedure one can compute 
$\hwt_{\tilde{\Sigma}} L(\lambda)$ recursively. 
The character formula in~\Cor{corchKW} allows to give the following formula for 
$\hwt_{\tilde{\Sigma}} L(\lambda)$ for $t(\lambda)=1,2$.

\subsubsection{}
\begin{cor}{cormaxsupp}
  Consider the partial order
$\tilde{>}$ on $\fh^*$ given by $\nu \tilde{>}\mu$ if $\nu-\mu\in\mathbb{N}\tilde{\Sigma}$. View $\KW(\lambda)$ as an element of $\cR_{\tilde{\Sigma}}$.

\begin{enumerate}
\item
In the $\osp$-case with 
$t(\lambda)=1,2$ the weight $\hwt_{\tilde{\Sigma}} L(\lambda)$ is a unique maximal element in
$\supp \KW(\lambda)$ with respect to the partial order $\tilde{>}$.

\item
In the $\osp$-case with 
$t(\lambda)=0$ the same holds if $\delta_n\pm\vareps_m\in\tilde{\Sigma}$.

\item In the $\fgl$-case $\lambda+\rho$ is a unique maximal element in
$\supp \KW(\lambda)$ with respect to the partial order $>$.
\end{enumerate}
\end{cor}
\begin{proof}
The assertions (i), (ii) follow by induction on $||\lambda||_{gr}$ if we 
combine~\Cor{corchKW} 
with~(\ref{uprise}).
Similarly, (iii) follows by induction on $||\lambda||_{gr}$ from~\Cor{corchKW} 
and the fact that $d^{\lambda,\mu}_{<}\not=0$ implies $\mu<\lambda$.
\end{proof}

\section{Proof of the formulas~(\ref{EL})}\label{sec:character}

The proof of~(\ref{EL}) follows the plan explained in~\cite{Gdex},  Section 3.
In~\ref{markedgraphs}, \ref{graphs} below we recall the main constructions of~\cite{Gdex}.

\subsection{Marked graphs}\label{markedgraphs}
Consider a directed graph $(V,E)$ where $V$ and $E$ are at most countable, where
the set of edges $E$ is equipped by two functions:
$b: E\to\mathbb{Z}$ and a function  $\kappa$ from $E$ to a commutative ring.

We say that $\iota: V\to \mathbb{N}$  (resp., $\iota: V\to \mathbb{Z}$)
defines a $\mathbb{N}$-{\em grading} (resp., $\mathbb{Z}$-grading) 
on this graph if
for each edge $\nu\overset{e}{\longrightarrow} \lambda$ 
one has $\iota(\nu)<\iota(\lambda)$.

\subsubsection{}\label{decrincr}
For a path
$$P:=\nu_1\overset{e_1}{\longrightarrow }\nu_2\overset{e_2}{\longrightarrow }\nu_3 \ldots
\overset{e_s}{\longrightarrow }\nu_{s+1}$$
we define
$$length(P):=s,\ \ \kappa(P):=\prod_{i=1}^s \kappa(e_i).$$

We call the path $P$
 {\em decreasing} (resp., {\em increasing}) if $b(e_1)>b(e_2)>\ldots >b(e_s)$
(resp., $b(e_1)<\ldots <b(e_s)$). 
We consider a path $P=\nu$ (with one vertex and zero edges) as a decreasing/increasing path  of zero length with  $\kappa(P)=1$.

\subsubsection{}\label{decreq}
\begin{defn}{}
We call two functions $b,b': E\to\mathbb{Z}$ 
{\em decreasingly-equivalent} if
for each path
$\nu_1\overset{e_1}{\longrightarrow }\nu_2\overset{e_2}{\longrightarrow }\nu_3$ one has
$$b(e_1)>b(e_2)\ \Longleftrightarrow\ b'(e_1)>b'(e_2).$$
\end{defn}

 \subsubsection{}
Observe that two decreasingly-equivalent graphs have 
the same  set of decreasing paths.

\subsubsection{}\label{Finass}
We denote the set of decreasing (resp., increasing) paths
from $\nu$ to $\lambda$ by $\cP^{>}(\nu,\lambda)$ (resp., $\cP^{<}(\nu,\lambda)$).

Let $(V,E)$ be a $\mathbb{Z}$-graded  graph with a finite number of edges
 between any two vertices.
Notice that in this case the number of paths between any two vertices is finite.

We introduce the square matrices $A^{<}(\kappa)=(a^{<}_{\lambda,\nu})_{\lambda,\nu\in V}\ $
and  $A^{>}(\kappa)=(a^{>}_{\lambda,\nu})_{\lambda,\nu\in V}\ $ by
$$a^{>}_{\lambda,\nu}:=\sum_{P\in \cP^{>}(\nu,\lambda)} \kappa(P),\ \ \ 
a^{<}_{\lambda,\nu}:=\sum_{P\in \cP^{<}(\nu,\lambda)}(-1)^{length(P)}\kappa(P).$$
Since  the graph is a $\mathbb{Z}$-graded, these matrices
are lower-triangular with  $a^{>}_{\lambda,\lambda}=a^{<}_{\lambda,\lambda}=1$. The following lemma is proven in~\cite{Gdex}, Section 3.4
(the proof is similar to one in~\cite{GS}, Thm. 4).

\subsubsection{}
\begin{lem}{lemAAA}
Let $(V,E)$ be a $\mathbb{Z}$-graded  graph with a finite number of edges
 between any two vertices. Assume that $b: E\to \mathbb{Z}$ satisfies the property

(BB) $ \ \ \ $
{\em for each path
$\ \nu_1\overset{e_1}{\longrightarrow }\nu_2\overset{e_2}{\longrightarrow }\nu_3\ $ one has $\ \ \ b(e_1)\not=b(e_2)$.} 

Then $A^{>}(\kappa)\cdot A^{<}(\kappa)=A^{<}(\kappa)\cdot A^{>}(\kappa)=\Id$.
\end{lem}

\subsection{Graphs $\hat{\Gamma}^{\chi}_{st}$ and $\Gamma^{\chi}_{st}$}\label{graphs}
We take $\fg:=\fgl(m|n)$ and fix a central character $\chi$. We
 define  $\hat{\Gamma}^{\chi}_{st}$
and its subgraph ${\Gamma}^{\chi}_{st}$ similarly to~\cite{Gdex}.

\subsubsection{Graph $\hat{\Gamma}^{\chi}_{st}$}\label{tildeGammachi}
Let $\hat{\Gamma}^{\chi}_{st}(z)$ be  a graph with the set of vertices
$V:=\Lambda^{\chi}_{st}$ and the following edges: if $K^{\lambda,\nu}_{\fp^{(i)},\fp^{(i-1)}}\not=\delta_{\nu,\lambda}$ (where $\delta_{\nu,\lambda}$ is the Kronecker symbol)
we join  $\nu,\lambda$ by 
the edge of  the form
$$\nu\overset{e}{\longrightarrow}\lambda\ \text{ with }\ \ b(e)=i.$$
The graph ${\Gamma}^{\chi}_{st}$ is obtained
from $\hat{\Gamma}^{\chi}_{st}$ by removing the edges of the form
$\nu\overset{e}{\longrightarrow} \lambda$
with $b(e)\leq \tail(\lambda)$. For the core-free case
$\Lambda^{\chi}_{st}=\Lambda^{\chi}$ and we denote the resulting graphs by
$\hat{\Gamma}^{\chi}$ and ${\Gamma}^{\chi}$ respectively.

 We denote by $P^{>}(\nu,\lambda)$ 
the set of decreasing paths from 
$\nu$ to $\lambda$ in the graph $\Gamma^{\chi}_{st}$.

By~\ref{identific-graph}  if $\nu\overset{e}{\longrightarrow}\lambda$
is an edge in $\hat{\Gamma}^{\chi}_{st}$ (resp., in $\Gamma^{\chi}_{st}$) then    $\lambda$ is obtained from $\nu$ by a  move (a non-tail move) 
of weight $s$
and for $i:=b(\lambda;\mu)=b(e)$ one has
 $K^{\lambda,\nu}_{\fp^{(i)},\fp^{(i-1)}}=z^s$.
In particular, $\hat{\Gamma}^{\chi}_{st}$
does not have multi-edges and is $\mathbb{Z}$-graded with respect to 
$||\lambda||$. 

\subsubsection{}
 Take $\nu,\lambda\in\Lambda^{\chi}$ 
with $\nu\not=\lambda$. By~\Lem{lemAA} (i)
\begin{equation}\label{pyatka}
K^{\lambda,\nu}_{\fp^{(i)},\fp^{(i-1)}}(-1)=(-1)^{||\lambda||-||\nu||+1}
\end{equation}
if $\hat{\Gamma}^{\chi}_{st}$ contains an edge  $\nu\overset{e}{\longrightarrow}\lambda$  with $b(e)=i$ and $K^{\lambda,\nu}_{\fp^{(i)},\fp^{(i-1)}}=0$ otherwise.
By~\Lem{lemAA} (ii) the graph ${\Gamma}^{\chi}_{st}$
is $\mathbb{N}$-graded with respect to 
$||\lambda||_{gr}$ and satisfies the following condition

(Tail) $\ \ \ \ \ $
for each edge $\nu \overset{e}{\longrightarrow }\lambda$ in  $\Gamma^{\chi}_{st}$
one has $\tail(\nu)\leq b(e)$.

This condition implies the following important property:
 a decreasing path $P$ in $\hat{\Gamma}^{\chi}_{st}$
lies in $\Gamma^{\chi}_{st}$ if and only if the last edge of $P$
lies in this graph.
Using this property,~(\ref{iterative}) and~(\ref{pyatka}) we obtain 
for $\lambda\in \Lambda^{\chi}_{st}$ and $\mu\in\Lambda^+_{m|n}$ 
\begin{equation}\label{formulaforK1}
\begin{array}{l}
K^{\lambda,\nu}_{\fg,\fp_{\lambda}}(-1)=
(-1)^{||\lambda||-||\mu||}\displaystyle\sum_{P\in {P}^{>}(\nu,\lambda)}(-1)^{length(P)}.\end{array}\end{equation}
(see Cor. 3.6 in~\cite{Gdex} for details). 
For $K^{\lambda,\nu}_{\fg,\fb}(-1)$ one has the similar formula in terms of the decreasing paths
in $\hat{\Gamma}^{\chi}_{st}$.

\subsubsection{}
Let $E$ be the set of edges in $\Gamma^{\chi}_{st}$.
We introduce $b': E\to \mathbb{Z}$ by
$$b'(\mu \overset{e}{\longrightarrow}\lambda):=b'(\nu,\lambda).$$
One readily sees that $b$ and $b'$ are decreasingly equivalent.
We denote by $\cP^{>}_{\nu,\lambda}$ the number of
paths from $\nu$ to $\lambda$ in $\Gamma^{\chi}_{st}$
which are increasing with respect to $b'$.

Moreover, $b'$ satisfies the property (BB).
 Using~\Lem{lemAAA} 
we conclude that for  a stable weight $\lambda$ one has
$$\ch L(\lambda)=\sum_{\mu\in\Lambda^+_{m|n}} (-1)^{||\lambda||-||\mu||} d^{\lambda,\mu}_<\mathcal{E}_{\mu},$$
where
$d^{\lambda,\mu}_<$ is the cardinality of $\cP^{>}_{\nu,\lambda}$.

\subsubsection{}
Notice that $\Gamma^{\chi}_{st}$ coincides with 
 the ``stable part''   (the full subgraph corresponding to the stable vertices) of  the component $D^{\chi}_{\fg}$ and  that $b'$
corresponds to the marking in this graph. This completes
the proof of~(\ref{EL}).

\subsection{Examples}
Consider the core-free case: $\fg=\fgl(r|r)$ and  $\chi_{\lambda}=\chi_0$.

\subsubsection{Case $r=1$}
In this case  $\mathcal{E}_{\lambda}=\ch L(\lambda)=e^{\lambda}$.

\subsubsection{Case $r=2$}
Set $\beta_1:=\vareps_1-\delta_2,\ \beta_2:=\vareps_2-\delta_1$.
 
The  weights $\lambda\in\Lambda^+_{2|2}$ with $\chi_{\lambda}=\chi_0$ 
are of  the form $s(\beta_1+\beta_2)+i\beta_1$ 
for $s\in\mathbb{Z}$, $i\in\mathbb{Z}_{\geq 0}$; we denote such weight 
by $(s;i)$. 
The diagram of $(s;i)$ has symbols $\times$ at the positions $s$ and $s+i+1$.
One has $||(s;i)||_{gr}=i$ and $\tail(s;i)=1+\delta_{i0}$.

The graph $D_{\fg}$ is described in \ref{exaDg}.
The decreasing paths are the paths of length at most $1$; combining~(\ref{Euler}) and~(\ref{formulaforK1}) we obtain
$$\mathcal{E}_{s;i}=\left\{ \begin{array}{ll}
\ch L(s;0) & \text{ if } i=0,\\
\ch L(s;1)+\ch L(s;0)+\ch L(s-1;0)& \text{ if } i=1,\\
\ch L(s;i)+\ch L(s;i-1) & \text{ if } i>1.
\end{array}\right.$$
For $j>0$  a vertex $(s;j)$ can be reached by increasing paths from vertices $(s;i)$
 for $0\leq i\leq j$ and from a vertex $(s-1;0)$; in both cases  the path
is unique; this gives
$$\ch L(s;j)=(-1)^j\mathcal{E}_{s-1;0}+(-1)^{j-i}\displaystyle\sum_{i=0}^j \mathcal{E}_{s;i},\ \ \ \sch L(s;j)=\mathcal{E}^-_{s-1;0}+\displaystyle\sum_{i=0}^j \mathcal{E}^-_{s;i}.$$

\subsection{Comparison with other character formulas}\label{other-formulas}
For the $\fgl(2|2)$-case a weight $(s,i)$ is a Kostant weight only if $i=0$; thus
the Kac-Wakimoto character formula does not hold for $L(s;i)$ with $i\not=0$.
By~\cite{Drouot} the restriction of any $L(s,i)$ for $i > 0$ is a sum of four simple $\fgl_0$-modules,so  the character of $L(s,i)$ is a sum over four Weyl character formula terms for $\fgl_0$.
Any $\fgl(2|2)$-module is always {\em partially disconnected} (PDC) in the sense of \cite{CHR2}. For PDC weights the authors establish the following character formula

\[
e^{\rho}R\cdot\mbox{ch }L\left(\lambda\right)=\frac{(-1)^{|(\lambda^{\rho})^{\Uparrow}-\lambda^{\rho}|_{S_{\lambda}}}}{t_{\lambda}}\jJ_W\left(\frac{e^{\left(\lambda^{\rho}\right)^{\Uparrow}}}{\prod_{\beta\in S_{\lambda}}\left(1+e^{-\beta}\right)}\right),
\]

where we refer to loc. cit for the definitions. The number $t_{\lambda}$ is two for $(s,0)$ and one for $(s,i), \ i > 0$. However already for $\fgl(3|3)$ there are simple modules which are not PDC.

The Su-Zhang formula \cite{SuZh1} 
expresses the character in terms of  $KW(\lambda,S)$ where $S$ is chosen to be maximal  (the cardinality of $S$ is equal to the atypicality of $L$) whereas we take $S$  
with cardinality equal to $\tail(L)$.

\section{Euler characters for $\fgl(m|n)$}\label{euler-gl}
In this section we define $\tail(\lambda)$ and the $\lambda^{\dagger}$ which appeared in (iv) in~\ref{overview} for the $\fgl(m|n)$-case. 
In addition, in~\Cor{corEla} we deduce the property 
(i) of Section~\ref{overview}.

Recall that in this case $\tail(\lambda)$ was introduced
only for stable weights;   the stable weight diagram
has a ``horizontal tail'', that is the leftmost part of the diagram is of the form $\underbrace{\times\ldots\times}_{\tail(\lambda)\text{ times}}$.
In~\ref{dagger} below we introduce $\tail(\nu)$
for non-stable weights and assign to each weight $\nu\in\Lambda^+_{m|n}$ 
a weight diagram $f^{\dagger}$ with a ``vertical tail'' of size $\tail(\nu)$;  the weight  $\nu^{\dagger}\in\Lambda^{\geq}_{m|n}$ is defined by 
$\diag(\nu^{\dagger}-\rho)=f^{\dagger}$.  The set
$\Lambda^{\dagger}_{m|n}:=\{\lambda^{\dagger}|\ \lambda\in\Lambda^+_{m|n}\}$
 consists of the diagrams where
at most one position contains more than one of the symbols $\circ,>,<,\times$
and, if such a position exists, it contains $\times^i$ for $i>1$ with no
 symbols $\times$ which preceed this position.
The correspondence $f\mapsto f^{\dagger}$ gives a bijection between
$\Lambda^+_{m|n}$ and $\Lambda^{\dagger}_{m|n}$, see~\ref{tail-up} for details and examples.

In this section $\fg:=\fgl(m|n)$.

\subsection{Straightening the tail}\label{tail-up}
Take  $\nu\in {\Lambda}^{\geq}_{m|n}$ and set $f:=\diag(\nu)$.

\subsubsection{}\label{dagger}
Recall that the atypicality
of $\nu$ is equal to the number of the symbols $\times$ in $f$.
We assign to $f$ $\tail(f)$ 
and the position $y_0(f)$ as follows.

If $\nu$ has atypicality zero ($f$ does not contain $\times$) we set
$\tail(f)=0$ and $y_0(f)=\infty$.

If $\nu$ is atypical we let
 $\tail(f):=s$, where $s$ is  the maximal number such that
$f$ does not have empty symbols between the first and the $s$th symbols $\times$;
we denote by $y_0(f)$ the coordinate of $s$th symbol $\times$ in $f$. 
We denote  by $f^{\dagger}$ the diagram obtained from $f$ by the following procedure:
all symbols $\times$ with the coordinates
less than $y_0(f)$ are moved to the position $y_0(f)$ and other symbols remain
the same.

For example, 
$$\begin{array}{llll}
f_1=\ 0>\circ\times>\times\times\circ\times< \  & \tail(f_1)=3  &\ y_0(f_1)=6 \ 
& f_1^{\dagger}=\ 0>\circ\circ>\circ\times^3\circ\times<\\
\\
f_2=\ 0\times^3>\times^2<\circ\times^2 & \tail(f_2)=5 &  y_0(f_2)=3 & 
 f_2^{\dagger}=\ 0\circ>\times^5<\circ\times^2.
 \end{array}$$

Thus $f^{\dagger}$ contains $\tail(f)$ symbols $\times$ at the position $y_0$
and no symbols $\times$ with smaller coordinates.  Notice that
$$(f^{\dagger})^{\dagger}=f^{\dagger},\ \ \ \ \tail(f^{\dagger})=
\tail(f),\ \ \ \ y_0(f^{\dagger})=y_0(f).$$

Let $\lambda^{\dagger}\in\Lambda^{\geq}_{m|n}$ be the weight corresponding to $f^{\dagger}$, i.e. $\diag(\lambda^{\dagger}-\rho)=f^{\dagger}$.

For instance, for $f_1$ as above we have 
$$\begin{array}{l}
\nu+\rho=\vareps_6+3(\vareps_5-\delta_1)+
4\vareps_4+5(\vareps_3-\delta_2)+6(\vareps_2-\delta_3)+ 8(\vareps_1-\delta_4)-9\delta_5\\
\nu^{\dagger}=\vareps_6+
4\vareps_5+6(\varesp_4+\vareps_3+\vareps_2-\delta_1-\delta_2-\delta_3)
+ 8(\vareps_1-\delta_4)-9\delta_5.
\end{array}$$

\subsubsection{}\label{y0f}
Set $s:=\tail(f)$. Recall that $f$ contains at most one symbol $\times$ in each position.
Each symbol $\times$ in $f$ corresponds to an isotropic root in $\Delta^+$:
for instance, for $f_1$ as above the rightmost $\times$ corresponds
to $\vareps_1-\delta_4$ and the leftmost $\times$ corresponds
to $\vareps_5-\delta_1$. Let $\vareps_p-\delta_q$ be the root corresponding to $\times$ at the position $y_0(f)$ (for $f_1$ in the above example
$y_0(f_1)=6$ and 
$\vareps_p-\delta_q=\vareps_2-\delta_3$). One has
$$0=(\nu|\vareps_p-\delta_q)=(\nu^{\dagger}|\vareps_p-\delta_q).$$
Moreover, $(\nu^{\dagger}|\vareps_i)=(\nu^{\dagger}|\delta_j)$ for
$i=p,p+1,\ldots,p'$ and  $j=q',q'+1,\ldots, q$ with
$p'-p=q-q'=s-1$. (In the above example,
$(\nu^{\dagger}|\vareps_i)=(\nu^{\dagger}|\delta_j)=6$
for $i=2,3,4$ and $j=1,2,3$).
We set
$$S_{\nu^{\dagger}}:=\{\vareps_{p+i}-\delta_{q'+i}\}_{i=0}^{s-1}.$$
By above, $(\nu^{\dagger}|S_{\nu^{\dagger}})=0$. (In the above example,
$S_{\nu^{\dagger}}=\{\vareps_2-\delta_1,
\vareps_3-\delta_2, \vareps_4-\delta_3\}$). For $\nu\in\Lambda^+_{m|n}$ we set
$$\KW(\nu):=\KW(\nu^{\dagger},S_{\nu^{\dagger}}).$$

\subsubsection{Remark}
Clearly, $S_{\lambda^{\dagger}}$ can be embedded to
a base compatible with $\Pi_0$:
$$S_{\lambda^{\dagger}}\subset \delta^{q'-1}\vareps^{p-1}(\vareps\delta)^s\delta^{n-s-q'+1}\vareps^{m-s-p+1}.$$ 
Using \ref{hwt} one can show that for $\lambda\in\Lambda^+_{m|n}$
the weight $\lambda^{\dagger}$ is the $\rho$-twisted highest
weight of $L:=L(\lambda)$ with respect to this base, which we denote by 
$\Sigma_L$.
If $\lambda$ is stable, then $\Sigma_L$ is introduced in~\Prop{propE}, see~\ref{exa11} below.

\subsubsection{Remark}
In the $\osp$-case the weight $\lambda^{\dagger}=\lambda+\rho$
has a "vertical tail" ($\diag(\lambda)$  has $\times^s$ in the zero position); in the $\fgl$-case $\diag(\lambda)$ has a "horizontal tail"
and the diagram of $\lambda^{\dagger}$ has a "vertical tail" of the same size.

\subsubsection{Example: stable weight $\lambda$}\label{exa11}
Let $\lambda\in\Lambda^+_{m|n}$ be a stable weight.
In this case the weight diagram of $\lambda^{\dagger}$ starts from
$\times^s$ and 
${S}_{\lambda^{\dagger}}=\{\vareps_{m-s+i}-\delta_i\}_{i=1}^s$.

Set $L:=L(\lambda)$. As in~\Prop{propE} we
 denote by $\Sigma_L$ the base corresponding to the word
$\vareps^{m-s}(\vareps\delta)^{s}
\delta^{n-s}$ and by $\rho_L$ the Weyl vector corresponding to $\Sigma_L$.
 Notice that ${S}_{\lambda^{\dagger}}\subset \Sigma_L$.
Observe $\Sigma_L$ is obtained from $\Sigma$ by  odd reflections with respect to the roots of $\fg_{\lambda}$; these reflections do not change the highest weight of $L$, so the highest weight of $L$ with respect to $\Sigma_L$
is $\lambda$. It is easy to see that 
$$\lambda^{\dagger}=\lambda+\rho_L.$$
Note that $(\lambda^{\dagger}| S_s)=0$.
Using~\Prop{propE} we get
$$s!Re^{\rho}\cE_{\lambda}=\KW(\lambda)=(-1)^{[\frac{s}{2}]}\KW(\lambda^{\dagger}, S_s).$$

\subsection{Remark}
Consider the set of  $\rho$-twisted highest
weights of $L:=L(\lambda)$: 
$$Hwt(L):=\{\hwt_{\Sigma'} L|\ 
\text{ where }\Sigma' \text{ is a base of }\Delta\}.$$

For each $\nu\in \Lambda_{m|n}$ let $s(\nu)$ be the maximal cardinality
of an iso-set orthogonal to $\nu$ which can be embedded
to a base compatible with $\Pi_0$. The number $s(\nu)$ can be visualized 
if we draw the weight diagram of $\nu$ using the same rules
as in~\ref{wtdiag} (eventhough $\nu$ is not always in $\Lambda^{\geq}_{m|n}$):
 $s(\nu)$ is equal to the maximal number of symbols $\times$ which 
occupy the same position in the weight diagram. In particular, for $\nu\in\Lambda^{\dagger}$,
$s(\nu)$  is the size of the ``vertical tail'' of the diagram.

By above, $\lambda^{\dagger}\in Hwt(L)$. One has $s(\lambda^{\dagger})=
\tail(\lambda)$ 
(and $S_{\lambda^{\dagger}}$ is an iso-set of the maximal cardinality
which is orthogonal $\lambda^{\dagger}$).

\subsubsection{Conjecture}\label{conjA}
$\tail(\lambda)=\max_{\nu\in Hwt(L(\lambda))} s(\nu)$.

%
%
%\subsubsection{Optimistic Conjecture}\label{conjB}
%The weight $\lambda^{\dagger}$ is the only weight
%in $Hwt(L)\cap \Lambda^{\dagger}$ satisfying $s(\nu)=\tail(\lambda)$.

\subsubsection{Remark: the $\osp$-case}
The formula in~\ref{conjA} does not hold for the $\osp$-case:
if $\lambda$ is a singly atypical $\osp(2m|2n)$-weight satisfying
the Kac-Wakimoto conditions which is not a Kostant weight, then the right-hand side of the formula is $1$,
whereas the left-hand side is $0$. A natural question is whether the formula holds for other weights.

\subsection{Translation functors and $\KW(\lambda)$}\label{KWtrans}
Retain notation of~\ref{Taa+1}.

\begin{lem}{TransandKW}
Take $\lambda\in\Lambda_{m|n}$.
Let $a\in\mathbb{Z}$ be such that exactly one of the positions $a,a+1$
in $\core(\lambda)$ is empty.

\begin{enumerate}
\item
$(T_{a,a+1}(\lambda))^{\dagger}=T_{a,a+1}(\lambda^{\dagger})$;
\item For  $\chi'=T_{a,a+1}(\chi_{\lambda})$ one has
$\Theta^V_{\chi', \chi_{\lambda}}\bigl(\KW(\lambda)\bigr)
=\KW(T_{a,a+1}(\lambda))$.
\end{enumerate}
\end{lem}
\begin{proof}
For (i) note that 
$(T_{a,a+1}(f))^{\dagger}=T_{a,a+1}(f^{\dagger})$  except  for the case $a=y_0(f)$ and $f(a+1)=\circ$; in the latter case the positions $a,a+1$ 
in  $\core(f)$ are empty.
Combining (i) and the formula $(\lambda^{\dagger})^{\dagger}=(\lambda^{\dagger})$ 
we reduce (ii) to the case when $\lambda=\lambda^{\dagger}$ (i.e.,
$\lambda\in\Lambda^{\dagger}$). We set $S:=S_{\lambda}$.
We will consider the case when 
$$\core(\lambda)(a)=>,\ \ \ \ \ \core(\lambda)(a+1)=\circ$$
(other cases are similar). In this case  
$V=V_{st}$.

Using the formula $P_{\chi}(\KW(\lambda))=\KW(\lambda)$ and 
 $W$-invariance of $\ch V$  we  get
$$\Theta^V_{\chi,\chi'}\bigl(\KW(\lambda)\bigr)=P_{\chi'}\bigl(\ch V\cdot \jJ_W\bigl(\frac{e^{\lambda}}{\prod_{\beta\in S}(1+e^{-\beta})}\bigr)\bigr)=
P_{\chi'}\bigl(\jJ_W\bigl(\frac{e^{\lambda}\ch V}{\prod_{\beta\in S}(1+e^{-\beta})}\bigr)\bigr)$$
which allows to rewrite (ii) in the following form
\begin{equation}\label{aaa}
P_{\chi'}\bigl(\jJ_W\bigl(\frac{e^{\lambda}\ch V}{\prod_{\beta\in S}(1+e^{-\beta})}\bigr)\bigr)
=\KW(T_{a,a+1}(\lambda)).
\end{equation}

Recall that 
$S=\{\vareps_{p+i}-\delta_{q+i}\}_{i=1}^s$ 
for $s:=\tail(\lambda)$ and some $p,q$. Set
$$A:=\{\gamma\in \{\vareps_i\}_{i=1}^m\cup\{\delta_j\}_{j=1}^n|\ (\gamma,S)=0\}$$
and
$S_i:=S\setminus\{\vareps_{p+i}-\delta_{q+i}\}\ \text{ for }i=1,\ldots,s$.
Using $\ch V=\sum_{i=1}^m e^{\vareps_i}+\sum_{j=1}^n e^{\delta_j}$
we get
$$\begin{array}{l}
\jJ_W\bigl(\frac{e^{\lambda}\ch V}{\prod_{\beta\in S}(1+e^{-\beta})}\bigr)=\displaystyle\sum_{\gamma\in A}\jJ_W\bigl(\frac{e^{\lambda+\gamma}}{\prod_{\beta\in S}(1+e^{-\beta})}\bigr) + \displaystyle\sum_{i=1}^s
\jJ_W\bigl(\frac{e^{\lambda+\vareps_{p+i}}}{\prod_{\beta\in S_i}(1+e^{-\beta})}\bigr) 
\\=
\displaystyle\sum_{\gamma\in A} \KW(\lambda+\gamma;S)+
\displaystyle\sum_{i=1}^s\KW(\lambda+\vareps_{p+i}; S_i).\end{array}$$

 Using~(\ref{PchiKW}) we obtain
$$P_{\chi'}\bigl(\jJ_W\bigl(\frac{e^{\lambda}\ch V}{\prod_{\beta\in S}(1+e^{-\beta})}\bigr)\bigr)=\displaystyle\sum_{\gamma\in B} \KW(\lambda+\gamma,S)+\displaystyle\sum_{\gamma\in B_0} \KW(\lambda+\gamma,S_i),$$
where $B:=\{\gamma\in A|\ \core(\lambda+\gamma)=
T_{a,a+1}(\core(\lambda))\}$ and
$$ B_0:=\{\vareps_{p+i}|\  i=1,\ldots, s\ \text{ s.t. } 
\ \core(\lambda+\vareps_i)=T_{a,a+1}(\core(\lambda))\}.$$

Denote by $f$ the weight diagram of $\lambda$. Recall that $f(a)=>$ and
$f(a+1)=\times^j$ for some $i$.
Since $\lambda\in\Lambda^{\dagger}$, $f$ has a ``vertical tail'' at the position $y:=y_0(f)$ (so $j=1$ if $a+1\not=y$).
Since $f(a)=>$   there exists a unique $k$  such that  
$(\lambda,\vareps_k)=a$. 
Note that $\core(\lambda+\vareps_i)=T_{a,a+1}(\core(\lambda))$  implies
$i=k$ and $\core(\lambda+\delta_i)=\core(\lambda)$ implies
$(\lambda,\delta_i)=-a-1$. Note that
$\vareps_k\in A$, so $B_0=\emptyset$. We get
$$
P_{\chi'}\bigl(\jJ_W\bigl(\frac{e^{\lambda}\ch V}{\prod_{\beta\in S}(1+e^{-\beta})}\bigr)\bigr)=\left\{\begin{array}{l}
\KW(\lambda+\vareps_k,S)\ \text{ if }a+1=y\\
 \KW(\lambda+\vareps_k,S)
+\!\!\displaystyle\sum_{i: (\lambda,\delta_i)=-a-1} \KW(\lambda+\delta_i,S)\text{ otherwise}
\end{array}\right.
$$

If $f(a+1)=\circ$, then   $(\lambda,\delta_i)\not=-a-1$ for all $i$ and
$\lambda+\vareps_k=T_{a,a+1}(\lambda)$
with $S_{\lambda+\vareps_k}=S$; thus $\KW(\lambda+\vareps_k,S)=\KW(T_{a,a+1}(\lambda))$
and 
 this gives~(\ref{aaa}).

Consider the case  $f(a+1)=\times$ with $a+1\not=y$. 
By~\Lem{lemprlambdaS} (ii), $\KW(\lambda+\vareps_k,S)=0$ (since 
$(\lambda+\vareps_k,\vareps_{k-1}-\vareps_k)=(S,\vareps_{k-1}-\vareps_k)=0$). 
Since $a+1\not=y$, 
there is a unique $i$ such that  $(\lambda,\delta_i)=-a-1$.
One has $\lambda+\delta_i=T_{a,a+1}(\lambda)$ and $S_{\lambda+\delta_i}=
S$, so~(\ref{aaa}) holds.

In the remaining case $a+1=y$. Then  $f(a+1)=\times^s$ and
$k=p+s+1$. Note that $(\lambda+\vareps_k,\vareps_i)=a+1$ if and only if $i=p+1,\ldots,
p+s+1$ and $(\lambda+\vareps_k,\delta_i)=-a-1$ if and only if $i=q+1,\ldots, q+s$. Set 
$$\mu:=(a+1)(\sum_{i=1}^{s+1} \vareps_{p+i}-\sum_{i=1}^s\delta_{q+i}).$$
Let $W_{\mu}\cong S_{s+1}\times S_s\subset W$ be the group 
of permutations of $\vareps_{p+1},\ldots,\vareps_{p+s+1}$ and of
$\delta_{q+1},\ldots,\delta_{q+s}$. Notice that 
$\lambda+\vareps_k$ is $W_{\mu}$-invariant. 
Choosing any set of representativies in $W/W_{\mu}$ we have
$$\KW(\lambda+\vareps_k,S)=
\jJ_W\bigl(\frac{e^{\lambda+\vareps_k}}{\prod_{\beta\in S}(1+e^{-\beta})}\bigr)=\jJ_{W/W_{\mu}}\bigl(e^{\lambda+\vareps_k} 
\jJ_{W_{\mu}}
(\frac{1}{\prod_{\beta\in S}(1+e^{-\beta})})\bigr).$$
Comparing the denominator identities for $\fgl(s+1|s)$ with respect to 
the bases $(\vareps\delta)^s\vareps$ and $\vareps(\vareps\delta)^s$ 
we get
$$\jJ_{W_{\mu}}
(\frac{1}{\prod_{\beta\in S}(1+e^{-\beta})})\bigr)=
\jJ_{W_{\mu}}
(\frac{e^{-\sum_{\beta\in S'}\beta}}{\prod_{\beta\in S'}(1+e^{-\beta})})\bigr),$$
where 
$S':=\{\vareps_{p+1+i}-\delta_{q+i}\}_{i=1}^s$. This gives
$$
\jJ_W\bigl(\frac{e^{\lambda+\vareps_k}}{\prod_{\beta\in S}(1+e^{-\beta})}\bigr)=\jJ_{W/W_{\mu}}\bigl(e^{\lambda+\vareps_k} 
\jJ_{W_{\mu}}
(\frac{e^{-\sum_{\beta\in S'}\beta}}{\prod_{\beta\in S'}(1+e^{-\beta})})\bigr)=
\jJ_W\bigl(\frac{e^{\lambda'}}{\prod_{\beta\in S'}(1+e^{-\beta})}\bigr),
$$
where $\lambda':=\lambda+\vareps_k-\sum_{\beta\in S'}\beta$.
One readily sees that
$\lambda'=T_{a,a+1}(\lambda)$ and $S'=S_{\lambda'}$. This completes the proof.
\end{proof}

\subsection{}
\begin{cor}{corEla}
\begin{enumerate}
\item
Let $T:\Fin^{\chi}\iso \Fin^{\chi'}$ be a composition
of the translation functors $T^{V}_{\chi,\chi'}$ 
which are equivalence of categories and let $\Theta_{\chi,\chi'}: \cR_{\Sigma}\to\cR_{\Sigma}$ be the corresponding composed map.
If $T(L(\lambda-\rho))=L(\lambda'-\rho)$, then
$\Theta_{\chi,\chi'}(\KW(\lambda))=\KW(\lambda')$.
\item
For each $\lambda\in\Lambda^+_{m|n}$ one has
$$\tail(\lambda)! Re^{\rho}\cE_{\lambda}=\KW(\lambda).$$
\end{enumerate}
\end{cor}

\subsubsection{}\label{RemEla}
Denote by $K(\lambda)$ the Kac module of the highest weight $\lambda$.
Take $\lambda'$ as in~\Cor{corEla} (i). From~\cite{Sselecta}, Thm. 5.1
it follows that $T(K(\lambda))=K(\lambda')$, where $,\lambda'$ are as above.
This gives the following formula
\begin{equation}\label{TKacmodule}
\Theta_{\chi,\chi'}(Re^{\rho} K(\lambda-\rho))=K(\lambda'-\rho)
\end{equation}
which will be used later (this formula can be also proven as~\Lem{TransandKW} (ii)).

\section{Euler supercharacters and the Duflo-Serganova functor} \label{sec:DS}

Let $\Sch(\mathfrak{g})$ be the ring of supercharacters
of $\fg$.
Recall that $\DS_x$ induces for any $x$ a homomorphism $\Sch(\mathfrak{g}) \to \Sch(\mathfrak{g}_x)$ which depends only on the rank of $x$. We denote this homomorphism by $\ds_j$, where
$j$ is the rank of $x$. We always assume that $j>0$ 
($\ds_0=Id$). 

\iffalse
\textcolor{blue}{If we keep 1.6, we may kill the  following text: START}
The image of $\ds_x$ is $\Sch(\fg_x)^{\sigma}$, where
$\sigma=Id$ (i.e., $\ds_x$ is surjective)  except
for  $\fg=\osp(2m|2n), F(4), D(2|1;a)$, where $\sigma$ is an involution
of $\fg_x$ (see~\cite{HR},\cite{Gcore}).

\subsection{$\DS_x$ for the exceptional Lie superalgebras}
For the exceptional Lie superalgebras all atypical blocks are of atypicality $1$.
\textcolor{red}{These  blocks are either isomorphic XXXXX of two type: the first type
 component $D^{\chi}_{\fg}$ is either as
in $\osp(2|2)$-case (then $\DS_x(L(\lambda))$ is always $\sigma$-simple)
or as in $\osp(3|2)$-case (then $\DS_x(L(\lambda))$ is $\sigma$-simple
only for two weights which correspond to the trivial and the standard
modules for $\osp(3|2)$).}   

Each  atypical block contains \textcolor{red}{a} ``Kostant weight'' 
$\lambda_0$ such that $\sch L(\lambda_0)=\cE^-_{\lambda_0}$ \textcolor{red}{and}
$\ds_1(\cE^-_{\lambda})=0$ for $\lambda\not=\lambda_0$

$\ds_1(\cE^-_{\lambda_0})=\cE^-_{\lambda'_0}$ if $(\lambda'_0)^{\sigma}=\lambda'_0$, 
 $\ds_1(\cE^-_{\lambda_0})=\cE^-_{\lambda'_0}+\cE^-_{\sigma(\lambda'_0)}$
otherwise.
For ``non-Kostant''  weights $\lambda$  one has,
see~\cite{Gdex} for details.

\textcolor{blue}{END}\fi

In this section $\fg$ stands for $\fgl(m|n)$ or $\osp(M|N)$.

\subsection{Euler supercharacters}\label{defpi}
Recall that $\pi$ is the involution of  $\mathbb{Z}[\Lambda_{m|n}]$
given by $\pi(e^{\lambda}):=(-1)^{p(\lambda)}e^{\lambda}$; we extend this involution to the ring of fractions of $\mathbb{Z}[\Lambda_{m|n}]$.
Recall that $\sch L(\lambda)=(-1)^{p(\lambda)}\pi(\ch L(\lambda))$.
For each $\nu\in\Lambda^+_{m|n}$ set
$$\cE_{\nu}^-:=(-1)^{p(\nu)}\pi(\cE_{\nu}).$$
Using the character formulas and~\ref{plambdanu} we obtain the following formulas.

\subsubsection{}
\begin{cor}{blambda}
For $\lambda\in\Lambda^+_{m|n}$ one has
$$\sch L(\lambda)=\sum_{\mu\in\Lambda^+_{m|n}}(-1)^{p(\lambda-\mu)+||\lambda||-||\mu||}
d^{\lambda,\mu}_< \mathcal{E}^-_{\mu},$$
where $d^{\lambda,\mu}_<$ is the number of increasing paths
from $\mu$ to $\lambda$ in $D_{\fg}$. 

If $\lambda$ is stable and $t\not=2$, then
$\sch L(\lambda)=\sum_{\mu\in\Lambda^+_{m|n}}
d^{\lambda,\mu}_< \mathcal{E}^-_{\mu}$.
\end{cor}

The main result of this section is the following theorem, which will be proven in \ref{proofcut} below. 

\subsection{}
\begin{thm}{thmcut}
Take $\lambda\in\Lambda^+_{m|n}$. If $\tail(\lambda)<j$, then
$\ds_j(\mathcal{E}^-_{\lambda})=0$. If $\tail(\lambda)\geq j$
let $\lambda'\in\Lambda^+_{m-j|n-j}$ be such that the weight diagram
$\lambda'+\rho'$ is obtained from the diagram of $\lambda+\rho$ by 
 the removal the first $j$ leftmost symbols $\times$. Then

$$\ds_x(\mathcal{E}^-_{\lambda})=\left\{\begin{array}{ll}
\mathcal{E}^-_{\lambda'}& \text{if }\ \tail(\lambda)>j;\\
\mathcal{E}^-_{\lambda'}& \text{if }\ \tail(\lambda)=j,\ \
\fg=\osp(2m+1|2n);\\
\mathcal{E}^-_{\lambda'}&  \text{if }\ \tail(\lambda)=j,\ \
\fg=\osp(2j|2n);\\
\sch K(\lambda') & \text{if }\
\tail(\lambda)=j,\ \fg=\fgl(m|n),
\end{array}\right.$$
where $K(\lambda')$ is the Kac $\fg'$-module with  the even highest weight
vector of weight $\lambda'$.

For $\osp(2m|2n)$ with $m>j=\tail(\lambda)$ one has
$$\ds_x(\mathcal{E}^-_{\lambda})=\left\{\begin{array}{ll}
\mathcal{E}^-_{\lambda'} & \text{ if }\ \ (\lambda')^{\sigma}=\lambda',\\
\mathcal{E}^-_{\lambda'}+\mathcal{E}^-_{(\lambda')^{\sigma}}& \text{ if }\
(\lambda')^{\sigma}\not=\lambda'.
\end{array}\right.$$
\end{thm}

\subsubsection{Remark}
For a typical module $N$ one has $\DS_x(N)=0$ for each $x\not=0$.
If $L(\lambda)$ is typical, then $Re^{\rho}\ch L(\lambda)=\KW(\lambda+\rho,\emptyset)$
and
$$\sch L(\lambda)=\cE_{\lambda}^-=(\pi(R))^{-1}\sum_{w\in W}
(-1)^{p(\lambda+\rho-w(\lambda+\rho))}\sgn w\cdot
e^{w(\lambda+\rho)}.$$
By above, 
$\ds_x(\cE_{\lambda}^-)=0$ (since $\DS_x(L(\lambda))=0$).
In particular, $\ds_x(\cE_{\lambda}^-)=0$ if $\lambda\not\in\Lambda^+_{m|n}$.

\subsubsection{Weight $\lambda'$}
If $\tail(\lambda)\geq j$
for $\osp$ or $\tail(\lambda)>j$
for $\fgl$, then 
$$\tail(\lambda')=\tail(\lambda)-j.$$
Thus $\lambda\mapsto\lambda'$ corresponds to the ``tail-cutting''.
Moreover, for $\fgl$-case $y_0(\lambda^{\dagger})=y_0((\lambda')^{\dagger})$ 
(see~\ref{dagger} for notation); in other words, the diagram
of $(\lambda')^{\dagger}$ is obtained from the diagram of $\lambda^{\dagger}$
by cutting $j$ elements from the ``vertical tail''; for example,
$$\begin{array}{lcr}
f=\ >\times\times<\times\times>\circ\ldots & &
 f^{\dagger}=\ >\circ\circ<\circ\times^4>\circ\ldots\\
f'=\ >\circ\circ <\times\times>\circ\ldots & & (f')^{\dagger}=\ >\circ\circ <\circ\times^2>\circ\ldots\end{array}$$

\subsubsection{Remark}
Let $j$ be the rank of $x$. We take $x\in \sum_{\beta\in S_j}\fg_{\beta}$ and
identify $\fg_x$ with
a subalgebra of $\fg$ as in~\cite{DS}. In the $\osp$-case
$\lambda'=\lambda|_{\fh_x}$; 
for $\fgl$ this  holds if $\lambda$ is stable.

\subsubsection{Remark}
In the $\osp$-case
 the ``tail-cutting''   is
``reversible'' (we can reconstruct the tail if our dog is still alive):
the weight diagram of $\lambda+\rho$ is obtained by adding $j$ symbols
$\times$ to the zero position in the diagram of $\lambda'+\rho'$. Therefore
$\ds_j(\mathcal{E}^-_{\lambda})=\ds_j(\mathcal{E}^-_{\nu})\not=0$ implies $\lambda=\nu$.
(The same holds for $\fgl$-case if $\tail(\lambda)>j$.) 

This gives the following corollary.

\subsubsection{}
\begin{cor}{}
In the $\osp$-case $\{\mathcal{E}^-_{\lambda}|\ tail(\lambda)\leq j\}$ is a basis of
the kernel of $\ds_j$.
\end{cor}

\subsubsection{Remark}
Take $\lambda\in\Lambda^+_{m|n}$ which is assumed to be stable for $\fgl$-case. Using the notation of~\ref{termsE} we 
 introduce
\begin{equation}\label{Elambdai}
\mathcal{E}_{\lambda,i}:=R^{-1}e^{-\rho}
\jJ_W\bigl(\frac{e^{\lambda+\rho}}{\prod_{\alpha\in\Delta(\fp^{(i)})_1^-}
(1+e^{\alpha})}\bigr),\ \ \  \cE_{\lambda,i}^-:=(-1)^{p(\lambda)}\pi(\cE_{\lambda,i}).  \end{equation}
Note that $\cE_{\lambda}=\cE_{\lambda,\tail(\lambda)}$.
Take $\lambda\in\Lambda^+_{m|n}$  which is assumed to be stable in the $\fgl$-case and retain notation of~\Thm{thmcut}. If $\tail(\lambda)\geq j$
for $\osp$ or $\tail(\lambda)>j$ for $\fgl$, then
$\cE^-_{\lambda', \tail(\lambda)-j}=\cE^-_{\lambda'}$.
In the $\fgl$-case with $\tail(\lambda)=j$ one has $\cE^-_{\lambda',0}=\sch K(\lambda')$.

\subsubsection{}\label{core}
Let $\chi$ be a central character of atypicality $j=rank\ x$.
Let $\nu$ be the weight of $\fg_x$ 
with the diagram equal to the diagram of $\chi$. This means that
for each $\lambda\in\Lambda^+_{m|n}$ with $\chi_{\lambda}=\chi$ one has
 $\nu=\lambda|_{\fh_x}$
 (the diagram of $\nu$ is obtained from the diagram
of $\lambda$ by removing all symbols $\times$). Note that
$\nu$ is a typical weight, so
 $\mathcal{E}^-_{\nu}=\sch L(\nu)$. We put $L^{core} = L(\nu)$ in the $\mathfrak{gl}$ and $\mathfrak{osp}(2m+1|2n)$-case and \[ L^{core} := \begin{cases} L(\nu) & \text{ if } \nu \text{ is } \sigma- \text{invariant} \\ L(\nu) \oplus L(\nu)^{\sigma} & \text{ else } \end{cases} \] in the $\mathfrak{osp}(2m|2n)$-case. The notion of $L^{core}$ was first introduced in  \cite{GS}, but there $L^{core}$ always equals $L(\nu)$ and therefore differs from ours in the $\mathfrak{osp}(2m|2n)$ in case $\nu$ is not $\sigma$-invariant.

\subsubsection{}
Extend $\sdim$ to a linear function on the Grothendieck ring $\Ch_{\xi}(\fg)$.
Clearly, $\sdim$ gives a linear function
on $\Sch(\fg)$. 

\begin{cor}{corsdim}
For $\lambda\in\Lambda^+_{m|n}$ 
$\sdim \cE^-_{\lambda}=\sdim (L(\lambda))^{\core}$ if $\lambda$ is a Kostant weight and $L(\lambda)$ has the maximal atypicality; $\sdim \cE^-_{\lambda}=0$ 
for other weights.
\end{cor}
\begin{proof}
Since $\sdim(\DS_x(N))=\sdim N$ (see~\cite{DS},\cite{Skw}),
the homorphism $\ds_x$ preserves $\sdim$. Take $x$ of the maximal rank ($=\min(m,n)$). By~\Thm{thmcut}, $\ds_x(\cE^-_{\lambda})=0$ if $\tail(\howl(\lambda))=\tail(\lambda)<\min(m,n)$.
Hence $\ds_x(\cE^-_{\lambda})\not=0$ implies
$\tail(\howl(\lambda))=\min(m,n)$ which means that
$\lambda$ is a Kostant weight and $L(\lambda)$ has the maximal atypicality. 
Now let $\lambda$ be a Kostant weight and $L(\lambda)$ has 
the maximal atypicality. The algebra $\fg_x$ is either a Lie algebra
($\fgl_{|m-n|}, \mathfrak{o}_{2(m-n)},\mathfrak{o}_{2(m-n)}, \mathfrak{sp}_{2n-2m}$ or $\osp(1|2(n-m))$ and $L^{core}$ is a $\fg_x$-module with
$\sch L^{\core}=\cE^-_{\lambda'}$ except for the case when
$\fg_x=\mathfrak{o}_{2(m-n)}\not=0$ and 
$\sch L^{\core}=\cE^-_{\lambda'}+\cE^-_{(\lambda')^{\sigma}}$.
\end{proof}

\subsubsection{}
\begin{cor}{image-of-ds}
Take $\lambda \in \Lambda_{m|n}^+$. If the rank of $x$ is equal to the atypicality of $\chi_{\lambda}$, then
 $$\ds_x L(\lambda)=m(\lambda) \ [L^{core}]$$
where $m(\lambda)$ is  equal to the number of
increasing paths from $\lambda$ to the weights with adjacent $\times$'s (see Section \ref{kostant}) and $[L^{core}]$ denotes the class of $L^{core}$ in the supercharacter ring.
\end{cor}

\subsection{Proof of \Thm{thmcut}}\label{proofcut}
Using (\ref{DSxj})  we reduce  the assertions to the case $j:=1$.
 Set $s:=\tail(\lambda)$.

\subsubsection{}\label{cutstable}
First, we consider the $\osp$-case and the $\fgl$-case with a stable
weight $\lambda$.

Take $\beta_0\in S_1$ ($\beta_0=\pm(\vareps_m-\delta_n)$ for $\osp$-case
and $\beta_0=\vareps_m-\delta_1$
for $\fgl(m|n)$).

We take $x\in\fg_{\beta_0}$. Set $\fg':=\DS_x(\fg)$.
By~\cite{DS} (and~\cite{Sgrs}) we can 
identify $\fg'$ with a subalgebra of $\fg$ such that $\fh\cap\fg'$ is
a Cartan subalgebra of $\fg'$ and a base $\Sigma'$ for $\Delta(\fg')$ satisfies
\begin{equation}\label{BeBe}
\Delta^+=\Delta^+({\Sigma}')\coprod \{\beta_0\} \coprod B\coprod
\{\alpha+\beta_0|\alpha\in B\}\end{equation}
for some $B\subset \Delta^+$. Let
$\rho'$ (resp., $R'$) be the Weyl vector (resp., denominator)
for $\fg'$ with respect to $\Sigma'$. 
As in~\ref{projpr} we define
$$\pr(e^{\nu})=c_{\nu}e^{\nu|_{\fh'}},$$
where  $c_{\nu}:=e^{-\pi i (\nu|\delta_q)}$ with 
$q=n$ for $\osp$-case, $q=1$ for $\fgl$-case  
($\beta_0|\delta_q)\not=0$). 
By~(\ref{BeBe})  one has
$$\rho-\rho'\in\mathbb{Z}\beta_0,\ \ \ \pr({R}(1+e^{-\beta_0}))={R}'.$$

By~\ref{A21} we have
$$\ds_x(\mathcal{E}^-_{\lambda})=(\pi\pr\pi)(\mathcal{E}^-_{\lambda})
=(-1)^{p(\lambda)}(\pi\pr)(\mathcal{E}_{\lambda}),$$
which allows to rewrite the required formula as follows:
$\pr(\mathcal{E}_{\lambda})=0$ if $s=0$ and 
\begin{equation}\label{17}
(-1)^{p(\lambda)-p(\lambda')}\pr(\mathcal{E}_{\lambda})=\left\{\begin{array}{ll}
\mathcal{E}_{\lambda'}& \text{if }\ s>1;\\
\mathcal{E}_{\lambda'}& \text{if }\ s=1,\ \
\fg=\osp(2m+1|2n);\\
\mathcal{E}_{\lambda'}&  \text{if }\ s=1,\ \
\fg=\osp(2|2n)\\
\ch K(\lambda') & \text{if }\
s=1,\ \fg=\fgl(m|n),\\
\mathcal{E}_{\lambda'} & \text{ for }\ \osp(2m|2n), m>s=1,\  \  (\lambda')^{\sigma}=\lambda',\\
\mathcal{E}_{\lambda'}+\mathcal{E}^-_{(\lambda')^{\sigma}}& \text{ for }\
\ \osp(2m|2n), m>s=1,\  (\lambda')^{\sigma}\not=\lambda'.
\end{array}\right.\end{equation}

 In the $\fgl$-case take $\lambda^{\dagger}$ 
as in~\ref{dagger}; in the  $\osp$-case we have
 $\lambda^{\dagger}=\lambda+\rho$. Using~\Prop{propE} we get
\begin{equation}\label{dsxx}
j_s\pr(\mathcal{E}_{\lambda})=
\pr\bigl(R^{-1}e^{-\rho}\KW(\lambda^{\dagger},S_s)\bigr)=c_{-\rho}
\cdot (R'e^{\rho'})^{-1}\pr
\bigl((1+e^{-\beta_0})\KW(\lambda^{\dagger},S_s)\bigr).
\end{equation}

For $s=0$ the formula~(\ref{prs0}) gives $\pr(\cE_{\lambda})=0$ as required.
From now on we assume $s>0$.
The pair $(\lambda^{\dagger},S_s)$ satisfies
the assumptions of \Prop{proppr}; using this proposition
and taking into account that for $\fgl$-case
$$R'e^{\rho'}\ch K(\lambda')=\KW(\lambda'+\rho',\emptyset)$$
we see that~(\ref{17}) holds up to a non-zero scalar $a_{\lambda}$ 
which can be computed directly. Instead of performing such computation
we can employ the following reasoning. One has
$$\sch \DS_x(L(\lambda))=\ds_x(\cE^-_{\lambda})+\sum_{\nu<\lambda}
d^{\lambda,\nu}_<\ds_x(\cE^-_{\nu}).$$

By above, $\ds_x(\cE^-_{\nu})$ is proportional to $\cE^-_{\nu'}$
(or to $\cE^-_{\nu'}+\cE^-_{(\nu')^{\sigma}}$ for $\osp(2m|2n)$),
where $\nu':=\nu|_{\fh'}$. By~\Cor{cormaxsupp}
$$\supp(\cE^-_{\nu'})\subset \nu'-\mathbb{N}\Sigma',$$
where $\cE^-_{\nu'}\ $ is viewed as element of $\cR_{\Sigma'}$, see~\ref{cR}.
The inequality $\nu<\lambda$ means that $\lambda-\nu\in\mathbb{N}\Sigma$ which implies
$\nu'\in \lambda'-\mathbb{N}\Sigma'$  by~(\ref{BeBe}). Hence
the coefficient of $e^{\lambda'}$ in 
$\ds_x(\cE^-_{\lambda})$ is equal to $\sdim \DS_x(L(\lambda))_{\lambda'}$.
Using the same reasoning for the formula
$$\sch L(\lambda')=\sum_{\nu'}
d^{\lambda',\nu'}_<\cE^-_{\nu'}$$
we conclude the coefficient of  $e^{\lambda'}$ in $\cE^-_{\lambda'}$ is $1$.
Combining $(\lambda,\beta_0)=0$ and $\beta_0\in\Sigma$ one readily sees that 
$ \DS_x(L(\lambda))_{\lambda'}=\mathbb{C}$, so 
$\sdim \DS_x(L(\lambda))_{\lambda'}=1$. Hence  the coefficients of $e^{\lambda'}$ in 
$\ds_x(\cE^-_{\lambda})$ and in $\cE^-_{\lambda'}$ are equal, so
 $a_{\lambda}=1$.

\subsubsection{}
Consider the case when $j=1$ and $\fg=\fgl(m|n)$. If $\lambda$ is stable,
the required formula is established in~\ref{cutstable}.  
Using the fact that $\DS_x$ commutes with translation functors,
we deduce from the stable case  the required formula for the non-stable case
taking into account~\Cor{corEla} and~\ref{RemEla} for $s>1$ and $s=1$ respectively. 
\qed

\section{Superdimensions and modified Superdimensions} \label{sec:sdim}

We discuss modified nontrivial trace and dimension functions on the thick ideal $I_k$ generated by the irreducible representations of atypicality $k$, and how they can be calculated explicitely by means of the Duflo-Serganova functor. We do this for the $\fosp(m|2n)$ and the $OSp(m|2n)$-case. For the $\fgl$-case see \cite{HW}.  %Recall from \cite{Skw} that the sign of the superdimension $\sdim(L(\lambda))$ is $(-1)^{p(\lambda)}$. Since we assume that the highest weight vector is always even, this sign is always $1$.

\subsection{The core of a block}

Recall that $\tilde{\mathcal{F}} = Rep(SOSp(m|2n))$ and that we have a decomposition  $\tilde{\F} = \F \oplus \Pi \F$ into two subcategories which are equivalent by the parity shift $\Pi$. We use the notation: \[ \F' = \F'(m|2n) = Rep(OSp(m|2n))\] for the finite dimensional algebraic representations of $OSp(m|2n)$. As for $\F$ the category decomposes $\F' = \mathcal{F}' \oplus \Pi \mathcal{F}'$ into two equivalent subcategories.

The irreducible typical module $L^{core}$ (as defined in~\ref{core}) attached to a block of atypicality $k$ in $\tilde{\mathcal{F}}$ is both an $\mathfrak{osp}(m-2k|2n-2k)$ and $OSp(m-2k|2n-2k)$-module. Therefore the core can be defined in the $\F'$-case as well.

The $DS_x$ functor on $\tilde{\mathcal{F}}$ induces a functor \[ DS_x: \F'(m|2n) \to  \F'(m-2k|2n-2k), \] where $k = rk(x)$ (see \cite{CH} for details).

If the rank of $x$ equals $\deff(\mathfrak{osp}(m|2n))$  we obtain
\begin{itemize}
\item $\mathfrak{g}_x = \mathfrak{o}(m-2n|0), \ \ m > 2n$;
\item  $\mathfrak{g}_x = \mathfrak{sp}(0|2n-2m), \ \ 2n > m \text{ even}$;
\item  $\mathfrak{g}_x = \mathfrak{osp}(1|2n-2m), \ \ 2n - m \text{ odd}$.
\end{itemize}

In the $OSp$-case we obtain representations of the groups $G_x = O(m-2n)$, $Sp(2n-2m)$ (considered as odd) and $OSp(1|2n-2m)$.

\subsection{Superdimensions}

If we apply $DS_x$ to an irreducible representation $L(\lambda)$ with atypicality equal to
 $rk(x)$, 
then $DS_x(L(\lambda))$ doesn't depend on the choice of $x$. Indeed the induced morphism on the supercharacter ring doesn't depend on $x$ and $\DS_x(L(\lambda))$ is semisimple. We simply write $\DS_k$ in this case.

%If $L(\lambda) \in \tilde{\F}$ has atypicality $k$, then by \cite[Corollary 2.2]{Skw} \begin{align*} DS_k (L(\lambda)) & = m(\lambda) L^{core} \oplus m{'}(\lambda) \Pi L^{core} \text{ for } \mathfrak{osp}(2m+1|2n) \\ DS_k (L(\lambda)) & = m(\lambda) L^{core} \oplus m{'}(\lambda) \Pi L^{core}  \\ & \oplus m^{''}(\lambda) (L^{core})^{\sigma} \oplus m^{'''}(\lambda) \Pi (L^{core})^{\sigma}))  \text{ for } \mathfrak{osp}(2m|2n) \end{align*} for certain positive integers $m(\lambda)$, $m^{'}(\lambda)$, $m^{''}(\lambda)$, $m^{'''}(\lambda)$.

%Since $||\nu|| = 0$ if $\nu$ is the weight of $L^{core}$, the parity rule of \cite{GH} yields \[ Iso(DS_k(L(\lambda)): L(\nu)) \in \Pi^{||howl(\lambda)||} \ \F(\mathfrak{g}_x) \] and therefore that $DS_k (L(\lambda))$ is either purely in $\F$ or purely in $\Pi \F$ and hence

The parity rule of \cite{GH} yields \[ \DS_k(L(\lambda))\in \Pi^{||howl(\lambda)||} \ \mathcal{F}(\mathfrak{g}_x) \] and hence

\begin{align*} \DS_k (L(\lambda)) & =  \Pi^{||howl(\lambda)||} (L^{core})^{\oplus m(\lambda)}.
\end{align*} for the positive integer $m(\lambda)$ defined in Corollary \ref{image-of-ds} (the number of
increasing paths from $\lambda$ to the weights with adjacent $\times$'s).

\subsubsection{$OSp$-modules}\label{OSP2m2n}
We first consider $\fg=\osp(2m|2n)$.
By~\cite[Proposition 4.11]{ES} the simple $OSp(2m|2n)$-modules are
either of the form $L(\lambda)$ if $\lambda\in\Lambda^+_{m|n}$ is $\sigma$-invariant or $L(\lambda)\oplus L(\lambda^{\sigma})$.
Thus the simple $OSp(2m|2n)$-modules are
in one-to-one correspondence with the unsigned
$\osp(2m|2n)$-diagrams. For $\osp(2m+1|2n)$ and any $\lambda \in \Lambda_{m|n}^+$ there are two irreducible $OSp(2m+1|2n)$-modules $L(\lambda,+)$ and $L(\lambda,-)$ which restrict to $L(\lambda)$. We will often simply write $L_{OSp}(\lambda)$ for an irreducible representation of $OSp$. The diagram 

\[ \xymatrix@C=1em@R=1em{ \cF'(G) \ar[dd]_{DS_x} \ar[dr]^{Res} & \\ 
& \tilde{\cF}(\mathfrak{g}) \ar[dl]^{DS_x} \\ \cF'(G_x) & } \] 

commutes for any $x$ since $DS_x(L(\lambda))$ is in $\cF'(G_x)$. It follows from this diagram that the multipliciy of $L^{core}$ in $DS_x(L_{OSp}(\lambda)$ is the same as for $Res(L_{OSp}(\lambda))$ if the restriction is irreducible and is twice the multiplicity of $DS(L(\lambda))$ if the restriction decomposes into two irreducible summands.

%As for $\mathfrak{osp}$ we obtain \[ DS_k (L_{OSp}(\lambda))  = m(\lambda) \Pi^{||howl(\lambda)||} L_{OSp}^{core}.\]

Since $DS$ is a symmetric monoidal functor it preserves the superdimension.

\subsubsection{}
\begin{cor}{} For $L(\lambda) \in \tilde{\mathcal{F}}$ of atypicality $k$ \[ \sdim L(\lambda) = (-1)^{||howl(\lambda)||} m(\lambda) \sdim L^{core}. \] In particular $\sdim L(\lambda) \neq 0$ if and only if $\lambda$ is maximal atypical.
\end{cor}

\subsection{Examples}\label{exaDS}

We list some low-rank examples of $DS(L(\lambda))$ (up to parity shift) as well as the superdimension of $L(\lambda)$ for $\osp(4|4)$ and $\osp(6|6)$. These can be computed by counting the increasing paths to the zero weight using Corollary \ref{image-of-ds}. The number $m(\lambda)$ can also be computed via the formalism of arc diagrams in \cite{GH}. %We abbreviate the notation such that $L(\lambda)$ is simply denoted $(\lambda)$.

\subsubsection{$\DS$ for $\osp(4|4)$  }
$$\begin{array}{llllll}
\lambda& (0,0);\ & (i-1,i);\ & (0,2);\ & (0,i); \ \ 2<i &  (i,j)\ \ \ 0<i<j-1;\\
\DS_1(L(\lambda))& (0); &   (i);     & (0) \oplus (2); & (0)^{\oplus 2}\oplus (i); &  (i)\oplus (j)\\
\sdim & 1 & 2 & 3 & 4 & 4
\end{array}$$

\subsubsection{$\DS$ for $\osp(6|6)$  }
$$\begin{array}{llllll}
\lambda&  (0,0, i)\ \ i\leq 3;\ & (0,0,4); & (0,0,i),\ \  4<i\\
\DS_1(L(\lambda))& (0,i); &  (0,0)\oplus(0,4); & (0,0)^{\oplus 2}\oplus (0,i);\\
\sdim & i+1 & 5 & 6
\end{array}$$

$$\begin{array}{lllllllll}
\lambda&  (0,1,2);\ & (0,1,3); & (0,1,4);& (0,1,i)\ \  4<i; \\
\DS_1(L(\lambda)) & (1,2) & (1,3) & (0,1)\oplus (1,4) &   (0,1)^{\oplus 2}\oplus (1,i)\\
\sdim & 4 & 8 & 10 & 12
\end{array}$$

$$\begin{array}{lllllllll}
\lambda & (0,2,3); & (0,2,4); & (0,2,i)\ \ 4<i\\
\DS_1(L(\lambda)) & (2,3)\oplus (0,3) & (2,4)\oplus (0,4)\oplus (0,2) &  (2,i)\oplus (0,4)\oplus (0,2)^{\oplus 2}\\
\sdim &  8 & 15 &18
\end{array}$$

$$\begin{array}{lllllllll}
\lambda & (0,3,4);  &(0,4,5); & (0,3,i) \ \ 4<i\\
\DS_1(L(\lambda)) & (3,4)\oplus (0,3)\oplus (0,4) &\\
\sdim &  12 & 12 & 20
\end{array}$$

For $i>0$
$$\begin{array}{lllllllll}
\lambda & (i,i+1,i+2); & (i,i+1,i+3); & (i,i+1,i+4) \ \ 4<i\\

\sdim &  4 & 8 & 12
\end{array}$$

\subsubsection{Other examples}
$\DS_1(L(0,1,\ldots,k))=L(1,\ldots,k)$ and $\sdim L(0,1,\ldots,k)=2^k$.

$\DS_1(L(0^i,1,\ldots,k))=L(0^{i-1},1,\ldots,k)$ and $\sdim L(0,0,1,\ldots,k)=2^{k}$.

\subsection{Modified Traces} In this section $\F$ means either $\tilde{\mathcal{F}}$ or $\tilde{\mathcal{F}}'$ unless otherwise specified.

If $\at(L(\lambda)) <n$, $\sdim(L) = 0$. However one can define a modified superdimension for $L$ as follows. Recall that a thick (tensor) ideal $I$ in $\tilde{\mathcal{F}}$ is a subset of objects which is closed under tensor products with arbitrary objects and closed under direct summands.
A trace on $I$ is by definition a family of linear functions
\[t = \{t_V:\End_{\F}(V)\rightarrow k \}\]
where $V$ runs over all objects of $I$ such that following two conditions hold.
\begin{enumerate}
\item  If $U\in I$ and $W$ is an object of $\F$, then for any $f\in \End_{\F}(U\otimes W)$ we have
\[t_{U\otimes W}\left(f \right)=t_U \left( t_R(f)\right) \] for the right trace $tr_R()$.
\item  If $U,V\in I$ then for any morphisms $f:V\rightarrow U $ and $g:U\rightarrow V$ we have
\[t_V(g\circ f)=t_U(f \circ g).\]
\end{enumerate}

For such a trace on $I$ we define \[\dim^I (X) = t_X(id_X), \ X \in I,\] the modified dimension of $(I,t)$. For an object $J \in \F$ let $I_J$ be the thick ideal generated by $J$. By Kujawa \cite[Theorem 2.3.1]{Kujawa-generalized} the trace on the ideal $I_{L}$, $L$ irreducible, is unique up to multiplication by an element of $\mathbb{C}$.

\subsection{The generalized Kac-Wakimoto conjecture}

Let $I_k$ be the thick ideal generated by all irreducible representations of atypicality $k$. The ideal $I_0$ coincides with $Proj$. The following theorem was proven for $\mathfrak{gl}(m|n)$ by Serganova \cite{Skw} and for $\mathfrak{osp}(m|2n)$ by Kujawa \cite{Kujawa-generalized}.
We give a slightly different simplified proof. Moreover we explain how to compute these modified superdimensions.

\subsubsection{}
\begin{thm}{} (Generalized Kac-Wakimoto conjecture) The ideal $I_k$ admits a non-trivial modified trace function. For irreducible $L(\lambda)$ the associated dimension function
$\dim^k$ satisfies $\dim^k L(\lambda) \neq 0$ if and only if the atypicality of $L(\lambda)$ is $k$.
\end{thm}

It was shown in \cite[Theorem 1.3.1]{Geer-Kujawa-Patureau-Mirand} that if an ideal $I$ carries a modified trace function, all indecomposable objects in $I$ are ambidextrous in the sense of \cite{Geer-Kujawa-Patureau-Mirand}. Since the $I_k$ define an exhaustive filtration of $\F$, the conjecture implies that every simple module in $\F$ is ambidextrous.

\subsection{A trace on $I_k$}

There are two different ways to see that $Proj \subset \F'$ carries a nontrivial trace function. It was proven in \cite[Theorem 4.8.2]{GKPM2} that $Proj \subset \F$ has such a trace function. This implies that $Proj \subset \F'$ has one as well using the restriction rules of Ehrig-Stroppel and the argument of \cite{Kujawa-generalized}.

Alternatively it follows from \cite{Heidersdorf-Wenzl-Deligne} that $Proj \subset  \F'$ carries such a trace function. Note that it is unique up to a scalar: Any $P \in Proj$ satisfies $<P> = Proj$. Indeed, $<P> \subset Proj$ is clear, and $Proj \subset P$ follows since $Proj$ is the smallest thick ideal \cite{CH}. We denote any normalization of this trace function by $Tr^0$.

\subsubsection{}
\begin{prop}{} The thick ideal $I_k \subset \F$ carries a nontrivial modified trace function $Tr^k$.
\end{prop}
\begin{proof} Let $L(\lambda) \in \F$. Then $DS_x(L(\lambda)) \in \F'$ for all $x$ by \cite{GH}. Let $X \in I_k$ and $f \in End(X)$. Then we define \[ Tr^k(f) = Tr^0 (DS_k (f)).\] Then $DS_k(X)$ is typical and therefore projective. Since $DS_k$ is a symmetric monoidal functor this defines a trace function. We claim that it is nontrivial. For $X = L(\mu) $ we obtain \[ DS_k(f) \in End(\Pi^{||howl(\mu)||} 
( L^{core})^{\oplus m(\lambda)}). \] Since the parity is either even or odd and $Tr^0$ is nontrivial for any typical module, we compute for $f \in  End(X)$  \begin{align*} Tr^k (id_L) & = Tr^0_{DS_k(L)} (id_{DS_k(L)}) \\ & = m(\lambda) Tr^0_{\Pi^{||howl(\mu)||} L^{core}}(id_{\Pi^{||howl(\mu)||} L^{core}}) \neq 0.\end{align*} The same proof works for $L_{OSp}(\lambda)$.
\end{proof}

\subsubsection{}
\begin{rem}{} It can be shown \cite{Kujawa-generalized} that $I_k$ is in fact generated by an arbitrary irreducible representation of atypicality $k$. Therefore the above trace is the unique modified trace up to a scalar.
\end{rem}

Since $\DS_k(L) = 0$ for any $L$ of atypicality $< k$, we obtain for the modified superdimension $\sdim^k (X) := Tr^k(id_X)$

\subsubsection{}
\begin{cor}{} Let $L(\lambda)$ be a representation of atypicality $\leq k$. Then $\sdim^k (L(\lambda)) \neq 0$ if and only if $at(L(\lambda))  = k$.
\end{cor}

\appendix

\section{Kac-Wakimoto terms and the rings $\cR,\cR_{\Sigma'}$}
\label{sec:RR}

In this section $\fg$ is $\fgl(m|n),\osp(M|N)$ or one of the exceptional
Lie superalgebras $F(4), G(3), D(2|1,a)$. We use the standard notation for the roots of $\fg_0$ and denote by
$\Pi_0$  a standard set of simple roots. In what follows 
we consider only bases $\Sigma$ of $\Delta$ which are compatible with $\Pi_0$, 
that is $\Delta^+(\Sigma)_0=\Delta^+(\Pi_0)$. 
By~\cite{Sgrs}, all such bases are connected by chains of odd reflections. In the $\fgl$ and $\osp$-cases these bases can be encoded by
words consisting of $m$ letters $\vareps$ and $n$ letters
$\delta$.

\subsection{Notation}\label{JW}
We denote by $W$  the Weyl group 
of $\fg_0$. We denote by $\Delta$ the set of roots of $\fg$ and set
$$\begin{array}{l}
\fh_{int}^*:=\{\lambda\in\fh^*|\ \forall w\in W\ \  \lambda-w\lambda\in\mathbb{Z}\Delta\},\\
P(\fg_0):=\{\lambda\in \fh^*|\ \forall w\in W\ \  \lambda-w\lambda\in\mathbb{Z}\Delta_0\}.\end{array}$$

For each non-isotropic root $\alpha$ let $r_{\alpha}\in W$ be the
 reflection  with respect to $\alpha$. 
For any subset $Y\subset W$ we denote by $\jJ_Y$ the linear operator
$P\mapsto \displaystyle\sum_{w\in Y} \sgn(w) w(P)$, where
$\sgn: W\to \mathbb{Z}_2$ is the standard sign homorphism 
(given by $\sgn r_{\alpha}=-1$).

\subsubsection{Choice of the Weyl vector}\label{Weylvector} 
We denote by $\rho_0$ a Weyl vector of $\fg_0$ which is 
an element of $\fh^*$ satisfying 
$$r_{\alpha}\rho_0=\rho_0-\alpha$$ for each 
$\alpha\in\Pi_0$.  Note that
$\rho_0$ is unique if $\Delta_0$ spans $\fh^*$, i.e. for $\fg\not=\fgl(m|n),\osp(2|2n)$. We choose the Weyl vector $\rho$ by the rule
$$\rho:=\rho_0-\rho_1,\ \ \rho_1=\frac{1}{2}\sum_{\alpha\in\Delta^+_1}\alpha.$$
If $\beta\in\Sigma$ is isotropic and
$\Sigma'=r_{\beta}\Sigma$ we have $\rho':=\rho+\beta$. Using~\cite{Sint}
(or a short case-by-case reasoning)
we obtain $\rho\in\fh^*_{int}$. We introduce
$$R_0:=\prod_{\alpha\in\Delta^+_0}(1-e^{-\alpha}),\ \
R_1(\Sigma):=\prod_{\alpha\in\Delta^+_1(\Sigma)}(1+e^{-\alpha}),\ \ R(\Sigma):=\frac{R_0}{R_1(\Sigma)}.$$
Note that the  term
$$e^{\rho_0-\rho}R_1(\Sigma)=\prod_{\alpha\in\Delta^+_1(\Sigma)}(e^{\alpha/2}+e^{-\alpha/2})$$
is $W$-invariant and does not depend on the choice of $\Sigma$. Hence
for each $\Sigma'$ satisfying $\Delta^+_0\subset \Delta^+(\Sigma)$ we have
$$R(\Sigma')e^{\rho'}=Re^{\rho},\ \ \text{ where }R:=R(\Sigma).$$

\subsection{Rings $\cR$ and $\cR_{\Sigma}$}\label{cR}
For a sum of the form $\sum_{\nu\in\fh^*} a_{\nu}e^{\nu}$ with $a_{\nu}\in\mathbb{Q}$
we define the support by the formula
$$\supp (\sum a_{\nu}e^{\nu})=\{\nu\in\fh^*|\ a_{\nu}\not=0\}.$$

Let $\cR_{\Sigma}$ be the set constisting of the sums $\sum_{\nu\in\fh^*} a_{\nu}e^{\nu}$ with $a_{\nu}\in\mathbb{Q}$ and such that
$$\supp (\sum a_{\nu}e^{\nu})\subset \cup_{i=1}^k (\nu_i-\mathbb{N}\Sigma')$$
for some $k$. Clearly, $\cR_{\Sigma}$ is a ring.
This ring contains $\ch N$ and $\sch N$ for any $N$ in the BGG-category
$\CO$.

\subsubsection{}
Denote by $\cR$ the ring of rational functions
of the form $\frac{P}{Q}$, where $P$ lies in the group ring 
$\mathbb{Q}[\fh^*]$ and $Q$ is  a product of the factors
of the form $1\pm e^{-\alpha}$ for $\alpha\in\Delta$. Using the formula
$$1\pm e^{-\alpha}=1\mp e^{-\alpha}+e^{-2\alpha}\mp e^{-3\alpha}+\ldots$$
we will view the element $\frac{P}{Q}\in\cR$ as  a series $\cR_{\Sigma}$; we will call this series the $\Sigma$-expansion of $\frac{P}{Q}$.
For instance, $R(\Sigma'),R(\Sigma')^{-1}\in\cR$ for any base $\Sigma'$ 
and $\Sigma$-expansion
of $R(\Sigma')^{-1}$  is equal to the character 
of a Verma module of the highest weight $0$ (defined with respect to 
the base $\Sigma$).

\subsubsection{}
\begin{lem}{Delta1R}
For any base $\Sigma'$ satisfying $\Delta^+(\Sigma')_0=\Delta^+(\Sigma)$ one has
$$\jJ_W\bigl(\frac{e^{\rho'}}{\prod_{\alpha\in\Delta^+_1(\Sigma')}(1+e^{-\alpha})}\bigr)=Re^{\rho}.$$
\end{lem}
\begin{proof}
By above, $e^{\rho_0-\rho'}R_1(\Sigma')$
is $W$-invariant so
$$\jJ_W\bigl(\frac{e^{\rho'}}{\prod_{\alpha\in\Delta^+_1(\Sigma')}(1+e^{-\alpha})}\bigr)=\jJ_W\bigl(\frac{e^{\rho'}}{R_1(\Sigma')}\bigr)=\frac{\jJ_W(e^{\rho_0})}{e^{\rho_0-\rho'}R_1(\Sigma')}\ .$$
The Weyl character formula for the trivial $\fg_0$-module gives $\jJ_W(e^{\rho_0})=R_0e^{\rho_0}$. Using the above identity
$R(\Sigma')e^{\rho'}=Re^{\rho}$ we obtain the required formula.
\end{proof}

\subsection{Projection $P_{\chi}$}
\label{trans}
 Let $\CO^{\chi}$ be the full subcategory of the category $\CO$ corresponding to a central character $\chi$. For $N\in\CO$ let $N^{\chi}$ be the projection of $N$ to $\CO{\chi}$. The character of $N^{\chi}$ can be expressed via $\ch N$ by the following procedure.

By above, $\mathcal{R}_{\Sigma'}$ contains the terms $\ch N, Re^{\rho}\ch N$ 
for any module $N\in \CO$. It is well-known that 
for $N\in\CO^{\chi}$ the $\Sigma'$-expansion of 
$Re^{\rho}\ch N$ satisfies
$$\supp (Re^{\rho}\ch N)\subset \{\mu+\rho|\ \chi_{\mu}=\chi_{\lambda}\}.$$
Introducing a projection $P_{\chi}: \cR_{\Sigma'}\to\cR_{\Sigma'}$   by
$\ P_{\chi}\bigl(\sum a_{\mu}e^{\mu}\bigr)=\displaystyle\sum_{\mu: \chi_{\mu-\rho}=\chi} a_{\mu}e^{\mu}$
we get
\begin{equation}\label{Pchi}
Re^{\rho}\ch N^{\chi}=P_{\chi}(Re^{\rho}\ch N).
\end{equation}

\subsubsection{}\label{ThetaV}
For a finite-dimensional module $V$
a translation functor $T_{\chi,\chi'}^V: \CO^{\chi}\to \CO^{\chi'}$ is given by
$T_{\chi,\chi'}^V(N):=(N\otimes V)^{\chi'}$. 
By above,
$$Re^{\rho} \ch (T_{\chi,\chi'}^V(N))=\Theta_{\chi,\chi'}^V(Re^{\rho}\ch N),$$
where $\Theta_{\chi,\chi'}^V:\ \mathcal{R}_{\Sigma'}\to \mathcal{R}_{\Sigma'}$
is given by
$$\Theta_{\chi,\chi'}^V(\sum a_{\mu}e^{\mu}):=P_{\chi'}\bigl(\ch V\cdot P_{\chi}(\sum a_{\mu}e^{\mu})\bigr).$$

\subsection{The terms $\KW(\lambda,S)$}\label{notatcut}
We say that a  subset $S\subset\Delta_1$ is an {\em iso-set}
if  $S$ is a basis of an isotropic subspace of $\fh^*$, i.e.
$S$ is linearly independent and $(S|S)=0$.

For $\lambda\in\fh^*_{int}$ and an iso-set $S\subset\Delta_1$ satisfying $(\lambda|S)=0$
we set
$$\KW(\lambda,S):=\jJ_W\bigl(\frac{e^{\lambda}}{\displaystyle\prod_{\beta\in S}(1+e^{-\beta})}\bigr).$$

\subsubsection{Remark}
For an arbitrary weight $\lambda\in\fh^*$ the group $W$ should be substituted 
by the ``$\lambda$-integral'' subgroup, see~\cite{GK}, Section 11.)

\subsubsection{}
Note that $\ \KW(\lambda,S)\cdot\displaystyle\prod_{\alpha\in\Delta^+_1}(1+e^{-\alpha})
\in S(\fh^*)$, so $\KW(\lambda,S)\in\cR$.
One readily sees that for the $\Sigma$-expansion of $\KW(\nu,S)$ we have
$$\supp \KW(\nu,S)\subset W(\nu+\mathbb{Z}S).$$
By~\cite{Sinv},\cite{Kcentre}, $\chi_{\mu-\rho}=\chi_{\nu-\rho}$
for each  $\mu\in \nu+\mathbb{Z}S$. Thus for $P_{\chi}$ introduced in~\ref{trans}
we have

\begin{equation}\label{PchiKW}
P_{\chi}(\KW(\lambda+\rho,S))=\delta_{\chi,\chi_{\lambda}}\KW(\lambda+\rho,S).
\end{equation}

\subsubsection{}
For $\fg=\fgl(m|n),\osp(M|N)$ 
one has $S=\{\pm \vareps_{p_i}\pm \delta_{q_i}\}_{i=1}^t$, where $p_i\not=p_j$, $q_i\not=q_j$ for
$i\not=j$. We denote the intersection of $\mathbb{Z}\Delta$ with the span of $\vareps_{p_1},\ldots,\vareps_{p_t},
\delta_{q_1},\ldots,\delta_{q_t}$ by $\fh(S)^*$. Notice that
 $S$ spans a maximal isotropic subspace in $\fh(S)^*$.

\subsubsection{}
\begin{lem}{lemprlambdaS}
(i) $w\KW(\lambda,S)=\sgn(w) \KW(\lambda,S)=
\KW(w\lambda,wS)$;

(ii) $\KW(\lambda,S)=0$ if there exists $\alpha\in\Delta_0$ such that
$(\alpha|S)=(\alpha|\lambda)=0$;

(iii)  $\KW(\lambda-\beta,S)=\KW(\lambda,(S\cup\{-\beta\})\setminus\{\beta\})$
for each $\beta\in S$;

 (iv) in the $\osp$-case if $(\lambda|\fh(S)^*)=0$, then
 $$\KW(\lambda-\beta,S)=-\KW(\lambda,S)\ \text{ for each }\beta\in S.$$
 \end{lem}\begin{proof}
(i), (iii) are straightforward and
  (ii) follows from (i) for $w:=r_{\alpha}$.
 For (iv) note that
 $$\KW(\lambda,S)+\KW(\lambda-\beta,S)=\jJ_W \bigl(\frac{e^{\lambda}+e^{\lambda-\beta}}{\prod_{\beta\in S}(1+e^{-\beta})}\bigr)=
\KW(\lambda,S\setminus\{\beta\}).$$
Since
 $\beta=\pm\vareps_i\pm\delta_j$ for some $i,j$
we have  $(\lambda|\delta_j)=
 (S\setminus\{\beta\}|\delta_j)=0$ so (iii) gives
$$\KW(\lambda;S\setminus\{\beta\})=0$$
 as required.
 \end{proof}

\subsubsection{Denominator identity}\label{denom}
Let $S$ be an iso-set of the cardinality $\min(m,n)$ and let
$\Sigma'$ be a base of $\Delta$ containing $S$ (for instance,
$\Sigma'=\Sigma$ for $\osp$-case and
$\Sigma'$ corresponding to $(\vareps\delta)^{m}\delta^{n-m}$ for $n\geq m$).
By~\cite{G1} one has $\KW(\rho', S)=jRe^{\rho}$, where $j$ is a certain integer
($j$ is the order of the ``smallest factor'' in $W$, for instance,
$j=m!$ for $\fgl(m|n)$ with $m\leq n$).

Consider the case $\fg=\fgl(s|s)$ or $\osp(2s+t|2s)$. Then 
$j=s!$ for $\fgl(s|s)$, $j=\max(2^{s-1} s!,1)$ for $\osp(2s|2s)$, and
 $j=2^s s!$ for $\osp(2s+t|2s)$ with $t=1,2$. Let
 $\Sigma'$  be the base corresponding  to the word
$(\vareps\delta)^s$; this base contains an iso-set
$\{\vareps_{i}-\delta_{i}\}_{i=1}^{s}$.
Note that $w\rho'=\rho'$ for any $w\in S_s\times S_s$;
using~\Lem{lemprlambdaS} (i) we obtain
$$jRe^{\rho}=\KW(\rho', \{\vareps_{i}-\delta_{i}\}_{i=1}^{s})=(-1)^{[\frac{s}{2}]}
\KW(\rho', \{\vareps_{i}-\delta_{s+1-i}\}_{i=1}^{s}).$$

%
%We fix a set of simple roots $\Delta^+_0\subset\Delta_0$ and 
%denote by $\rho_0$ the Weyl vector of $\fg_0$:
%$$2\rho_0=\sum_{\alpha\in\Delta^+_0}\alpha.$$ 
%We consider the bases
%of $\Delta$ which are compatible with $\Delta^+_0$.  These bases can be encoded by
%the  words consisting of $m$ letters $\vareps$ and $n$ letters $\delta$.

\subsection{The term $\frac{\KW(\lambda,S)}{Re^{\rho}}$}\label{KWR}
Recall that $L_{\fg_0}(\lambda-\rho_0)$ is finite-dimensional if and only if $\lambda\in  P^{++}(\fg_0)$, where
$$P^{++}(\fg_0):=
\{\lambda\in P(\fg_0)| \ \forall w\in W\ \lambda-w\lambda\in\mathbb{Z}_{\geq 0}\Delta^+,\ \lambda\not=w\lambda\}.$$

The character ring $\Ch(\fg_0)$  has a basis
$\{\ch L_{\fg_0}(\lambda-\rho_0)\}_{\lambda\in P^{++}(\fg_0)}$.
This allows to extend $\dim$ to the  linear map $\dim: \Ch(\fg_0)\to \mathbb{Z}$ having
$$\dim (\ch N)=\dim N$$ for any finite-dimensional module $N$.

The Weyl character and the Weyl dimension formulas give the following.

\subsubsection{}
\begin{lem}{}
Take $\lambda\in P(\fg_0)$. One has
\begin{enumerate} 
\item $\jJ_W(e^{\lambda})=0$ if and only if 
$\lambda\not\in W P^{++}(\fg_0)$;

\item
 $\frac{\jJ_W(e^{\lambda})}{R_0e^{\rho_0}}\in  \Ch(\fg_0)$;

\item 
$\dim \bigl(\frac{\jJ_W(e^{\lambda})}{R_0e^{\rho_0}}\bigr)=\prod_{\alpha\in\Delta^+_0}\frac{(\lambda|\alpha)}{(\rho_0|\alpha)}$.
\end{enumerate}
\end{lem}
\begin{proof}
If $\lambda\not\in W P^{++}(\fg_0)$, then $r_{\alpha}\lambda=\lambda$ for some 
$\alpha\in\Delta^+_0$ and thus $\jJ_W(e^{\lambda})=0$ and both sides of (iii)
are equal to zero. Now take $\lambda\in W P^{++}(\fg_0)$, that is
 $\lambda=w\nu$ for $\nu\in  P^{++}(\fg_0)$. Then $\jJ_W(e^{\lambda})=\sgn(w)\jJ_W(e^{\nu})$. Using the Weyl character 
formula we get
$$ (R_0e^{\rho_0})^{-1}\jJ_W(e^{\lambda})=\sgn(w) (R_0e^{\rho_0})^{-1}\jJ_W(e^{\nu})
=\sgn(w)\ch L(\nu-\rho_0)$$
which establishes (ii).
The Weyl dimension formula  gives
$$\dim \bigl(\frac{\jJ_W(e^{\lambda})}{R_0e^{\rho_0}}\bigr)=\sgn(w)\dim  L(\nu-\rho_0)=
\sgn(w) \prod_{\alpha\in\Delta^+_0}\frac{(\nu|\alpha)}{(\rho_0|\alpha)}=\sgn(w)\prod_{\alpha\in w^{-1}\Delta^+_0}\frac{(\lambda|\alpha)}{(\rho_0|\alpha)}.$$
One has $(-1)^{\#\{\alpha\in w^{-1}\Delta^+_0\cap (-\Delta^+_0)\}}=\sgn(w)$; 
this gives (iii) for $\lambda\in WP^{++}(\fg_0)$.
\end{proof}

\subsubsection{} 
For a subset $U\subset\Delta$ we will use the notation
$$\ssum(U):=\sum_{\beta\in U} \beta.$$

Observe that all weights of  a finite-dimensional $\fg$-module lie in $P(\fg_0)$. 

Take $\lambda\in P(\fg_0)+\rho$. Recall that $\rho_0-\rho=\rho_1=\frac{1}{2}\sum_{\alpha\in\Delta^+_1}\alpha$. One has  
$$\frac{\KW(\lambda,S)}{Re^{\rho}}
=\frac{\jJ_W\bigl(e^{\lambda+\rho_1}
\displaystyle\prod_{\beta\in\Delta^+_1\setminus S}(1+e^{-\beta})\bigr)}{R_0e^{\rho_0}}=\sum_{U\subset\Delta^+_1\setminus S}\frac{\jJ_W\bigl(e^{\lambda+\rho_1-\ssum(U)}\bigr)}{R_0e^{\rho_0}}.$$

\subsubsection{}
\begin{cor}{cor123}
For each $\lambda\in P(\fg_0)$ the term $\frac{\KW(\lambda,S)}{Re^{\rho}}$
 lies in $\Ch(\fg_0)$ and
$$\dim \bigl(\frac{\KW(\lambda+\rho,S)}{Re^{\rho}}\bigr)=
\sum_{U\subset\Delta_1^+\setminus S}\prod_{\alpha\in\Delta^+_0}\frac{(\lambda+\rho_0-\ssum(U)|\alpha)}
{(\rho_0|\alpha)}.$$
Moreover, 
$$\frac{\KW(\lambda+\rho,S)}{Re^{\rho}}=\sum n_{\lambda\mu} \ch L_{\fg_0}(\mu)$$
with the coefficients given by
$$n_{\lambda\mu}=\sum_{U\subset\Delta^+_1\setminus S} \sum_{w\in W} \sgn(w)\delta_{w(\mu+\rho_0),\lambda+\rho_0-\ssum(U)}.$$
\end{cor}

\subsubsection{Example}
If $\ch L$ is given by the Kac-Wakimoto formula 
$$Re^{\rho}\ch L=j^{-1}\KW(\lambda+\rho, S),$$
then
$$\dim L=j^{-1}\sum_{U\subset\Delta_1^+\setminus S}\prod_{\alpha\in\Delta^+_0}\frac{(\lambda+\rho_0-\sum_{\beta\in U}\beta|\alpha)}
{(\rho_0|\alpha)}.$$

\subsubsection{$L(\lambda)$ as a $\fg_0$-module}\label{Lasg0}
A Verma module $M(\lambda)$ has a filtration with the factors of the form
$\{M_{\fg_0}(\lambda-\ssum(U))\}_{U\subset\Delta^+_1}$. Notice that if 
$M_{\fg_0}(\lambda-\ssum(U))$ has a finite-dimensional quotient, then this 
quotient is $L_{\fg_0}(\lambda-\ssum(U))$. Hence 
$[L(\lambda):L_{\fg_0}(\lambda-\mu)]\not=0$ implies
$\mu=\ssum(U)$ for some $U\subset \Delta^+_1$.
The multiplicity $m_{\lambda;U}:=[L(\lambda):L_{\fg_0}(\lambda-\ssum(U))]$ can be computed using~\Cor{cor123}:
\begin{equation}\label{multik}
m_{\lambda;U}=\sum_{\mu}  (-1)^{||\lambda||-||\mu||}d_{<}^{\lambda,\mu} 
\sum_{U'\subset\Delta^+_1\setminus S_{\mu}} j_{\mu}^{-1} \sum_{w\in W} \sgn(w)\delta_{w(\lambda+\rho_0-\ssum(U)),\mu^{\dagger}+\rho_1-\ssum(U')}.\end{equation}
For the $\osp$-case this gives
$$m_{\lambda;U}=\sum_{\mu}  (-1)^{||\lambda||-||\mu||}d_{<}^{\lambda,\mu} 
\sum_{U'\subset\Delta^+_1\setminus S_{\mu}} j_{\mu}^{-1} \sum_{w\in W} \sgn(w)\delta_{w(\lambda+\rho_0-\ssum(U)),\mu+\rho_0-\ssum(U')}.$$

\subsubsection{Remark}\label{xiversion}
A variation of the above reasoning allows to find the
{\em graded multiplicities}
$$[L(\nu)_0:L_{\fg_0}(\mu)]+\xi [L(\nu)_1:L_{\fg_0}(\mu)]$$
using the Gruson-Serganova character formula. In order to do this
we  define the graded version of $\KW(\lambda,S)$ by the following procedure.

Let $\xi$ be a formal (even) variable satisfying $\xi^2=1$.
We denote by  $\Ch_{\xi}(\fg)$  the ring of $\xi$-characters of
the finite-dimensional $\fg$-modules and view
 $\Ch_{\xi}(\fg)$  as a subring of $\cR[\xi]$.
For $\nu\in\mathbb{Z}\Delta$ 
consider the map $\Xi:e^{\nu}\mapsto \xi^{p(\nu)}e^{\nu}$ and
extend this map to the rational functions $\frac{P}{Q}$, where $P,Q$
are polynomials in $e^{\nu}$ with $\nu\in\mathbb{Z}\Delta$.
This allows to define for $\lambda\in\fh^*_{int}$ the term
$\KW_{\xi}(\lambda,S)$ by the formula
$$\KW_{\xi}(\lambda,S):=e^{\lambda}\Xi\bigl(e^{-\lambda}\KW(\lambda,S)\bigr).$$
Note that $\KW_{\xi}(\lambda,S)$ and $\Xi(R^{\pm 1})$  lie in the
ring $\cR[\xi]$ and can be viewed as elements of
 $\cR_{\Sigma}[\xi]$. 
Taking $\lambda\in P(\fg_0)$ we have
$$\frac{\KW_{\xi}(\lambda+\rho,S)}{\Xi(R)e^{\rho}}=\frac{J_W\bigl(e^{\lambda+\rho_0}\displaystyle\prod_{\beta\in\Delta_1^+\setminus S} (1+\xi e^{-\beta}) \bigr)}{R_0e^{\rho_0}}\in \Ch_{\xi}(\fg_0).$$
The graded multiplicity
$[\frac{\KW(\lambda+\rho,S)}{\Xi(R)e^{\rho}}:\ch L_{\fg_0}(\mu)]$
is given by
$$\sum_{U\subset\Delta^+_1\setminus S} \xi^{\# U}\sum_{w\in W} \sgn(w)
\delta_{w(\mu+\rho_0),\lambda+\rho_0-\ssum(U)}.$$
 The Gruson-Serganova formula~(\ref{GSformula}) gives the following formula for
$\ch_{\xi} L$:
\begin{equation}\label{grGrS}
\Xi(R)e^{\rho}\ch_{\xi} L=\frac{\prod_{\alpha\in\Delta_0^+}(1-e^{-\alpha})}{\prod_{\alpha\in\Delta_1^+}(1+\xi e^{-\alpha})}
e^{\rho}\ch_{\xi} L=\sum_{L'\in \Irr} \pm  b_{L,L'} \KW(L'),\end{equation}
(where $L=L(\lambda), L'=L(\nu)$ and the sign $\pm$  is given by $(-1)^{p(\lambda^{\dagger}-\nu^{\dagger})}$);
combining the above formulas one obtains an analogue of ~(\ref{multik}) for the graded multiplicity
of $L_{\fg_0}(\mu)$ in $L(\lambda)$.

\subsection{The map $\pr$}\label{projpr}
Let $\fg$ be $\fgl(m|n)$ or $\osp(M|2n)$.
Fix an odd root $\beta_0$ of the form $\beta_0=\pm(\vareps_p-\delta_q)$.

Let $e^{\mu},\mu\in\fh^*$ be a basis of the group algebra
$\mathbb{C}[\fh^*]$. Consider a projection
$\pr: \mathbb{C}[\fh^*] \to \mathbb{C}[\fh^*]$
given by
$$\pr(e^{a\vareps_p}):=1,\ \ \pr(e^{a\delta_q}):=e^{i\pi a},\ \ \ \pr(e^{a\vareps_t}):=e^{a\vareps_{t}},\ \ \ \pr(e^{a\delta_j}):=e^{a\delta_j}$$
for any $a\in\mathbb{C}$ and the indeces $t\not=p, j\not=q$.
Note that $\pr$ is an algebra homomorphism and
$\pr(e^{\beta_0})=-1$.
We extend $\pr$ to the rational functions of the form $\frac{P}{Q}$, where
$P,Q\in \mathbb{C}[\fh^*]$ is such that $\pr (Q)\not=0$.

Since $\pr$ is an algebra homomorphism for each $\lambda\in\fh^*$ one has
\begin{equation}\label{prs0}
\pr(\KW(\lambda,\emptyset)(1+e^{-\beta_0}))=0.
\end{equation}

\subsubsection{}\label{A21}
Take a  non-zero vector $x\in\fg_{\beta_0}$.
Identify
 $\fg':=\DS_x(\fg)$ with the subalgebra of $\fg$
(recall that $\fg'=\fgl(m-1|n-1)$ for $\fg=\fgl(m|n)$, $\fg'=\osp(M-2|2n-2)$ for $\fg=\osp(M|2n)$). Recall that $\fh'=\fg'\cap\fh$ is a Cartan subalgebra of $\fg'$.

Observe that
$$\pr(e^{\lambda})=c_{\lambda} e^{\lambda|_{\fh'}}\ \text{ for }\
c_{\lambda}:=e^{-\pi i (\lambda|\delta_q)} $$
and that the restriction of
$\pi\pr\pi$ to the supercharacter ring $\jJ(\fg)$ is equal to $\ds_x$ \cite{CHR}.

\subsubsection{}\label{assxi}
Set
$$\xi:=\left\{
\begin{array}{l}
\frac{1}{2}(\sum_{i=1}^m\vareps_i-\sum_{i=1}^n\delta_i)\ \text{ for } \osp(2m+1|2n),\\
0 \text{ otherwise. }
\end{array}\right.$$
Notice that $\xi|_{\fh'}$ is equal to the vector $\xi$ defined for $\fg'$;
we denote this vector by $\xi'$.

We assume that an iso-set $S$ and $\lambda\in \fh^*$ satisfy
\begin{equation}\label{assmSla}
\beta_0\in S\subset (S_m\times S_n)\beta_0,\ \ \ \ \
(\lambda-\xi|\fh(S)^*)=0.\end{equation}
and set
$$S':=S\setminus\{\beta_0\},\ \ \ \lambda':=\lambda|_{\fh'}.$$
By above,
$\pr(e^{\lambda})=e^{-\pi i (\lambda|\delta_q)}e^{\lambda'}$;
observe that $(\lambda|\delta_q)=0$ (resp., $(\lambda|\delta_q)=-1/2$)
for  $\osp(2m|2n)$ (resp., $\osp(2m+1|2n)$).

\subsubsection{}
\begin{prop}{proppr}
Let $S,\lambda$ be as in (\ref{assmSla}).

For  $\fg=\osp(2m|2n)$ with $m>1$ and $|S|=1$ one has
$$\pr(\KW(\lambda,S)(1+e^{-\beta_0}))=\KW(\lambda',\emptyset)
+\KW((\lambda')^{\sigma},\emptyset).
$$
For other cases
$$\pr\bigl(\KW(\lambda,S)(1+e^{-\beta_0})\bigr)=ae^{-\pi i (\lambda|\delta_q)}\KW(\lambda',S'),$$
where $a=|S|$ for $\fg=\osp(2|2n),\fgl(m|n)$ and
$a=2|S|$ for $\osp(2m+1|2n)$ and for $\osp(2m|2n)$ with $m,|S|>1$.

\end{prop}
\begin{proof}
Denote by
$W'$ the Weyl group of $\fg'$ and notice that
$W'=\Stab_W \beta_0$.
Set
$$s:=|S|,\ \ \ \ c:=e^{-\pi i (\lambda|\delta_q)}.$$

One has
$$\begin{array}{rl}
\pr\bigl(\KW(\lambda,S)(1+e^{-\beta_0})\bigr)
&=\displaystyle\sum_{w\in W}\sgn(w)  y(w),\\
  \text{ where }\ \ \ \ &
y(w):=
 \pr\bigl(\frac{e^{w\lambda}(1+e^{-\beta_0})}{\displaystyle\prod_{\beta\in S}(1+e^{-w\beta})}\bigr).\end{array}$$
Observe  that $\pr(1+e^{\alpha})=0$ for $\alpha\in\Delta$ is equivalent to
 $\alpha=\pm\beta_0$.
Since $\pr$ is an algebra homomorphism, this gives
$y(w)=0$ if $\pm\beta_0\not\in wS$.
Therefore
$$\pr(\KW(\lambda,S)(1+e^{-\beta_0}))=Y_++Y_-,$$
where
$$Y_{\pm}:=\displaystyle\sum_{w\in W: \pm\beta_0\in wS}\sgn(w)  y(w) .$$

Each $\beta\in S'$  can be written as
$\beta_0=w_{\beta}\beta$ for
   $w_{\beta}:=r_{\vareps_i-\vareps_p}r_{\delta_j-\delta_q}$.
   Setting $w_{\beta_0}:=Id$ we have $\sgn(w_{\beta})=1$ and
 $$w_{\beta}\lambda=\lambda,\ \  w_{\beta}\beta_0=\beta,\
 w_{\beta}\beta=\beta_0,\ \
 w_{\beta}(\beta')=\beta'\ \text{ for } \beta'\in S\setminus\{\beta,\beta_0\}.
 $$
for each $\beta\in S$.
 The operator $\pr$ commutes with the action of $w'$ for $w'\in W'$. This gives
$$y(w'w_{\beta})=\frac{\pr(e^{w'w_{\beta}\lambda})}{\displaystyle\prod_{\beta'\in S'}
(1+e^{-w'\beta'})}=w'\bigl(\frac{\pr(e^{\lambda})}{\displaystyle\prod_{\beta'\in S'}
(1+e^{-\beta'})}\bigr)\ \text{ for any }\ w'\in W'.$$
Since $W'=\Stab_{W}\beta_0$ one has
$\{w\in W|\ \beta_0\in WS\}= W'w_{\beta}$.

By \ref{assxi}, $\pr(e^{\lambda})=ce^{\lambda'}$. Summarizing we obtain
$$
Y_+=
\displaystyle\sum_{\beta\in S}\sum_{w'\in W'}\sgn(w') y(w'w_{\beta})=cs\jJ_{W'}\bigl(\frac{e^{\lambda}}{\displaystyle\prod_{\beta'\in S'}
(1+e^{-\beta'})}\bigr),$$
that is $Y_+=cs\KW(\lambda';S')$.

For $\fg=\fgl(m|n),\osp(2|2n)$ the set $WS$ does not contain $-\beta_0$ so
$Y_-=0$;
this
completes the proof for these cases.

 For the remaining cases $\fg=\osp(M|N)$  with $M>2$ the set $WS$ contains
 $-\beta_0$. For $\fg=\osp(2m|2n)$ with $s>1$ we fix
$\beta_1:=\pm(\vareps_i-\delta_j)\in S'$;
 for $\osp(2m|2n)$ with $m>1$ and $s=1$ we set $i:=m$ if $p\not=m$ and $i:=m-1$ if
 $p=m$. We set
 $$w_-:=\left\{\begin{array}{ll}
 r_{\vareps_p}r_{\delta_q} & \text{ for }\ \ \fg=\osp(2m+1|2n)\\
 r_{\vareps_i}r_{\delta_j}r_{\vareps_p}r_{\delta_q} & \text{ for }\ \
 \fg=\osp(2m|2n),\ s>1\\
 r_{\vareps_i}r_{\vareps_p}r_{\delta_q}& \text{ for }\ \
 \fg=\osp(2m|2n),\ s=1.
 \end{array}
 \right.$$
Notice that $w_-\in W$ and
$w_-\beta_0=-\beta_0$. Therefore
$$\ \ \
\{w\in W|\ -\beta_0\in wS\}=\coprod_{\beta\in S}
W'w_-w_{\beta}$$
and thus
$$
Y_-=\sum_{\beta\in S}\sum_{w'\in W'} \sgn(w'w_-) y(w'w_-w_{\beta})
$$
For $w'\in W'$ we have
$$y(w'w_-w_{\beta})=\pr\bigl(\frac{e^{w'w_-\lambda}(1+e^{-\beta_0})}
{\displaystyle\prod_{\beta\in S}
(1+e^{-w'w_-\beta})}\bigr)=-w'\pr\bigl(\frac{e^{w_-\lambda}}{\displaystyle\prod_{\beta\in S'}
(1+e^{-w_-\beta})}\bigr).$$
Therefore
$$Y_-=-s\sgn(w_-)\jJ_{W'}\bigl(\pr\bigl(\frac{e^{w_-\lambda}}{\displaystyle\prod_{\beta\in S'}
(1+e^{-w_-\beta})}\bigr)\bigr).
$$

For $\osp(2m+1|2n)$ one has
$w_-S'=S'$ and
$w_-\lambda=\lambda+\beta_0$, that is $\pr(e^{w_-\lambda})=-ce^{\lambda'}$.
Therefore $Y_-=cs\KW(\lambda',S')$ as required.

For $\osp(2m|2n)$  with $s>1$ one has
$$w_-\lambda=\lambda,\ \ \ w_-S'=(S'\cup\{-\beta_1\})\setminus\{\beta_1\}.$$
 Using \Lem{lemprlambdaS}  we get
$$Y_-=-s\KW(\lambda',S'\cup\{-\beta_1\})\setminus\{\beta_1\}=s|\KW(\lambda',S').$$

For the remaining case
 $\osp(2m|2n)$ with $m>1$ and $S=\{\beta_0\}$ we have
 $$\sgn(w_-)=-1,\ \ \ \pr(e^{w_-\lambda})=e^{r_{\vareps_j}\lambda'},$$
 that is $Y_-=\KW(r_{\vareps_i}\lambda',\emptyset)$.
Since $r_{\vareps_i}\lambda'=(\lambda')^{\sigma}$
this completes the proof.
\end{proof}

\end{document}